\documentclass{article} 
\usepackage{amsmath,amssymb,amsthm} 
\usepackage[OT2,T1]{fontenc}
\usepackage[utf8]{inputenc}
\usepackage{enumitem}
\usepackage{graphicx}
\usepackage{geometry}
\usepackage{hyperref}
\usepackage{tikz}
\usetikzlibrary{cd}
\usepackage{tikz-cd}
\usepackage{xcolor}
\usepackage{ stmaryrd }
\usepackage{mathrsfs}
\usepackage{titling}
\usepackage{breqn}
\usepackage{mathtools}

\usepackage[alphabetic,lite]{amsrefs}

\usetikzlibrary{arrows,positioning,decorations.pathmorphing,
	decorations.markings
}
\tikzset{inner sep=0pt,
	root/.style={circle,draw,minimum size=7pt,thick},
	fatroot/.style={circle,draw,minimum size=10pt,thick},
	short root/.style={circle,fill,minimum size=7pt},
	doublearrow/.style={postaction={decorate},
		decoration={markings,mark=at position .7
			with {\arrow{angle 60}}},double distance=3pt,thick}
}

\DeclareUnicodeCharacter{00A0}{ }

\newtheorem{proposition}{Proposition}[section]
\newtheorem{definition}[proposition]{Definition}
\newtheorem{remark}[proposition]{Remark}

\newtheorem{theorem}[proposition]{Theorem}
\newtheorem{lemma}[proposition]{Lemma}
\newtheorem{corollary}[proposition]{Corollary}

\numberwithin{equation}{subsection}

\newcommand{\G}{\mathbb{G}}
\DeclareMathOperator{\GL}{GL}
\DeclareMathOperator{\PGL}{PGL}

\DeclareMathOperator{\Sp}{Sp}
\DeclareMathOperator{\PSp}{PSp}

\DeclareMathOperator{\SO}{SO}

\DeclareMathOperator{\liesl}{\mathfrak{sl}}

\DeclareMathOperator{\lieh}{\mathfrak{h}}
\DeclareMathOperator{\liet}{\mathfrak{t}}
\DeclareMathOperator{\Ad}{Ad}
\DeclareMathOperator{\der}{der}

\DeclareMathOperator{\Id}{Id}

\DeclareMathOperator{\Hom}{Hom}
\DeclareMathOperator{\Aut}{Aut}
\DeclareMathOperator{\Out}{Out}

\DeclareMathOperator{\Gal}{Gal}

\DeclareMathOperator{\Sym}{Sym}
\DeclareMathOperator{\vol}{vol}

\DeclareMathOperator{\rank}{rank}
\DeclareMathOperator{\image}{image}

\DeclareMathOperator{\Mat}{Mat}
\DeclareMathOperator{\Isom}{Isom}

\DeclareMathOperator{\Spec}{Spec}

\DeclareMathOperator{\Frac}{Frac}

\newcommand{\sh}[1]{\mathscr{#1}}

\newcommand{\A}{\mathbb{A}}
\renewcommand{\P}{\mathbb{P}}

\renewcommand{\O}{\mathcal{O}}
\newcommand{\GIT}{\mathbin{/\mkern-6mu/}}
\newcommand{\HH}{\mathrm{H}}

\newcommand{\Real}{\mathbb{R}}
\newcommand{\CC}{\mathbb{C}}
\newcommand{\Q}{\mathbb{Q}}
\newcommand{\Z}{\mathbb{Z}}
\newcommand{\F}{\mathbb{F}}

\DeclareMathOperator{\Sel}{Sel}
\DeclareMathOperator{\rk}{rk}

\DeclareSymbolFont{cyrletters}{OT2}{wncyr}{m}{n}
\DeclareMathSymbol{\Sha}{\mathalpha}{cyrletters}{"58}

\newcommand{\extp}{\@ifnextchar^\@extp{\@extp^{\,}}}
\def\@extp^#1{\mathop{\bigwedge\nolimits^{\!#1}}}
\newcommand{\overbar}[1]{\mkern 1.5mu\overline{\mkern-1.5mu#1\mkern-1.5mu}\mkern 1.5mu}
\DeclareUnicodeCharacter{00A0}{ }
\newcommand{\define}[1]{{\fontfamily{cmss}\selectfont{#1}}}

\DeclareMathOperator{\GrLie}{GrLie}
\DeclareMathOperator{\GrLieE}{GrLieE}
\newcommand{\intbigG}{\underline{G}}
\newcommand{\intbigH}{\underline{H}}
\newcommand{\intbigV}{\underline{V}}
\newcommand{\intbigB}{\underline{B}}
\newcommand{\intbigT}{\underline{T}}
\newcommand{\intbigt}{\underline{\mathfrak{t}}}
\newcommand{\intlieh}{\underline{\mathfrak{h}}}
\newcommand{\bigH}{H}
\newcommand{\bigh}{\lieh}
\newcommand{\bigG}{G}
\newcommand{\bigB}{B}
\newcommand{\bigP}{P}
\newcommand{\bigT}{T}

\renewcommand{\bigg}{\mathfrak{g}}
\newcommand{\bigV}{V}
\newcommand{\bigtheta}{\theta}
\newcommand{\bigsigma}{\sigma}
\newcommand{\bigkappa}{\kappa}
\newcommand{\bigLambda}{\Lambda}
\newcommand{\bigpi}{\pi}
\newcommand{\bigA}{A}

\newcommand{\liec}{\mathfrak{c}}

\newcommand{\gencurve}{X}

\newcommand{\intbigcurve}{\mathcal{C}}
\newcommand{\bigaffcurve}{C^{\circ}}
\newcommand{\bigprojcurve}{C}
\newcommand{\Jac}{J}
\newcommand{\intJac}{\mathcal{J}}
\newcommand{\NeronJac}{\mathscr{J}}
\newcommand{\genJac}{J_{X}}

\newcommand{\CJac}{\bar{\mathcal{J}}}

\newcommand{\bigHeis}{\mathscr{H}}
\newcommand{\bigExt}{\mathscr{U}}
\newcommand{\genHeis}{\mathcal{H}}
\newcommand{\genExt}{\mathcal{U}}

\newcommand{\height}{\mathrm{ht}}

\newcommand{\Gmtorsor}[1]{\mathbb{V}(#1)^{\times}}

\DeclareMathOperator{\disc}{disc}
\DeclareMathOperator{\ord}{ord}
\DeclareMathOperator{\rs}{rs}
\DeclareMathOperator{\reg}{reg}

\usepackage{microtype}

\title{The average size of the 2-Selmer group of a family of non-hyperelliptic curves of genus $3$}
\author{Jef Laga}

\begin{document}
\maketitle
\begin{abstract}
	We show that the average size of the $2$-Selmer group of the family of Jacobians of non-hyperelliptic genus-$3$ curves with a marked rational hyperflex point, when ordered by a natural height, is bounded above by $3$.
	We achieve this by interpreting $2$-Selmer elements as integral orbits of a representation associated with a stable $\Z/2\Z$-grading on the Lie algebra of type $E_6$ and using Bhargava's orbit-counting techniques. 
	We use this result to show that the marked point is the only rational point for a positive proportion of curves in this family. 
	The main novelties are the construction of integral representatives using certain properties of the compactified Jacobian of the simple curve singularity of type $E_6$, and a representation-theoretic interpretation of a Mumford theta group naturally associated to our family of curves.
	
\end{abstract}
\tableofcontents	
\section{Introduction}

The statistical behaviour of Selmer groups of Jacobians of families of algebraic curves is a topic that has seen many advances in recent years. 
In \cite{BS-2selmerellcurves}, Bhargava and Shankar determined the average size of the $2$-Selmer group of the family of elliptic curves in short Weierstrass form when ordered by height, showing that it is equal to $3$. 
Bhargava and Gross \cite{Bhargava-Gross-hyperellcurves} generalized their results to the family of hyperelliptic curves of genus $g$ with a marked rational Weierstrass point. Poonen and Stoll \cite{PoonenStoll-Mosthyperellipticnorational} used the latter to prove that for each $g\geq 3$, a positive proportion of such hyperelliptic curves have exactly one rational point, and this proportion tends to $1$ as $g$ tends to infinity. See \cite{Shankar-2selmerhypermarkedpoints}, \cite{ShankarWang-hypermarkednonweierstrass} for similar results for families of hyperelliptic curves with other types of marked points and \cite{BS-3Selmer}, \cite{BS-4Selmer},\cite{BS-5Selmer}, \cite{Thorne-Romano-E8} for analogous results for $n$-Selmer groups of (hyper)elliptic curves with $n\geq 3$.

\subsection{Statement of results}

This paper is a contribution to the arithmetic statistics of non-hyperelliptic genus-$3$ curves. 
Such curves are canonically embedded in $\P^2$ as smooth plane quartics.
Let $X$ be a (smooth, projective, geometrically connected) genus-3 curve over $\Q$ that is not hyperelliptic and $P\in X(\Q)$ a marked rational point. 
We say $P$ is a \define{hyperflex} if $4P$ is a canonical divisor or equivalently, the tangent line at $P$ in the canonical embedding meets $X$ only at $P$. 
Any pair $(X,P)$ with $P$ a hyperflex is isomorphic to a pair $(\bigprojcurve_b,P_{\infty})$ where $\bigprojcurve_b$ is the projective completion of the plane curve
\begin{equation}\label{equation: e6 family beginning paper}
y^3 = x^4+(p_2x^2+p_5x+p_8)y+p_6x^2+p_9x+p_{12}, 
\end{equation}
where $b = (p_2,\dots,p_{12}) \in \Q^6$, and where $P_{\infty}$ is the unique point at infinity. Pairs $(\bigprojcurve_b,P_{\infty})$ given by Equation (\ref{equation: e6 family beginning paper}) are isomorphic if and only if the coefficients are related by a substitution $(p_i) \mapsto (\lambda^ip_i)$ for some $\lambda\in \Q^{\times}$, which explains the subscripts of the coefficients. 
Call such an equation \define{minimal} if $p_i \in \Z$ and the following two conditions are satisfied:
\begin{itemize}
	\item There exists no prime $q$ such that $q^i$ divides $p_i$ for all $i \in \{2,5,6,8,9,12\}$.
	\item Either we have $p_5 >0$, or we have $p_5=0$ and $p_9 \geq 0$. 
\end{itemize}
Then any pair $(X,P)$ arises from a unique minimal equation. 
Write $\sh{E} \subset \Z^6$ for the subset of integers $(p_2,p_5,p_6,p_8,p_9,p_{12})$ such that Equation (\ref{equation: e6 family beginning paper}) defines a smooth curve $\bigprojcurve_b$, and write $\sh{E}_{\min} \subset \sh{E}$ for the subset for which the equation is minimal. 
For $b \in \sh{E}$, write $\Jac_b$ for the Jacobian variety of $\bigprojcurve_b$, a principally polarized abelian threefold over $\Q$.
For $b\in \sh{E}$ we define the \define{height} of $b$ by the formula 
$$\height(b) \coloneqq \max_i |p_i(b)|^{72/i}.$$
Note that for any $a >0$, the set $\{b \in \sh{E} \mid \height(b) <a \}$ is finite. 
Our first main theorem concerns the average size of the $2$-Selmer group of $\Jac_b$. In what follows, we let $\mathcal{F}$ be either $\sh{E}_{\min}$ or a subset of $\sh{E}$ defined by finitely many congruence conditions (see \S\ref{section: proof of main theorems}).
\begin{theorem}[Theorem \ref{theorem: main theorem}]\label{theorem: first main theorem intro}
	When ordered by height, the average size of the $2$-Selmer group $\Sel_2\Jac_b$ for $b\in \mathcal{F}$ is bounded above by $3$. More precisely, we have 
	\begin{equation*}
	\limsup_{a\rightarrow \infty} \frac{ \sum_{b\in \mathcal{F},\; \height(b)<a }\# \Sel_2\Jac_b    }{\#  \{b \in \mathcal{F}\mid \height(b) < a\}}	\leq 3.
	\end{equation*}
\end{theorem}
We expect that the limit exists and equals $3$, see the discussion of Step $(3)$ in \S\ref{subsection: intro methods}.
Thorne \cite{Thorne-E6paper} has proved that the average size of the $2$-Selmer set of $\bigprojcurve_b$ (a pointed subset of $\Sel_2 \Jac_b$) for $b\in \sh{E}$, when ordered by height, is finite. From this he deduces that a positive proportion of members of the family of affine curves $\bigaffcurve_b$ for $b\in \sh{E}$, obtained from $\bigprojcurve_b$ by removing the point at infinity, have integral points everywhere locally but no integral points globally. 
Theorem \ref{theorem: first main theorem intro} provides an explicit estimate on the size of the full $2$-Selmer group, not just the $2$-Selmer set. We therefore obtain more Diophantine consequences. 
For example, Bhargava and Shankar observed that bounding the $2$-Selmer group gives an upper bound on the average rank of elliptic curves. 
In our case we can bound the average of the Mordell--Weil rank $\rk(\Jac_b)$ of $\Jac_b$, the rank of the finitely generated abelian group $\Jac_b(\Q)$.
Using the inequalities $2 \rk(\Jac_b) \leq 2^{\rk(\Jac_b)} \leq \# \Sel_2 \Jac_b $, we obtain:
\begin{corollary}
	The average rank $\rk(\Jac_b)$ for $b\in \mathcal{F}$ is bounded above by $3/2$. 
\end{corollary}
Another corollary is a bound on the number of rational points of $\bigprojcurve_b$ for $b\in \sh{E}$, in the spirit of \cite[Corollary 1.4]{Bhargava-Gross-hyperellcurves}. 
Write $\delta$ for the proportion of curves in $\mathcal{F}$ satisfying Chabauty's condition, namely $\rk(\Jac_b) \leq \text{genus}(\bigprojcurve_b)-1 = 2$. 
Then Theorem \ref{theorem: first main theorem intro} implies that 
$$
\delta+ (1-\delta)\cdot 2^3 \leq 3,
$$
so $\delta \geq 5/7$. 
A computation shows that at least $85.7\%$ of curves in our family have good reduction at $7$, and for such curves we have $\#\bigprojcurve_b(\F_7)\leq 22$.
Stoll's refined bound \cite[Corollary 6.7]{Stoll-Twists} on the Chabauty method implies: 
\begin{corollary}
	A majority (in fact at least $61\%$) of curves $\bigprojcurve_b$ for $b\in \sh{E}$ have at most $26$ rational points.
\end{corollary}

Our second main result shows that the Chabauty method at the prime $2$ implies that a positive proportion of curves in our family have only one rational point, using the methods of Poonen and Stoll \cite{PoonenStoll-Mosthyperellipticnorational}.

\begin{theorem}[Theorem \ref{theorem: poonen stoll analogue}]\label{theorem: intro poonen stoll}
	A positive proportion of curves $\bigprojcurve_b$ for $ b\in\sh{E}$ have only one rational point. More precisely, the quantity
	\begin{equation*}
	\liminf_{a\rightarrow \infty} \frac{ \# \{b \in \sh{E} \mid \height(b)<a, \, \bigprojcurve_b(\Q)=\{P_{\infty}\} \}   }{\#  \{b \in \sh{E} \mid \height(b) < a \}}
	\end{equation*}
is strictly positive.	
\end{theorem}

\subsection{Methods}\label{subsection: intro methods}

Bhargava and his collaborators have developed a general strategy for obtaining statistical results on $2$-Selmer groups of families of curves (and many other arithmetic objects). 
Roughly speaking, the proofs of these theorems have the following structure. 
For a family of curves $\mathcal{F}$ of interest, one hopes to find a representation $V$ of a reductive group $G$ over $\Q$ so that the $2$-Selmer groups of (the Jacobians of) the curves in $\mathcal{F}$ can be embedded in the set of $G(\Q)$-orbits of $V(\Q)$. Moreover, after fixing integral structures on $G$ and $V$, orbits corresponding to $2$-Selmer elements should have integral representatives. 
If the representation $V$ is coregular (meaning that $V\GIT G \coloneqq \Spec \Q[V]^G$ is isomorphic to affine space) and satisfies some additional properties, then Bhargava's orbit-counting techniques allow us to count integral orbits in $V$ and sieve out those orbits not corresponding to $2$-Selmer elements. 
In \cite{Bhargava-Gross-hyperellcurves}, Bhargava and Gross studied the $2$-Selmer group of odd hyperelliptic curves of genus $g$ in this way, using the representation of $\SO_{2g+1}$ on the space of traceless, self-adjoint $(2g+1)\times (2g+1)$-matrices. 
Our proof of Theorem \ref{theorem: first main theorem intro} has the same structure, although most of the proofs of the individual parts are very different in nature. 
We now explain in steps how we $(1)$ find the representation $(G,V)$, $(2)$ prove that $2$-Selmer elements admit integral representatives and $(3)$ count integral orbits.

For Step $(1)$, we follow the approach taken by Thorne \cite{Thorne-thesis} using a combination of Vinberg theory and the Grothendieck--Brieskorn correspondence. Given a (split, adjoint) simple algebraic group $H$ over $\Q$ with Lie algebra $\lieh$, there exists an involution $\theta\colon H\rightarrow H$, uniquely defined up to conjugation by an element of $H(\Q)$,
with the property that the group $G \coloneqq \left(H^{\theta}\right)^{\circ}$ is split and that the $G$-representation $V \coloneqq \lieh^{d\theta = -1}$ has good invariant-theoretic properties.
We call such $\theta$ a \define{stable involution}.
(In \S\ref{section: setup} we make an explicit choice for such an involution in the $E_6$ case, after having fixed a pinning of $H$.)
Write $B \coloneqq V\GIT G = \Spec \Q[V]^G$ and $\pi: V\rightarrow V\GIT G$ for the canonical projection. Vinberg theory shows that $B$ is isomorphic to $\A^{\rank H}$, so $V$ is coregular. 
The theory of the Kostant section shows that $\pi$ has a section $\sigma : B \rightarrow V$ and for each $b\in B(\Q)$ we call $\sigma(b) \in \bigV(\Q)$  
the `distinguished orbit' or `reducible orbit' (playing a role analogous to that of reducible binary quartic forms in \cite{BS-2selmerellcurves}).  
Taking a transverse slice to the $G$-action on $V$ at a subregular nilpotent element of $V$ defines a closed subscheme $\bigaffcurve \rightarrow V$. 
The restriction of $\pi$ to $\bigaffcurve$ defines a family of curves $\bigaffcurve  \rightarrow B$.   

If $H$ is simply laced (so of type $A_n, D_n$ or $E_n$), Thorne \cite[Theorem 3.8]{Thorne-thesis} shows that the fibre $\bigaffcurve_0$ above $0 \in B$ is a simple curve singularity of the same type as $H$ and $\bigaffcurve \rightarrow B$ is a semi-universal deformation of its central fibre.
Moreover in each case there exists a natural compactification $\bigprojcurve \rightarrow B$ of the family $\bigaffcurve \rightarrow B$, and if the fibre $\bigprojcurve_b$ above a point $b\in B(\Q)$ is smooth then he shows that there is a natural Galois equivariant isomorphism $\Jac_b[2] \simeq Z_G(\sigma(b))$ where $\Jac_b$ is the Jacobian of $\bigprojcurve_b$. This last isomorphism, combined with the well-known interpretation of the $\bigG(\Q)$-orbits on $\bigV_b(\Q)$ in terms of the Galois cohomology of $Z_G(\sigma(b))$ (Lemma \ref{lemma: AIT}), gives the link between the $2$-Selmer group of $\Jac_b$ and the orbits of the representation $V$. 
If $H$ is of type $A_{2g}$, the singularity is of the form $(y^2 = x^{2g+1})$ and the family $\bigprojcurve \rightarrow B$ is isomorphic to the family of odd hyperelliptic curves of genus $g$ considered by Bhargava and Gross.
If $H$ is of type $E_6$, the singularity is of the form  $(y^3 = x^4)$ and if we write $B = \Spec \Q[p_2,p_5,p_6,p_8,p_9,p_{12}]$ then the family $\bigprojcurve \rightarrow B$ is isomorphic to the family given by Equation (\ref{equation: e6 family beginning paper}). 

For Step $(2)$, we follow the same strategy as \cite{Thorne-Romano-E8} where the authors prove a similar result for a different representation in their study of the $3$-Selmer groups of genus-$2$ curves. It turns out that proving that a $G(\Q_p)$-orbit has an integral representative amounts to proving that a certain object, consisting of a reductive group over $\Q_p$ with extra data, extends to an object over $\Z_p$ (see Proposition \ref{proposition: G-orbits in terms of groupoids}).
We achieve this by deforming to the case of square-free discriminant and using a general result on extending reductive group schemes over open dense subschemes of regular arithmetic surfaces (Lemma \ref{lemma: extend objects complement codimension 2}).  
In \cite{Thorne-Romano-E8} the authors use the Mumford representation to perform this step explicitly. 
Here we complete the deformation step by exploiting properties of the compactified Jacobian of the $E_6$ curve singularity $(y^3 = x^4)$ in the sense of Altman and Kleiman \cite{AltmanKleiman-CompactifyingThePicardScheme}, and by using Bertini theorems over $\Q_p$ and $\F_p$. 
A crucial ingredient is the fact that the total space of the relative compactified Jacobian of the semi-universal deformation of the singularity is nonsingular. 
The techniques applied here work verbatim for any of the families described in Step $(1)$ where the centre of the simply connected group of the corresponding Dynkin diagram has odd order, namely $A_{2n}, E_6$ and $E_8$. (This condition ensures that $\bigprojcurve \rightarrow \bigB$ has geometrically integral fibres, which leads to a good theory of the compactified Picard scheme.)
This provides a way of proving the existence of integral representatives in many of the previously considered cases in the literature. (Our method only works for sufficiently large primes $p$ but this does not cause any problems in the counting step.)
It should be straightforward to make this strategy work for all the families of Step $(1)$.

For Step $(3)$, we follow the ideas of Bhargava closely, about which we will make two remarks. 
First of all, because we cannot prove a uniformity estimate like \cite[Theorem 2.13]{BS-2selmerellcurves}, we only obtain an upper bound in our estimates on integral orbits. We expect that similar uniformity estimates hold in our case, which would allow us to use the so-called square-free sieve to show that the average size of the $2$-Selmer group of $\Jac_b$ is in fact equal to $3$.
Secondly, the substantial work of `cutting off the cusp' has already been done in \cite{Thorne-E6paper} so counting integral orbits is a formal matter for us given the robustness of Bhargava's counting techniques.

Why are we able to estimate the size of the full $2$-Selmer group and not just the $2$-Selmer set as in \cite{Thorne-E6paper}? Apart from a way of constructing integral representatives mentioned above, this is based on the following novelty. 
Thanks to \cite{thorne-planequarticsAIT}, we have a way of embedding the full $2$-Selmer group of curves in the orbits of our representations.
But the construction does not make it clear that the $2$-Selmer group has in its image the reducible orbit. 
Controlling this is crucial for our counting techniques since we only count irreducible orbits in Step $(3)$. 

We prove that in fact the identity element of the $2$-Selmer group is mapped to the reducible orbit, using the following strategy. Fix $b\in \bigB(\Q)$ such that $\bigprojcurve_b$ is smooth with Jacobian $\Jac_b$, and let $G^{sc}$ denote the simply connected cover of $G$. It turns out that proving this statement for $\bigprojcurve_b$ amounts to proving that the simply connected centralizer $\mathcal{U}\coloneqq Z_{G^{sc}}(\sigma(b))$ of $\sigma(b)$ is isomorphic to a subgroup $\mathcal{H}$ of the Mumford theta group related to a certain canonical line bundle on $J_b$. 
The strategy to prove that $\mathcal{U}$ and $\mathcal{H}$ are isomorphic is inspired by the following simple observation: let $C$ be a smooth projective geometrically connected genus-$g$ curve over $\Q$ with Jacobian $J_C$ and let $Z$ be a finite group scheme over $\Q$ that satisfies $Z_{\overbar{\Q}} \simeq \left(\Z/2\Z\right)^{2g}$. If there exists a $Z$-torsor $\widetilde{C} \rightarrow C$ such that $\widetilde{C}$ is geometrically connected, then $Z \simeq \Jac[2]$ as finite group schemes over $\Q$. 
In our case roughly the same principles apply. The finite \'etale $\Q$-groups $\mathcal{U}$ and $\mathcal{H}$ are central extensions of $\Jac_b[2]$ by $\{\pm 1\}$ which lie in the same isomorphism class over $\overbar{\Q}$. 
By constructing a $\mathcal{U}$-torsor and $\mathcal{H}$-torsor arising from the properties of our representation and the geometry of Mumford theta groups respectively, we realize $\mathcal{U}$ and $\mathcal{H}$ as quotients of the \'etale fundamental group of the open curve $\bigaffcurve_{b,\overbar{\Q}}$ with respect to some rational basepoint (we actually have to take a tangential basepoint following \cite{Deligne-droiteprojective}). 
We then show that this \'etale fundamental group essentially has only one quotient with the required group-theoretic properties, so $\mathcal{U}$ and $\mathcal{H}$ both inherit the same Galois action from this fundamental group. This proves that $\mathcal{U}$ and $\mathcal{H}$ are isomorphic over $\Q$, proving the required statement. 
This argument proves a conjecture of Thorne \cite[Conjecture 4.16]{Thorne-thesis} in the $E_6$ case.
We note that although our construction of orbits (in particular Theorem \ref{theorem: inject 2-descent into orbits}) is very much based on ideas developed in \cite{thorne-planequarticsAIT}, we have phrased the proofs in a way independent of that paper because of some simplifications in the argument and the more general form that we prove.

To prove Theorem \ref{theorem: intro poonen stoll} we merely have to adapt certain arguments of \cite{PoonenStoll-Mosthyperellipticnorational} in a straightforward way. 
One of the crucial ingredients in their argument for hyperelliptic curves is an equidistribution result for $2$-Selmer elements under the mod $2$ reduction of the logarithm map. 
Since we only obtain an upper bound in Theorem \ref{theorem: first main theorem intro}, we only obtain an `at most equidistribution' result (Theorem \ref{theorem: equidistribution selmer}) but this is enough for our purposes. 

The results of this paper will be used in forthcoming work \cite{Laga-F4paper} (which was in fact the main motivation for this paper) where we consider the subfamily of curves defined by setting $p_5=p_9=0$ in Equation (\ref{equation: e6 family beginning paper}). The Jacobian of a curve in this subfamily splits as a product of an elliptic curve and a Prym variety, which in that case is a $(1,2)$-polarized abelian surface. 
Studying the Lie algebra embedding $F_4\subset E_6$ leads to estimates of Selmer groups of these Prym surfaces, which provides evidence for the heuristics of Poonen and Rains \cite{PoonenRains-maximalisotropic} in the case of non-principally polarized abelian varieties.

\subsection{Organization}

We now describe the organization of the paper. 
In \S\ref{section: setup} we define the group $\bigG$ and the representation $\bigV$ of $\bigG$ whose orbits we study as a central topic of this paper. 
We review some properties of this representation and make the connection to the family of curves with Equation (\ref{equation: e6 family beginning paper}). 
In \S\ref{section: orbit parametrization}, we construct orbits associated with $2$-Selmer elements. 
In \S\ref{section: integral representatives} we prove that orbits coming from $2$-Selmer elements admit integral representatives away from small primes. 
In \S\ref{section: counting}, we employ Bhargava's orbit-counting techniques to give the estimates we need in order to prove Theorem \ref{theorem: first main theorem intro}.
In \S\ref{section: proof of main theorems} we combine all of these ingredients and prove Theorem \ref{theorem: first main theorem intro}.
Finally in \S\ref{section: applications to rational points} we prove Theorem \ref{theorem: intro poonen stoll}.

\subsection{Acknowledgements}

I thank my supervisor Jack Thorne for suggesting the problem, providing many useful suggestions and his constant encouragement. I also want to thank Marius Leonhardt, Davide Lombardo and Beth Romano for their comments on an earlier draft of this paper.
Finally I wish to thank Bjorn Poonen for sharing with me the proof of Lemma \ref{lemma: sqfree disc implies regular node}. 

This project has received funding from the European Research Council (ERC) under the European Union’s Horizon 2020 research and innovation programme (grant agreement No. 714405).

\subsection{Notation and conventions}

For a field $k$ we write $k^s$ for a fixed separable closure and $\Gamma_k = \Gal(k^s/k)$ for its absolute Galois group. 

We will often use the equivalence of categories between finite \'etale group schemes over $k$ (called finite $k$-groups) and finite groups with a continuous $\Gamma_k$-action. As such we may identify a finite $k$-group with its set of $k^s$-points. 

We define a \define{lattice} to be a finitely generated free $\Z$-module $\Lambda$ together with a symmetric and positive-definite bilinear form $(\cdot,\cdot)\colon \Lambda\times  \Lambda \rightarrow \Z$. We write $\Lambda^{\vee}\coloneqq \{\lambda\in \Lambda \otimes \Q \mid (\lambda, \Lambda) \subset \Z\}$ for the \define{dual lattice} of $\Lambda$, which is naturally identified with $\Hom(\Lambda,\Z)$. 
We say $\Lambda$ is a \define{root lattice} if $(\lambda,\lambda)$ is an even integer for all $\lambda\in \Lambda$ and the set 
$$\{ \alpha \in \Lambda \mid (\alpha,\alpha) =2 \}$$
generates $\Lambda$. If $\Phi \subset \Real^n$ is a simply laced root system then $\Lambda= \Z\Phi$ is a root lattice. In that case we define the type of $\Lambda$ to be the Dynkin type of $\Phi$. 

If $S$ is a scheme, an \define{\'etale sheaf of root lattices} $\Lambda$ over $S$ is defined as a locally constant \'etale sheaf of finite free $\Z$-modules together with a bilinear pairing $\Lambda\times \Lambda \rightarrow \Z$ such that for every geometric point $\bar{s}$ of $S$ the stalk $\Lambda_{\bar{s}}$ is a root lattice. In that case $\Aut(\Lambda)$ is a finite \'etale $S$-group.

If $X$ is a scheme over $S$ and $T\rightarrow S$ a morphism we write $X_T$ for the base change of $X$ to $T$. If $T = \Spec A$ is an affine scheme we also write $X_A$ for $X_T$. 

If $G$ is a smooth group scheme over $S$ then we write $\HH^1(S,G)$ for the set of isomorphism classes of \'etale sheaf torsors under $G$ over $S$, which is a pointed set coming from non-abelian \v{C}ech cohomology. If $S = \Spec R$ we write $\HH^1(R,G)$ for the same object.


If $G\rightarrow S$ is a group scheme acting on $X\rightarrow S$ and $x \in X(T)$ is a $T$-valued point, we write $Z_G(x) \rightarrow T$ for the centralizer of $x$ in $G$. It is defined by the following pullback square:
\begin{center}
	\begin{tikzcd}
		 Z_G(x) \arrow[d] \arrow[r ] & T \arrow[d] \\
		G\times_S X \arrow[r]   & X\times_S X                        
	\end{tikzcd}
	\end{center}
Here $G\times_S X \rightarrow X \times_S X$ denotes the action map and $T \rightarrow X\times_S X$ denotes the composite of $x$ with the diagonal $X\rightarrow X \times_S X$. 

If $x$ is an element of a Lie algebra $\lieh$ then we write $\mathfrak{z}_{\lieh}(x)$ for the centralizer of $x$ in $\lieh$, a subalgebra of $\lieh$.


A \define{$\Z/2\Z$-grading} on a Lie algebra $\lieh$ over a field $k$ is a direct sum decomposition 
$$\lieh = \bigoplus_{i\in \Z/2\Z} \lieh(i) $$
of linear subspaces of $\lieh$ such that $[h(i),h(j)] \subset \lieh(i+j)$ for all $i,j \in \Z/2\Z$. This is equivalent to giving a $\mu_2$-action on $\lieh$ by considering the $(\pm 1)$-part of such an action.
If $2$ is invertible in $k$ then giving a $\Z/2\Z$-grading is equivalent to giving an involution of $\lieh$. 

We call a triple $(X,H,Y)$ an \define{$\liesl_2$-triple} of a Lie algebra $\lieh$ if $X,Y,H$ are nonzero elements of $\lieh$ satisfying the following relations:
\begin{equation*}
[H,X] = 2X , \quad [H,Y] = -2Y ,\quad [X,Y] = H .
\end{equation*}

If $V$ is a vector space over a field $k$ we write $k[V]$ for the graded algebra $\Sym(V^{\vee})$. Then $V$ is naturally identified with the $k$-points of the scheme $\Spec k[V]$, and we call this latter scheme $V$ as well.
If $G$ is a group scheme over $k$ we write $V \GIT G\coloneqq \Spec k[V]^G$ for the \define{GIT quotient} of $V$ by $G$.

\begin{table}
\centering
\begin{tabular}{|c | c | c |}
	\hline
	   Symbol & Definition & Reference in paper \\
	\hline       
		$\bigH$ & Split adjoint group of type $E_6$ & \S\ref{subsection: a stable grading} \\
		$\bigT$ & Split maximal torus of $H$ & \S\ref{subsection: a stable grading} \\
		$\bigtheta$ & Stable involution of $H$ & \S\ref{subsection: a stable grading} \\
		$\bigG$ & Fixed points of $\theta$ on $H$ & \S\ref{subsection: a stable grading} \\
		$\bigV$ & $(-1)$-part of action of $\bigtheta$ on $\bigh$ & \S\ref{subsection: a stable grading}\\
		$\bigB$ & GIT quotient $\bigV\GIT \bigG$ & \S\ref{subsection: a stable grading} \\
		$\Delta \in \Q[B]$ & Discriminant polynomial & \S\ref{subsection: a stable grading} \\
		$\pi\colon \bigV \rightarrow \bigB$ & Invariant map & \S\ref{subsection: a stable grading} \\
		$\sigma\colon \bigB \rightarrow \bigV$ & Kostant section & \S\ref{subsection: distinguished orbit} \\
		$\bigaffcurve \rightarrow \bigB$ & Family of affine curves & \S\ref{subsection: a family of curves}  \\
		$\bigprojcurve \rightarrow \bigB$ & Family of projective curves & \S\ref{subsection: a family of curves} \\
		$\Jac \rightarrow \bigB^{\rs}$ & Jacobian variety of $\bigprojcurve^{\rs} \rightarrow \bigB^{\rs}$ & \S\ref{subsection: a family of curves} \\
		$p_2,\dots ,p_{12}$ & Invariant polynomials of $\bigG$-action on $\bigV$ &  \S\ref{subsection: a family of curves} \\
		$\bigA \rightarrow \bigB^{\rs}$ & Centralizer of $\sigma|_{\bigB^{\rs}}$ in $H$ & \S\ref{subsection: a family of curves}\\
		$\bigLambda \rightarrow \bigB^{\rs}$ & Character group scheme of the torus $A\rightarrow \bigB^{\rs}$& \S\ref{subsection: a family of curves}\\
		$\bigHeis \rightarrow \bigB^{\rs}$ & Subgroup of Mumford Theta group & \S\ref{subsection: Mumford theta groups} \\
		$\bigExt \rightarrow \bigB^{\rs}$ & Centralizer of $\sigma|_{\bigB^{\rs}}$ in $\bigG^{sc}$ & \S\ref{subsection: comparting two central extensions} \\
		$N$ & Sufficiently large integer &\S\ref{subsection: integral structures} \\
		$S$ & $\Z[1/N]$ &\S\ref{subsection: integral structures} \\
		$\intbigH, \intbigG, \intbigV$ & Extensions of above objects over $\Z$ &\S\ref{subsection: integral structures} \\
		$\intbigcurve \rightarrow \intbigB$ & Extension of $\bigprojcurve \rightarrow \bigB$ over $\Z$ &\S\ref{subsection: integral structures} \\
		$\intJac \rightarrow \intbigB_S^{\rs}$ & Jacobian of $\intbigcurve^{\rs}_S \rightarrow \intbigB^{\rs}_S$ &\S\ref{subsection: integral structures} \\
		$\CJac \rightarrow \intbigB_S$ & Compactification of $\intJac \rightarrow \intbigB_S^{\rs}$ & \S\ref{subsection: compactifications} \\

	\hline
\end{tabular}
\caption{Notation used throughout the paper}
\label{table 1}
\end{table}

\section{Setup} \label{section: setup}

\subsection{Definition of the representation}\label{subsection: a stable grading}

Let $\bigH$ be a split adjoint semisimple group of type $E_6$ over $\Q$. We suppose that $\bigH$ comes with a pinning $(\bigT,\bigP,\{X_{\alpha}\})$. So $\bigT \subset \bigH$ is a split maximal torus (which determines a root system $\Phi(\bigH,\bigT) \subset X^*(T)$), $\bigP\subset \bigH$ is a Borel subgroup containing $\bigT$ (which determines a root basis $\Delta_{\bigH} \subset \Phi(\bigH,\bigT)$) and $X_{\alpha}$ is a generator for each root space $\bigh_{\alpha}$ for $\alpha \in \Delta_{\bigH}$. The group $\bigH$ is of dimension $78$.

Write $\check{\rho} \in X_*(\bigT)$ for the sum of the fundamental coweights with respect to $\Delta_{\bigH}$, characterised by the property that $(\alpha\circ \check{\rho})(t) = t$ for all $\alpha \in \Delta_{\bigH}$. 
Write $\zeta\colon H\rightarrow H$ for the unique nontrivial automorphism preserving the pinning: it is an involution inducing the order-$2$ symmetry of the Dynkin diagram of $E_6$. Let $\bigtheta \coloneqq \zeta \circ \Ad(\check{\rho}(-1)) = \Ad(\check{\rho}(-1)) \circ\zeta$. Then $\theta$ defines an involution of $\lieh$ and thus by considering $(\pm1)$-eigenspaces it determines a $\Z/2\Z$-grading 
$$\bigh = \bigh(0) \oplus \bigh(1).$$
Let $\bigG \coloneqq \bigH^{\bigtheta}$ be the centralizer of $\bigtheta$ in $\bigH$ and let $\bigV\coloneqq \bigh(1)$: the space $V$ defines a representation of $\bigG$ by restricting the adjoint representation. If we write $\bigg$ for the Lie algebra of $\bigG$ then $\bigV$ is a Lie algebra representation of $\bigg = \bigh(0)$. The pair $(\bigG,\bigV)$ is the central object of study of this paper.  

The results of \cite{Reeder-torsion} applied to the Kac diagram of $\bigtheta$ \cite[\S7.1; Table 3]{GrossLevyReederYu-GradingsPosRank} show that $G$ is isomorphic to $\PSp_8$ and $V$ is the unique irreducible $42$-dimensional subrepresentation of $\wedge^4(8)$, where $(8)$ denotes the defining representation of $\Sp_8$. 

The following proposition summarizes some properties of the representation $\bigV$. In particular, it shows that regular semisimple orbits over algebraically closed fields are well understood. For a field $k/\Q$ and $v\in \bigV(k)$, we say $v$ is \define{regular, nilpotent, semisimple} respectively if it is so when considered as an element of $\bigh(k)$.

\begin{proposition} \label{prop : graded chevalley}
	Let $k/\Q$ be a field. The following properties are satisfied:
	\begin{enumerate}
		\item $\bigV_k$ satisfies the Chevalley restriction theorem: if $\mathfrak{a} \subset \bigV_k$ is a Cartan subalgebra, then the map $N_{\bigG}(\mathfrak{a}) \rightarrow W_{\mathfrak{a}} \coloneqq N_{\bigH}(\mathfrak{a})/Z_{\bigH}(\mathfrak{a})$ is surjective, and the inclusions $\mathfrak{a} \subset \bigV_k \subset \lieh_k$ induce isomorphisms
		$$\mathfrak{a}\GIT W_{\mathfrak{a}} \simeq \bigV_k\GIT \bigG \simeq \lieh_k \GIT \bigH  .$$
		In particular, the quotient is isomorphic to affine space. 
		\item Suppose that $k$ is separably closed and let $x,y\in \bigV(k)$ be regular semisimple elements. Then $x$ is $\bigG(k)$-conjugate to $y$ if and only if $x,y$ have the same image in $\bigV\GIT \bigG$. 
		\item Let $\Delta \in \Q[\bigV]^{\bigG}$ be the restriction of the Lie algebra discriminant of $\bigh$ to the subspace $\bigV$. Then for all $x\in \bigV(k)$, $x$ is regular semisimple if and only if $\Delta(x) \neq 0$, if and only if the $\bigG$-orbit of $x$ is closed in $\bigV_k$ and the stabilizer $Z_{\bigG}(x)$ is finite. 
	\end{enumerate}
\end{proposition}
\begin{proof}
	These are classical results in the invariant theory of graded Lie algebras due to Vinberg and Kostant--Rallis; we refer to \cite[\S2]{Thorne-thesis} for precise references. Note that the discriminant of a Lie algebra is by definition the image of the product of all the roots in a fixed Cartan subalgebra under the Chevalley isomorphism. 
\end{proof}

We note that Cartan subalgebras of $\lieh$ contained in $\bigV$ do exist: we will construct a family of tori $\bigA \rightarrow \bigB^{\rs}$ in \S\ref{subsection: a family of curves} whose Lie algebras provide such examples.
We write $\bigB \coloneqq \bigV\GIT \bigG = \Spec \Q[\bigV]^{\bigG}$ and $\bigpi\colon \bigV \rightarrow \bigB$ for the natural quotient map.
We have a $\G_m$-action on $V$ given by $\lambda \cdot v = \lambda v$ and there is a unique $\G_m$-action on $B$ such that $\pi$ is $\G_m$-equivariant.

\subsection{The distinguished orbit}\label{subsection: distinguished orbit}

We describe a section of the GIT quotient $\pi\colon \bigV \rightarrow \bigB$ whose construction is originally due to Kostant. 
Let $E \coloneqq \sum_{\alpha \in \Delta_{\bigH}} X_{\alpha} \in \bigh$. Then $E\in  \bigh(1)$ is regular and nilpotent. By \cite[Lemma 2.14 and Lemma 2.17]{Thorne-thesis} there exists a unique normal $\liesl_2$-triple $(E,X,F)$ containing $E$.  By definition, this means that $(E,X,F)$ is an $\liesl_2$-triple with the additional property that $X\in \bigh(0)$ and $F \in \bigh(1)$. We define the affine linear subspace $\bigkappa \coloneqq \left(E +\mathfrak{z}_{\bigh}(F) \right) \cap \bigV \subset \bigV$.
\begin{proposition}\label{proposition: Kostant section E6}
    \begin{enumerate}
        \item The composite map $\bigkappa \hookrightarrow \bigV\rightarrow \bigB$ is an isomorphism. 
        \item $\bigkappa$ is contained in the open subscheme of regular elements of $\bigV$.
        \item The morphism $\bigG \times \kappa \rightarrow \bigV, (g,v) \mapsto g\cdot v$ is \'etale.
    \end{enumerate}
\end{proposition}
\begin{proof}
    Parts 1 and 2 are \cite[Lemma 3.5]{Thorne-thesis}; the last part is \cite[Proposition 3.4]{Thorne-thesis}, together with the fact that $G\times \kappa$ and $V$ have the same dimension.
\end{proof}

Write $\bigsigma\colon \bigB \rightarrow \bigV$ for the inverse of $\bigpi|_{\bigkappa}$. We call $\bigsigma$ the \define{Kostant section} for the group $\bigH$. It determines a distinguished orbit over $\Q$ for every $b\in \bigB(\Q)$ in the representation $\bigV$, playing an analogous role to reducible binary quartic forms as studied in \cite{BS-2selmerellcurves}. It will be used to organize the set of rational orbits with fixed invariants.

\subsection{A family of curves}\label{subsection: a family of curves}

We introduce a family of curves and relate it to stabilizers of regular semisimple elements in the representation $V$.
We say an element $b\in B$ is \define{regular semisimple} if it has nonzero discriminant and we write $\bigB^{\rs}\subset \bigB$ for the open subscheme of regular semisimple elements of $\bigB$, the complement of the discriminant locus in $\bigB$. For a $\bigB$-scheme $U$ we write $U^{\rs}$ for the restriction to the regular semisimple locus. 
For example if $k/\Q$ is a field and $v\in \bigV(k)$, then $v\in\bigV^{\rs}(k)$ if and only if $v$ is regular semisimple in the sense of \S\ref{subsection: a stable grading} by Part 3 of Proposition \ref{prop : graded chevalley}.

The following straightforward lemma shows that for $v\in V^{\rs}(k)$ the isomorphism class of $Z_G(v)$ only depends on $\pi(v)$. 
\begin{lemma}\label{lemma: centralizers with same invariants isomorphic}
Let $S$ be a scheme and $v,v' \colon S \rightarrow V^{\rs}$ be morphisms such that $\pi(v) = \pi(v')$. Then $Z_G(v) \simeq Z_G(v')$ as group schemes over $S$. 	
\end{lemma}
\begin{proof}
    This follows from the fact that $v,v'$ are \'etale locally $\bigG$-conjugate and that $Z_G(v)$ is abelian, see \cite[Part 2 of Proposition 4.1]{Thorne-thesis}.
\end{proof}

We define $\bigA$ as the centralizer $Z_{\bigH}(\bigsigma|_{\bigB^{\rs}})$, a maximal torus of $\bigH_{\bigB^{\rs}} = \bigH \times \bigB^{\rs}$. (Recall that by our conventions, the centralizer of a $\bigB^{\rs}$-point of $\bigV$ is a group scheme over $\bigB^{\rs}$.)
This defines for every field $k/\Q$ and $b\in \bigB^{\rs}(k)$ a maximal torus $A_b$ in $H_k$.
We write $\bigLambda$ for the character group $X^*(\bigA)$ of $\bigA$, an \'etale sheaf of $E_6$ root lattices over $\bigB^{\rs}$.

\begin{lemma}\label{lemma: centralizer kostant same as mod 2 root lattice}
	The involution $\bigtheta$ restricts to the inversion map on $\bigA$, so $Z_{\bigG}(\sigma|_{\bigB^{\rs}}) = A[2]$.
	Moreover we have a natural isomorphism $ Z_G(\sigma|_{\bigB^{\rs}}) \simeq \bigLambda/2\bigLambda$ of group schemes over $\bigB^{\rs}$. 
\end{lemma}
\begin{proof}
	The first claim follows from the fact that $\bigtheta$ is a stable involution and can be deduced from \cite[Lemma 2.21]{Thorne-thesis}.
	To prove that $A[2] \simeq \Lambda/2\Lambda$, we note that the Cartier dual of $A[2]$ equals $\Lambda/2\Lambda$. The pairing on $\Lambda$ defines an injective map $\Lambda\rightarrow \Lambda^{\vee}$ whose image has index $3$, so its mod $2$ reduction is an isomorphism. This proves that $Z_{\bigG}(\sigma|_{\bigB^{\rs}})= A[2]$ is naturally isomorphic to the Cartier dual of $\Lambda^{\vee}/2\Lambda^{\vee}$, which is $\Lambda/2\Lambda$. 
\end{proof}

We note that since $\Lambda$ is an \'etale sheaf of $E_6$ root lattices over $\bigB^{\rs}$, $\Lambda/2\Lambda$ is a finite \'etale group scheme over $\bigB^{\rs}$ and the pairing $(,)$ on $\Lambda$ induces a pairing $\Lambda/2\Lambda \times \Lambda/2\Lambda \rightarrow \{\pm 1\} , (\lambda,\mu) \mapsto (-1)^{(\lambda,\mu)}$, where we view $\{\pm 1 \}$ as a constant group scheme over $\bigB^{\rs}$. 
We define the morphism of $\bigB^{\rs}$-schemes $q_{\Lambda}\colon \Lambda/2\Lambda \rightarrow \{ \pm 1 \}$ by sending an $S$-point $\lambda$ to $(-1)^{(\lambda,\lambda)/2} $.
Then $q_{\Lambda}$ is a quadratic form on $\Lambda/2\Lambda$, in the sense that $q_{\Lambda}(\lambda+\mu) = (-1)^{(\lambda,\mu)} q_{\Lambda}(\lambda) q_{\Lambda}(\mu) $ for all $S$-points $\lambda,\mu$.

The following important proposition gives a connection between the representation $V$ and a family of algebraic curves parametrized by $\bigB$.

\begin{proposition}\label{proposition: bridge jacobians root lattices}
	We can choose polynomials $p_2, p_5, p_6, p_8, p_9, p_{12} \in \Q[\bigV]^{\bigG}$ with the following properties:
	\begin{enumerate}
		\item Each polynomial $p_i$ is homogeneous of degree $i$ and $\Q[\bigV]^{\bigG} \simeq \Q[p_2, p_5, p_6, p_8, p_9, p_{12}]$. Consequently, there is an isomorphism $\bigB\simeq \A^6_{\Q}$.
		\item Let $\bigaffcurve \rightarrow \bigB$ be the family of affine curves given by the equation
		\begin{equation}\label{equation : E6 family middle of paper}
		y^3 = x^4+y(p_2x^2+p_5x+p_8)+p_6x^2+p_9x+p_{12}.
		\end{equation}
		Let $\bigprojcurve\rightarrow \bigB$ be the completion of $\bigaffcurve \rightarrow \bigB$ inside $\P^2_{\bigB}$. If $k/\Q$ is a field and $b\in \bigB(k)$, then $\bigprojcurve_b$ is smooth if and only if $b\in \bigB^{\rs}(k)$. 
		
		\item Let $\Jac \rightarrow \bigB^{\rs}$ be the relative Jacobian of its smooth part \cite[\S9.3; Theorem 1]{BLR-NeronModels}. 
		Then there is an isomorphism $\bigLambda/2\bigLambda \simeq \Jac[2] $ of finite \'etale group schemes over $\bigB^{\rs}$ that sends the pairing on $\bigLambda/2\bigLambda$ to the Weil pairing $\Jac[2] \times \Jac[2] \rightarrow \{ \pm 1\}$. 
		\item There exists an isomorphism $Z_{\bigG}(\sigma|_{\bigB^{\rs}}) \simeq \Jac[2]$ of finite \'etale group schemes over $\bigB^{\rs}$. 
		
	\end{enumerate}
\end{proposition}

\begin{proof}
	Part 1 follows from the isomorphism $\Q[V]^G\simeq \Q[\lieh]^{\bigH}$ of Proposition \ref{prop : graded chevalley} and the well-known description of the invariant polynomials of the adjoint action of $\bigH$ on $\lieh$; see for example \cite[Theorem 3.5]{Panyushev-Invarianttheorythetagroups}.
	Part 2 follows from \cite[Theorem 3.8, case $E_6$]{Thorne-thesis} and \cite[Corollary 3.16]{Thorne-thesis}, together with the fact that $C_b$ is always smooth at the point at infinity.
	Part 3 follows from \cite[Corollary 4.12]{Thorne-thesis}.
	Finally, Part 4 follows from combining Part 3 with Lemma \ref{lemma: centralizer kostant same as mod 2 root lattice}.
\end{proof}

For the remaining part of this paper we fix a choice of polynomials $p_2, p_5, p_6, p_8, p_9, p_{12} \in \Q[\bigV]^{\bigG}$ satisfying the conclusions of Proposition \ref{proposition: bridge jacobians root lattices}. Recall that we have defined a $\G_m$-action on $\bigB$ which satisfies $\lambda \cdot p_i = \lambda^ip_i$. 
The assignment $\lambda \cdot (x,y) := (\lambda^3 x,\lambda^4 y)$ defines a $\G_m$-action on $\bigprojcurve$ such that the morphism $\bigprojcurve\rightarrow \bigB$ is $\G_m$-equivariant.

\subsection{Further properties of $\Jac[2]$}\label{subsection: further properties of J[2]}

We give some additional properties of the group scheme $\Jac[2]
 \rightarrow \bigB^{\rs}$, which by Proposition \ref{proposition: bridge jacobians root lattices} we may identify with $\bigLambda/2\bigLambda \rightarrow \bigB^{\rs}$. 
Before we state them, we recall some definitions and set up notation. 
Recall from \S\ref{subsection: a stable grading} that $T$ is a split maximal torus of $H$. Let $\liet$ be its Lie algebra and $\Lambda_T$ its character group. 
Write $W\coloneqq N_G(T)/T$ for the Weyl group of $T$, a constant group scheme over $\Q$. 
Part 1 of Proposition \ref{prop : graded chevalley} implies that the natural map $\bigB=\bigV\GIT \bigG \rightarrow \bigh\GIT \bigH$ is an isomorphism. Write $\liet \rightarrow \liet \GIT W \simeq \bigh \GIT \bigH \simeq \bigB$ for the composite of the inverse of this isomorphism with the Chevalley isomorphism $\liet \GIT W \simeq \bigh \GIT \bigH$ and the natural projection map $\liet \rightarrow \liet \GIT W$. 
Restricting to regular semisimple elements defines a finite \'etale cover $f\colon\liet^{\rs} \rightarrow \bigB^{\rs}$ with Galois group $W$.

\begin{proposition}\label{proposition: monodromy of J[2]}
    We have the following: 
    \begin{enumerate}
        \item The finite \'etale group scheme $\bigLambda/2\bigLambda\rightarrow \bigB^{\rs}$ becomes trivial after the base change $f\colon\liet^{\rs} \rightarrow \bigB^{\rs}$, where it is isomorphic to the constant group scheme $\Lambda_T/2\Lambda_T$. The monodromy action is given by the natural action of $W$ on $\Lambda_T/2\Lambda_T$.
        \item The only section of $\bigLambda/2\bigLambda\rightarrow \bigB^{\rs}$ is the zero section. 
        \item If $q\colon \bigLambda/2\bigLambda \rightarrow \{\pm 1\}$ is a $\bigB^{\rs}$-morphism such that $q(\lambda+\mu) = (-1)^{(\lambda,\mu)}q(\lambda)q(\mu)$ for all $S$-points $\lambda,\mu$ of $\bigLambda/2\bigLambda$, then $q = q_{\Lambda}$. 
        \item The only automorphism of the $\bigB^{\rs}$-group scheme $\Lambda/2\Lambda$ fixing the pairing $\Lambda/2\Lambda\times \Lambda/2\Lambda \rightarrow \{\pm 1\}, (\lambda,\mu) \mapsto (-1)^{(\lambda,\mu)}$ is the identity.
    \end{enumerate}

\end{proposition}
\begin{proof}
    The first claim follows from the fact that the torus $A \rightarrow \bigB^{\rs}$ is isomorphic to the constant torus $T\times \liet^{\rs} \rightarrow \liet^{\rs}$ after pulling back along $f$, with monodromy given by the action of $W$ on $T$.  
    
    Indeed, a straightforward adaptation of Lemma \ref{lemma: centralizers with same invariants isomorphic} to the case of the adjoint action of $\bigH$ on $\lieh$ shows that if $x,x'\colon S \rightarrow \lieh^{\rs}$ are $S$-points which agree after composing with $\lieh^{\rs} \rightarrow \lieh\GIT \bigH \simeq \bigB$, then $Z_{\bigH}(x) \simeq Z_{\bigH}(x')$ as group schemes over $S$. (Here $\lieh^{\rs}\subset \lieh$ denotes the subset of regular semisimple elements.)
    In particular, we can apply this to the $\liet^{\rs}$-points $i\colon \liet^{\rs} \rightarrow \lieh^{\rs}$ (where $i$ denotes the inclusion) and $\sigma \circ f$ (where $\sigma$ denotes the Kostant section). 
    Comparing their centralizers, we obtain an isomorphism $T\times{\liet^{\rs}} \simeq A_{\liet^{\rs}} $.
    Since this isomorphism is induced by \'etale locally conjugating $i$ and $\sigma\circ f$ by elements of $\bigH$, the monodromy action is indeed given by the natural action of $W$ on $\bigT$.

    
    For the second claim, it suffices to prove that the only $W$-invariant element of $\Lambda_T/2\Lambda_T$ is the identity. This is an easy exercise in the combinatorics of the root lattice of type $E_6$. 
    
    For the third claim, note that the $\bigB^{\rs}$-scheme of quadratic refinements of the non-degenerate pairing $\Lambda/2\Lambda\times \Lambda/2\Lambda\rightarrow \{\pm 1\}$ is a torsor for the group $\bigLambda/2\bigLambda \rightarrow \bigB^{\rs}$ by \cite[\S1]{GrossHarris-theta}. Since the latter group scheme has only one $\bigB^{\rs}$-point by the second claim, the lemma follows.
    
	Finally we treat the fourth claim. By the previous claim such an isomorphism must preserve the quadratic form $q_{\Lambda}$.
	So it suffices to prove that every automorphism of $\Lambda_T/2\Lambda_T$ preserving the quadratic form $q(\lambda)= (-1)^{(\lambda,\lambda)/2}$ and commuting with every element of $W$ is the identity. 
    But since the natural map $W \rightarrow \Aut(\Lambda_T/2\Lambda_T,q)$ is an isomorphism \cite[Remark 4.3.4]{Lurie-minisculereps} and the centre of the Weyl group of $E_6$ is trivial, the proposition follows.

    \end{proof}

For later purposes, it is useful to know that the isomorphism $\bigLambda/2\bigLambda \simeq \Jac[2]$ intertwines certain quadratic forms on both sides, as we now explain. 

On the one hand, in \S\ref{subsection: a family of curves} we have defined a quadratic form $q_{\Lambda} \colon \Lambda/2\Lambda \rightarrow \{ \pm 1\}$ satisfying $q_{\Lambda}(\lambda+\mu) = (-1)^{(\lambda,\mu)} q_{\Lambda}(\lambda) q_{\Lambda}(\mu) $ for all $\lambda,\mu$.

On the other hand, we can use the theory of theta characteristics to define a quadratic form on $\Jac[2]$, as follows. 
(We refer the reader to \cite{GrossHarris-theta} for basics on theta characteristics.)
For every field $k/\Q$ and every $b\in \bigB^{\rs}(k)$ the curve $\bigprojcurve_{b}$ has a marked point $P_{\infty}$ which is a hyperflex in the canonical embedding. This implies that $4P_{\infty}$ is a canonical divisor, so $\kappa_b = 2P_{\infty}$ is a theta characteristic.
The following well-known result of Mumford \cite{Mumford-thetacharacteristicsalgebraiccurve} shows that to this data we can associate a quadratic form. To state it in a general set-up, let $\gencurve/k$ be a smooth projective curve with Jacobian variety $\genJac$. 
We define for a divisor $D$ on $\gencurve$ the quantity $h^0(D) \coloneqq \dim_k \HH^0(\gencurve,\O_{\gencurve}(D))$.

\begin{lemma}\label{lemma: mumford construction quadratic form to theta char}
	Let $\kappa$ be a divisor on $\gencurve$ such that $2\kappa$ is canonical. 
	Then the map $q_{\kappa}\colon \genJac[2] \rightarrow \{\pm 1\}$ defined by 
\begin{displaymath}
q_{\kappa}(\omega) \coloneqq (-1)^{h^0(\kappa + \omega) +h^0(\kappa) }
\end{displaymath}
	is a quadratic refinement of the Weil pairing: for all $\omega, \eta \in \genJac[2]$, we have $q_{\kappa}(\omega+\eta) = e_2(\omega,\eta)q_{\kappa}(\omega)q_{\kappa}(\eta)$, where $e_2 \colon \genJac[2]\times \genJac[2] \rightarrow \{\pm 1\} $ denotes the Weil pairing.
\end{lemma}

We apply the above construction to the fibres of $\bigprojcurve^{\rs} \rightarrow \bigB^{\rs}$ and the theta characteristic $\kappa = 2P_{\infty}$.
In fact by \cite[Theorem 1]{Mumford-thetacharacteristicsalgebraiccurve} this procedure can be globalized: we obtain a quadratic form $q_{\kappa} \colon \Jac[2] \rightarrow \{\pm1\}$ refining the Weil pairing $e_2 \colon \Jac[2]\times \Jac[2] \rightarrow \{\pm 1\}$.

\begin{proposition}\label{prop: quadratic forms identified}
	Under the isomorphism $\bigLambda/2\bigLambda \simeq \Jac[2]$ of Proposition \ref{proposition: bridge jacobians root lattices}, the quadratic forms $q_{\Lambda}$ and $q_{\kappa}$ are identified. 
\end{proposition}
\begin{proof}
	Write $q\colon \bigLambda/2\bigLambda \rightarrow \{\pm 1\}$ for the composite of $q_{\kappa}$ with the above isomorphism. It suffices to prove that $q_{\Lambda}$ and $q$ are equal. Since both $q_{\Lambda}$ and $q$ are quadratic refinements of the same pairing on $\bigLambda/2\bigLambda$ by Proposition \ref{proposition: bridge jacobians root lattices}, this follows from Part 3 of Proposition \ref{proposition: monodromy of J[2]}.
\end{proof}

The following lemma relates the bitangents of a curve in our family with the $2$-torsion of the Jacobian and will be useful in \S\ref{section: applications to rational points}. Recall that $\Gamma_k$ denotes the absolute Galois group of a field $k$. 

\begin{lemma}\label{lemma: bitangents and 2-torsion}
	Let $k/\Q$ be a field and $b\in \bigB^{\rs}(k)$. Let $\mathcal{B}$ be the set of bitangents of $\bigprojcurve_b$ over $k^s$ different from the line at infinity in Equation (\ref{equation : E6 family middle of paper}), equipped with its natural $\Gamma_k$-action. 
	If $\Gamma_k$ acts transitively on $\mathcal{B}$, then $J_b[2](k)=0$.
\end{lemma}
\begin{proof}
	It is well-known \cite[\S4]{GrossHarris-theta} that bitangents of $\bigprojcurve_b$ correspond to odd theta characteristics of $\bigprojcurve_b$; this correspondence identifies the line at infinity with $2P_{\infty}$. (Recall that $P_{\infty}$ denotes the unique point at infinity of $\bigprojcurve_b$.)
	The assignment $\kappa \mapsto \kappa - 2P_{\infty}$ defines a $\Gamma_k$-equivariant bijection from the set of theta characteristics to the set of $2$-torsion points on $\Jac_b$. Moreover under the identification $\Jac_b[2] \simeq \Lambda_b/2\Lambda_b$ from Proposition \ref{proposition: bridge jacobians root lattices} which identifies the quadratic forms $q_{\kappa_b}$ and $q_{\Lambda_b}$ (Proposition \ref{prop: quadratic forms identified}), the set of odd theta characteristics is mapped bijectively to the zero set of the quadratic form $q_{\Lambda_b}$ on $\Lambda_b/2\Lambda_b$. The proof now follows from Lemma \ref{lemma: root lattice transitive zero set no invariants} below.
	
\end{proof}

\begin{lemma}\label{lemma: root lattice transitive zero set no invariants}
	Let $\Lambda$ be a root lattice of type $E_6$ with quadratic form $q\colon \Lambda/2\Lambda \rightarrow \{\pm 1\},\, \lambda \mapsto (-1)^{(\lambda,\lambda)/2}$. Let $G$ be a subgroup of $\Aut\left(\Lambda/2\Lambda, q \right)$ such that $G$ acts transitively on the set $\{v\in \Lambda/2\Lambda \mid v\neq 0 , \, q(v) =1 \}$.
	Then $\left(\Lambda/2\Lambda \right)^G=\{0\}$.
\end{lemma}
\begin{proof}
	Suppose that $v \in \Lambda/2\Lambda$ is a nonzero element fixed by every element of $G$. The assumptions on $G$ imply that $q(v)=-1$.
	For $i \in \{0,1\}$ define
	$$
	S_i \coloneqq \{w \in \Lambda/2\Lambda \mid (v,w)=i \}.
	$$
	Then $\Lambda/2\Lambda = S_0\sqcup S_1$ and each $S_i$ is stable under $G$. We claim that both $S_0$ and $S_1$ contain nonzero elements which take the value $1$ at $q$. 
	This would prove the lemma since it contradicts the transitivity of $G$ on such elements.
	To prove the claim, note that the group $\Aut\left(\Lambda/2\Lambda,q\right)$ acts transitively on the set of non-zero elements of $\Lambda/2\Lambda$ which take the value $-1$ at $q$ since every such element is the image of a root in $\Lambda$. So it suffices to prove the claim for a single $v$, in which case it can easily be checked explicitly.

\end{proof}

\subsection{The discriminant polynomial}\label{subsection: discriminant polynomial}

We compare the discriminant $\Delta\in \Q[\bigB]$ which is defined using Lie theory with the discriminant of a plane quartic curve. We keep the notation of the beginning of \S\ref{subsection: further properties of J[2]}.

Recall that $\Delta$ is defined as the image of $\prod_{\alpha} d\alpha\in \Q[\liet]^W$ under the chain of isomorphisms $\Q[\liet]^W \rightarrow \Q[\lieh]^{\bigH} \rightarrow \Q[\bigV]^{\bigG} = \Q[\bigB]$, where $\alpha \in \Phi(H,T)$ runs over the set of roots of $H$. 
Since $\Phi(H,T)$ has $72$ elements, $\Delta$ is homogenous of degree $72$.

\begin{lemma}\label{lemma: discriminant geometrically irreducible}
	For every field $k/\Q$, $\Delta$ is irreducible in $k[\bigB]$.
\end{lemma}
\begin{proof}
	It suffices to prove that we cannot partition $\Phi(H,T)$ into two nonempty $W$-invariant subsets, which is true since $W$ acts transitively on $\Phi(H,T)$.
\end{proof}

Now let $R$ be any ring and $F\in R[x,y,z]$ be a homogenous polynomial of degree $4$. In \cite[Definition 2.2]{Saito-Discriminanthypersurfacevendim}, the (divided) discriminant $\disc(F)\in R$ is defined. 
It is an integral polynomial in the coefficients of $F$ and $\disc(F) \in R^{\times}$ if and only if the plane quartic $(F=0) \subset \P^2_{R}$ is smooth over $R$.  
It satisfies the transformation properties $\disc(\lambda F) = \lambda^{27} F$ and $F((x,y,z)\cdot A) = (\det A)^{36} F(x,y,z)$ for every $\lambda\in R$ and $A \in \Mat_3(R)$ (Equations (2.2.3), (2.2.4) in loc. cit.).

We define $\Delta_0\in\Q[\bigB]$ as the discriminant of the (homogenized) polynomial appearing in Equation (\ref{equation : E6 family middle of paper}):
\begin{equation}\label{equation: definition discriminant Delta0}
	\Delta_0 \coloneqq \disc(y^3z-x^4-(p_2x^2z+p_5xz^2+p_8z^3)y-(p_6x^2z^2+p_9xz^3+p_{12}z^4)).
\end{equation}

\begin{proposition}\label{proposition: discriminant Delta and Delta0 agree}
	The polynomials $\Delta$ and $\Delta_0$ agree up to an element of $\Q^{\times}$. 
\end{proposition}
\begin{proof}
Since $\Delta$ and $\Delta_0$ have the same vanishing locus (Part 2 of Proposition \ref{proposition: bridge jacobians root lattices}) and $\Delta$ is irreducible (Lemma \ref{lemma: discriminant geometrically irreducible}), it suffices to prove that $\Delta_0$ is weighted homogenous of degree $72$ in the variables $p_2,\cdots,p_{12}$. 
Write $F_B\in \Q[B][x,y,z]$ for the polynomial appearing in the right hand side of Equation (\ref{equation: definition discriminant Delta0}). 
Using the transformation properties of $\disc$ we obtain
\begin{align*}
	\Delta_0(\lambda \cdot b )= \disc(\lambda^{12}F_B(\lambda^{-3}x,\lambda^{-4}y,z)) = \lambda^{12.27-7.36}\disc(F_B) = \lambda^{72}\Delta_0(b),
\end{align*}
as desired.
\end{proof}

\section{Orbit parametrization}\label{section: orbit parametrization}

The purpose of this section is to prove that for each $b\in \bigB^{\rs}(\Q)$, we can construct a natural injection $\Sel_2\Jac_b \hookrightarrow \bigG(\Q)\backslash \bigV_b(\Q)$, see Corollary \ref{corollary: Sel2 embeds}. 
In \cite{thorne-planequarticsAIT}, such an embedding was already constructed, but it is crucial to know that the distinguished orbit $\bigG(\Q)\cdot \sigma(b)$ lies in the image of this embedding and to have a more general version for the purposes of constructing integral representatives (Theorem \ref{theorem: inject 2-descent into orbits}).
The technical input is an isomorphism between two central extensions (Proposition \ref{propostion: 2 central extensions coincide}), established in \S\ref{subsection: comparting two central extensions}. 
The reader is advised to read \S\ref{subsection: Mumford theta groups}, take Corollary \ref{corollary: commutative diagram corollary central extensions} on faith and jump straight to \S\ref{subsection: twisting and embedding the selmer group}.

\subsection{Mumford theta groups}\label{subsection: Mumford theta groups}

In this subsection we construct a finite subgroup $\mathcal{H}$ of a certain Mumford theta group when a curve with a rational theta characteristic is given and realize this group as the deck transformations of a covering of schemes. 
A general reference is \cite[Chapters 6, 11]{BirkenhakeLange-CAV}.

Let $k/\Q$ be a field and $\gencurve/k$ a smooth projective geometrically integral curve of genus $g\geq 2$. 
Write $\genJac$ for its Jacobian variety and $\genJac^{g-1}$ for the $\genJac$-torsor of line bundles of degree $g-1$ on $\gencurve$. 
The variety $\genJac^{g-1}$ has a distinguished divisor $W_{g-1}$ given by the image of the Abel--Jacobi map $\gencurve^{g-1}  \rightarrow \genJac^{g-1}$, called the \define{theta divisor}. 
For an element $a\in \genJac(k)$ (respectively $a\in \genJac^{g-1}(k)$), we write $t_a$ for the translation map $t_a\colon \genJac\rightarrow \genJac$ (respectively $t_a\colon \genJac \rightarrow \genJac^{g-1}$).
We say a line bundle $\kappa \in \genJac^{g-1}(k^s)$ (or any divisor representing it) is a \define{theta characteristic} if $\kappa^{\otimes 2}$ is isomorphic to the canonical bundle.

Suppose that $\kappa \in \genJac^{g-1}(k)$ is a $k$-rational theta characteristic.
In this case $\sh{M} = \O_{\genJac}(t_{\kappa}^{*}W_{g-1})$ is a symmetric line bundle. We define the \define{Mumford theta group} $G(\sh{M}^2)$ of $\sh{M}^2$ to be the set
$$\left\{ (\omega,\phi) \mid \omega \in \genJac[2](k^s),\phi \colon \sh{M}^2 \xrightarrow{\sim} t_{\omega}^*\sh{M}^2 \right\}$$
with multiplication given by 
$(\omega,\phi)\cdot (\tau,\psi) = (\omega+\tau, t^*_{\omega}\psi \circ \phi).$
This group admits a natural $\Gamma_k$-action and fits into a central extension 
$$1\rightarrow \G_{m,k} \rightarrow G(\sh{M}^2) \rightarrow \genJac[2] \rightarrow 1.$$
The next lemma follows from the definition of the Weil pairing. 
\begin{lemma}
	Let $\omega,\tau\in \genJac[2](k^s)$, and let $\tilde{\omega},\tilde{\tau}$ be lifts of these elements to $G(\sh{M}^2)(k^s)$. 
	Then $\tilde{\omega}\tilde{\tau}\tilde{\omega}^{-1}\tilde{\tau}^{-1} = e_2(\omega,\tau)$, where $e_2\colon \genJac[2] \times \genJac[2] \rightarrow \{\pm 1\}$ denotes the Weil pairing on $\genJac[2]$. 
\end{lemma}

Since $\sh{M}$ is symmetric, there exists a unique isomorphism $f\colon\sh{M}\xrightarrow{\sim}[-1]^*\sh{M}$ that is the identity at the fibre $\sh{M}_0$ above $0 \in \genJac$. 
For every $\omega\in \genJac[2]$, we thus obtain an isomorphism $f_{\omega}\colon \sh{M}_{\omega} \xrightarrow{\sim} \sh{M}_{-\omega} = \sh{M}_{\omega}$, hence a scalar $q_{\sh{M}}(\omega) \in \left(k^s\right)^{\times}$. 
Since $[-1]^*f \circ f = \Id_{\sh{M}}$, we see that $q_{\sh{M}}(\omega) = \pm 1$. 
The next proposition shows that $q_{\sh{M}} $ is a quadratic refinement of the Weil pairing $e_2$. 

\begin{lemma}\label{lemma: quadratic form mumford theta group is in fact a quadratic form}
	The map $q_{\sh{M}}$ agrees with the quadratic form $q_{\kappa}$ from Lemma \ref{lemma: mumford construction quadratic form to theta char}: for every $\omega \in \genJac[2](k^s)$ we have 
	$$q_{\sh{M}}(\omega) = (-1)^{h^0(\omega+\kappa)+h^0(\kappa)},$$
	where $h^0(D) = \dim_k \HH^0(X,\O_{X}(D))$. 
	Consequently for every $\omega,\tau\in \genJac[2](k^s)$ we have 
	\begin{equation*}
	q_{\sh{M}}(\omega+\tau) = e_2(\omega,\tau) q_{\sh{M}}(\omega) q_{\sh{M}}(\tau).
	\end{equation*}
\end{lemma}
\begin{proof}
	By \cite[Proposition 2 of \S2]{Mumford-eqdefAVs} we have $q_{\sh{M}}(\omega) = (-1)^{m_{\omega+\kappa}(W_{g-1}) + m_{\kappa}(W_{g-1})}$, where $m_x(D)$ denotes the multiplicity of a divisor $D$ at a point $x$.
	By Riemann's singularity theorem, the multiplicity of the theta divisor $W_{g-1}$ at a point $a\in J^{g-1}_X$ is exactly $h^0(a)$.
	Combining the last two sentences proves the first identity.
	The second one follows from the first one and Lemma \ref{lemma: mumford construction quadratic form to theta char}.
\end{proof}

Following Mumford \cite[Definition above Proposition 3 of \S2]{Mumford-eqdefAVs}, the quadratic form $q_{\sh{M}}$ allows us to define the subgroup $\mathcal{H} \subset G(\sh{M}^2)$ as 
$$\mathcal{H} \coloneqq \left\{\widetilde{\omega} \in G(\sh{M}^2) \mid \widetilde{\omega}^2 = q_{\sh{M}}(\omega)    \right\}.$$
(Here we write $\omega$ for the projection of $\widetilde{\omega}$ in $\genJac[2]$.)
Lemma \ref{lemma: quadratic form mumford theta group is in fact a quadratic form} implies that $\mathcal{H}$ is indeed a subgroup and it inherits a $\Gamma_k$-action since $\sh{M}$ is defined over $k$. 
It fits into the central extension 
\begin{equation*}
1\rightarrow \{\pm1 \} \rightarrow \mathcal{H} \rightarrow \genJac[2] \rightarrow 1.
\end{equation*}
We now show how we can realize $\mathcal{H}$ as the Galois group of a covering space of schemes. This approach is certainly not new but we have been unable to find an adequate reference for it. 
First we recall for an invertible sheaf $\sh{L}$ on $\genJac$ its associated $\G_m$-torsor $\Gmtorsor{\sh{L}} \rightarrow \genJac$, the complement of the zero section in the total space of $\sh{L}$. 
For a scheme $S$ over $\Spec k$, the $S$-points of $\Gmtorsor{\sh{L}}$ are given by pairs $(x,\alpha)$ where $x: S\rightarrow \genJac$ is an $S$-valued point of $\genJac$ and $\alpha$ is an isomorphism $\O_S \xrightarrow{\sim} x^*\sh{L}$. 

We now define the morphism $p\colon \Gmtorsor{\sh{M}^2} \rightarrow \Gmtorsor{\sh{M}}$ which will be the desired $\mathcal{H}$-torsor and sits in the following commutative diagram:
\begin{center}
	\begin{tikzcd}
		\Gmtorsor{\sh{M}^2} \arrow[d] \arrow[r, "p"] & \Gmtorsor{\sh{M}} \arrow[d] \\
		\genJac \arrow[r, "\times 2"]                      & \genJac                          
	\end{tikzcd}
\end{center}
First we choose a rigidification of $\sh{M}$ i.e. an isomorphism $\sh{M}_0 \simeq k$. (The morphism we will construct depends on this choice but this does not cause any problems.)
This induces rigidifications of the line bundles $[2]^*\sh{M}$ and $\sh{M}^4$ and there is a unique isomorphism $F\colon [2]^*\sh{M} \xrightarrow{\sim} \sh{M}^4$ afforded by the theorem of the cube which respects these rigidifications. 
Given a pair $(x,\alpha)$ corresponding to an $S$-valued point of $\Gmtorsor{\sh{M}^2}$, consider the tensor square $\alpha^{\otimes2}$ of $\alpha$, which is an isomorphism $\alpha^{\otimes 2}\colon \O_S \xrightarrow{\sim} x^*\sh{M}^4$. 
Pulling back $F$ along $x$ defines an isomorphism
$$\left([2]\circ x\right)^*\sh{M}  = x^*\left([2]^*\sh{M}\right) \simeq x^*\sh{M}^4. $$
Composing $\alpha^{\otimes2}$ with the inverse of this isomorphism defines an isomorphism $\beta\colon \O_S \xrightarrow{\sim}  \left([2]\circ x\right)^*\sh{M}$. 
We define $p$ on $S$-points of $\Gmtorsor{\sh{M}^2}$ by sending the pair $(x,\alpha)$ to the pair $([2]\circ x,\beta)$ via the procedure just described.

\begin{proposition}\label{proposition: V(M) admits a H-torsor}
	The morphism $p\colon \Gmtorsor{\sh{M}^2}\rightarrow \Gmtorsor{\sh{M}}$ has the natural structure of a right $\mathcal{H}$-torsor.
\end{proposition}
\begin{proof}
	For the proof of this proposition it will be  useful to give a different interpretation of $\mathcal{H}$. For any $(\omega,\phi) \in G(\sh{M}^2)$, there is a unique $\Phi \colon [2]^*\sh{M} \rightarrow [2]^*\sh{M}$ such that the following diagram commutes:
	\begin{center}
		\begin{tikzcd}
			\sh{M}^4 \arrow[r, "\phi^{\otimes2}"] \arrow[d, "F^{-1}"] & t_{\omega}^*\sh{M}^4 \arrow[d, "t_{\omega}F^{-1}"]      \\
			{[2]^*\sh{M}} \arrow[rd, "\Phi", dashed]                  & {t^*_{\omega}[2]^*\sh{M}} \arrow[d, "\simeq"] \\
			& {[2]^*\sh{M}}                                
		\end{tikzcd}
	\end{center}
	Here $t_{\omega}^*[2]^*\sh{M}\simeq [2]^*\sh{M}$ is the canonical isomorphism. The morphism $\Phi$ does not depend on the choice of rigidification of $\sh{M}$. 
	Then \cite[Proposition 6 of \S2]{Mumford-eqdefAVs} shows that $(\omega,\phi)$ lies in the subgroup $\mathcal{H}$ of $G(\sh{M}^2)$ if and only if $\Phi$ is the identity. 
	Using this fact we can define the action of $\mathcal{H}$ on $\Gmtorsor{\sh{M}^2}$ as follows. Take a pair $(x,\alpha)$ corresponding to an $S$-valued point of $\Gmtorsor{\sh{M}^2}$ and an $S$-valued point $(\omega,\phi) \in \mathcal{H}$. We define 
	$$(x,\alpha) \cdot (\omega,\phi) := (t_{\omega}\circ x, x^*\phi \circ \alpha).$$
	One readily checks that this is a well-defined right action of $\mathcal{H}$ on $\Gmtorsor{\sh{M}^2}$ which is $\Gamma_k$-equivariant. The different interpretation of $\mathcal{H}$ shows that the action commutes with $p$.  Moreover, it acts simply transitively on the geometric fibres of $p$.
	
\end{proof}

We specialize the above construction to our situation of interest: for each $b\in \bigB^{\rs}(k)$, the theta characteristic $\kappa = 2P_{\infty}$ on $\bigprojcurve_b$ defines a central extension of finite group schemes over $k$:
\begin{equation} \label{equation: subgroup mumford theta extension}
1 \rightarrow \{\pm 1 \} \rightarrow \mathcal{H}_b \rightarrow \Jac_b[2] \rightarrow 1.
\end{equation}

We can globalize this to the family of smooth projective curves $\bigprojcurve^{\rs}\rightarrow \bigB^{\rs}$. Indeed, recall that $\Jac\rightarrow \bigB^{\rs}$ denotes the relative Jacobian of this family. 
Since $\bigprojcurve^{\rs} \rightarrow \bigB^{\rs}$ has a section $P_{\infty}$, the scheme $\Jac$ parametrizes rigidified line bundles on $\bigprojcurve^{\rs} \rightarrow \bigB^{\rs}$ \cite[Theorem 9.2.5]{Kleiman-PicardScheme}.
We can define a line bundle $\sh{M}$ on $\Jac$ using the relative theta divisor (see the proof of \cite[\S9.4; Proposition 4]{BLR-NeronModels} for its construction). By adapting the definition of $G(\sh{M})$ to the relative situation (see \cite[Expos\'e 7; Definition 3.1]{Pinceauxcourbesgenresdeux}), we obtain a $\bigB^{\rs}$-group scheme $\sh{G}(\sh{M})$ sitting in an exact sequence of smooth group schemes (Proposition 3.2 of loc. cit.):
$$
1\rightarrow \G_{m,\bigB^{\rs}} \rightarrow \sh{G}(\sh{M}^2) \rightarrow \Jac[2] \rightarrow 1.
$$
By the same procedure as the beginning of this section, we obtain a quadratic form $q_{\sh{M}}\colon \Jac[2] \rightarrow \{\pm 1\}$ and we define $\bigHeis$ as the kernel of the group homomorphism $\sh{G}(\sh{M}) \rightarrow \{\pm 1\}, \widetilde{\omega} \mapsto q_{\sh{M}}(\omega) \widetilde{\omega}^2$. 
It sits in an exact sequence of finite \'etale group schemes
\begin{align*}
1\rightarrow \{\pm 1\} \rightarrow \bigHeis \rightarrow \Jac[2] \rightarrow 1
\end{align*}
which for each $k$-point $b$ specializes to the exact sequence (\ref{equation: subgroup mumford theta extension}). 
Once a rigidification for $\sh{M}$ is chosen (which is possible since $\bigB^{\rs}$ has trivial Picard group), we can define a morphism $p\colon \Gmtorsor{\sh{M}^2} \rightarrow \Gmtorsor{\sh{M}}$ and the same logic as the proof of Proposition \ref{proposition: V(M) admits a H-torsor} shows that $p$ acquires the structure of an $\bigHeis$-torsor.

\subsection{Comparing two central extensions}\label{subsection: comparting two central extensions}

In this section we compare $\bigHeis$ with a finite \'etale group scheme coming from the representation theory of the pair $(G,V)$. The consequences of this comparison that will be used later in the paper are summarized in Corollary \ref{corollary: commutative diagram corollary central extensions}.

Recall from \S\ref{subsection: a stable grading} that the group $\bigG$ is a split simple group over $\Q$ isomorphic to $\PSp_8$. Write $\bigG^{sc} \rightarrow \bigG$ for its simply connected cover. We have an exact sequence 
$$1\rightarrow \{\pm 1\} \rightarrow \bigG^{sc} \rightarrow \bigG \rightarrow 1. $$
In \S\ref{subsection: a family of curves} we have defined a family of maximal tori $\bigA\rightarrow \bigB^{\rs}$ in $\bigH$ with the property that $\bigA \cap \bigG_{\bigB^{\rs}} = \bigA[2]$. 
By Lemma \ref{lemma: centralizer kostant same as mod 2 root lattice} there is a natural isomorphism of $\bigB^{\rs}$-group schemes $\bigA[2] \simeq \bigLambda/2\bigLambda$. 
Taking the pullback of the inclusion $\bigLambda/2\bigLambda \hookrightarrow \bigG_{\bigB^{\rs}} = \bigG\times \bigB^{\rs}$ along the morphism $\bigG_{\bigB^{\rs}}^{sc}  \rightarrow \bigG_{\bigB^{\rs}}$ yields a commutative diagram with exact rows
\begin{center}
\begin{tikzcd}
		1 \arrow[r] & \{\pm1\} \arrow[r]                & \bigG_{\bigB^{\rs}}^{sc} \arrow[r] & \bigG_{\bigB^{\rs}} \arrow[r] & 1 \\
		1 \arrow[r] & \{\pm1\} \arrow[r] \arrow[u, "="] & \bigExt \arrow[r] \arrow[u]       & {\bigLambda/2\bigLambda   } \arrow[r] \arrow[u] & 1
\end{tikzcd}
\end{center}
where the righthand square is pullback. The finite \'etale group scheme $\bigExt \rightarrow \bigB^{\rs}$ is a central extension of $\bigLambda/2\bigLambda$ by $\{\pm 1\}$. It is isomorphic to $Z_{\bigG^{sc}}(\sigma|_{\bigB^{\rs}})$, the simply connected centralizer of the Kostant section.

On the other hand, in the previous subsection we have defined a group scheme $\bigHeis$, a subgroup of a Mumford theta group, which fits in the exact sequence of \'etale group schemes 
$$1\rightarrow \{\pm 1\} \rightarrow \bigHeis \rightarrow \Jac[2] \rightarrow 1.$$

The following proposition is a central technical result of this paper. It lifts the isomorphism $\Lambda/2\Lambda \rightarrow \Jac[2]$ obtained in \cite[Corollary 4.12]{Thorne-thesis} to an isomorphism between the nonabelian groups $\bigExt$ and $\bigHeis$. 

\begin{proposition}\label{propostion: 2 central extensions coincide}
	There exists a unique isomorphism $\bigExt\simeq\bigHeis$ of group schemes over $\bigB^{\rs}$ that preserves the subgroup $\{\pm 1\}$ and such that the induced isomorphism $\bigLambda/2\bigLambda\simeq \Jac[2]$ coincides with the one from Proposition \ref{proposition: bridge jacobians root lattices}. 
\end{proposition}

Since $\bigExt$ and $\bigHeis$ are finite \'etale, by \cite[Tag \href{https://stacks.math.columbia.edu/tag/0BQM}{0BQM}]{stacksproject} it suffices to prove the statements over the generic point $\eta$ of $\bigB^{\rs}$.
This we will achieve at the end of this section after some preparatory lemmas. 
Write $\genHeis$ and $\genExt$ for the generic fibres of $\bigHeis$ and $\bigExt$ respectively. 
We write $k$ for the function field of $\eta$ with separable closure $k^s$ and absolute Galois group $\Gamma_k$. We choose a geometric generic point $\bar{\eta}\colon \Spec k^s\rightarrow \bigB^{\rs}$ over $\eta$. 

We first prove that such an isomorphism exists when we forget the $\Gamma_k$-action. 
\begin{lemma}\label{lemma: 2 central extensions are isomorphic a abstract groups}
	There is an isomorphism of groups $\genExt_{\bar{\eta}}\simeq \genHeis_{\bar{\eta}}$ compatible with the central extensions. 
\end{lemma}
\begin{proof}
    By \cite[Theorem 2.4.1]{Lurie-minisculereps}, central extensions of $\bigLambda_{\bar{\eta}}/2\bigLambda_{\bar{\eta}}$ by $\{\pm 1\}$ as abstract groups are classified by quadratic forms of $\Lambda_{\bar{\eta}}/2\Lambda_{\bar{\eta}}$. 
    According to \cite[Proposition A.2]{thorne-planequarticsAIT}, the quadratic form corresponding to $\genExt_{\bar{\eta}}$ is given by the standard quadratic form $\bigLambda_{\bar{\eta}}/2\bigLambda_{\bar{\eta}} \rightarrow \{\pm 1\} \colon \lambda \mapsto (-1)^{(\lambda,\lambda)/2}$. 
    By Proposition \ref{prop: quadratic forms identified} and Lemma \ref{lemma: quadratic form mumford theta group is in fact a quadratic form}, this coincides with the quadratic form corresponding to $\genHeis_{\bar{\eta}}$ transported along the isomorphism $\Lambda_{\bar{\eta}}/2\Lambda_{\bar{\eta}} \simeq \Jac_{\bar{\eta}}[2]$.
\end{proof}

For the rest of this section we fix abstract groups $\widetilde{\mathrm{V}}$ and $\mathrm{V}$ and a central extension 
\begin{displaymath}
1 \rightarrow \{\pm 1 \} \rightarrow \widetilde{\mathrm{V}} \rightarrow \mathrm{V} \rightarrow 1
\end{displaymath}
that is isomorphic to the central extension $\genExt_{\bar{\eta}}$ of $\bigLambda_{\bar{\eta}}/2\bigLambda_{\bar{\eta}}$ by $\{\pm1\}$. (We hope that the group $\mathrm{V}$, which is only used in \S\ref{subsection: comparting two central extensions}, will not be confused with the representation $V$.)
This extension comes with a quadratic form $q\colon \mathrm{V} \rightarrow \{ \pm 1 \}$ defined by $q(v) = \widetilde{v}^2$ where $\widetilde{v}$ is a lift of $v$ to $\widetilde{\mathrm{V}}$. 

It will be useful to give a presentation of the group $\widetilde{\mathrm{V}}$. 
Let $e_1,\dots,e_6$ be a basis for the $\F_2$-vector space $\mathrm{V}$, which we assume satisfies $q(e_1) = -1$.
If we choose a lift $\widetilde{e}_i \in \widetilde{\mathrm{V}}$ of $e_i$, a presentation for $\widetilde{\mathrm{V}}$ is given as follows:
\begin{itemize}
    \item The generators are given by the symbols $\widetilde{e}_i$ for $i =1\dots,6$.
    \item The relations are given by (we set $-1 \coloneqq \widetilde{e}_1^2$):
    \begin{displaymath}
    \begin{cases}
    (-1)^2 = 1,\\
    \widetilde{e}_i^2 = q(e_i),\\
    [\widetilde{e}_i,-1] = 1,\\
    [\widetilde{e}_i,\widetilde{e}_j]=q(e_i)q(e_j)q(e_i+e_j).
    \end{cases}
    \end{displaymath}
    
\end{itemize}

The proof of the following lemma is purely group-theoretic.

\begin{lemma}\label{lemma: group theory with central extension and the free group}
	Let $F_6$ be the free group on six generators and $f\colon F_6 \rightarrow \mathrm{V} $ a surjective homomorphism. 
	If $\widetilde{f},\widetilde{f}'$ are two surjective homomorphisms $F_6 \rightarrow \widetilde{\mathrm{V}}$ lifting $f$, then there exists a unique isomorphism $\phi \colon \widetilde{\mathrm{V}} \rightarrow \widetilde{\mathrm{V}}$ such that $\widetilde{f}' = \phi \widetilde{f}$. 
\end{lemma}
\begin{proof}
    To prove the lemma it suffices the prove that $\ker \widetilde{f} = \ker \widetilde{f}'$. 
    Since any two lifts of an element of $\mathrm{V}$ to an element of $\widetilde{\mathrm{V}}$ differ by an element of $\{ \pm 1\}$, there exists a function $\chi: F_6 \rightarrow \{\pm 1 \}$ such that $\widetilde{f}'(g) = \chi(g) \widetilde{f}(g)$ for all $g \in F_6$. 
    Since the subgroup $\{\pm 1\}$ is central in $\widetilde{\mathrm{V}}$, $\chi$ is a homomorphism of groups. 
    So it will be enough to show that $\ker(\chi \widetilde{f}) = \ker \widetilde{f}$ for every character $\chi: F_6 \rightarrow \{ \pm 1\}$, where now $\widetilde{f}$ is a preferred choice of lifting of $f$. 
    We make this choice as follows. Choose generators $g_1,\dots,g_6$ of $F_6$ and let $e_i = f(g_i)$. We may assume that $q(e_1) = -1$. 
    Choose an element $\widetilde{e}_i \in \widetilde{\mathrm{V}}$ lying above $e_i$. 
    We define $\widetilde{f}\colon F_6 \rightarrow \widetilde{\mathrm{V}}$ by sending $g_i$ to $\widetilde{e}_i$. 
    Then the presentation of $\widetilde{\mathrm{V}}$ given above implies that the kernel of $\widetilde{f}$ is generated by the following words:
    \begin{displaymath}
    \begin{cases}
    g_1^4,\\
    g_i^2Q(g_i), \\
    [g_i,g_1^2], \\
    [g_i,g_j]Q(g_i)Q(g_j)Q(g_ig_j).
    \end{cases}
    \end{displaymath}
    Here we set $Q(g) \coloneqq g_1^2$ if $q(f(g)) = -1$ and $Q(g) \coloneqq 1$ otherwise. 
    Since every such word has trivial image under every character $\chi: F_6 \rightarrow \{ \pm 1\}$, we see that the kernel of $\chi \widetilde{f}$ is generated by the same words. This concludes the proof of the lemma. 
\end{proof}

We now investigate the structure of the \'etale fundamental group of the affine curve $\bigaffcurve_{\eta} = \bigprojcurve_{\eta} \setminus \{P_{\infty}\}$ where $P_{\infty}$ is the marked $k$-rational point at infinity. 
Choose an isomorphism between $k[[t]]$ and the completed local ring of $\bigprojcurve_{\eta}$ at $P_{\infty}$, and write $\Spec k[[t]] \rightarrow \bigprojcurve_{\eta}$ for the induced map on schemes. Let $y\colon \Spec k((t)) \rightarrow \bigaffcurve_{\eta}$ be the restriction of this map to $\bigaffcurve_{\eta}$. 
Let $\Omega$ be a separable closure of $k((t))$ and $\bar{y}\colon\Spec \Omega \rightarrow \bigaffcurve_{\eta} $ be a geometric point over it. 
The geometric point $\bar{y}$ will serve as our basepoint of $\bigaffcurve_{\bar{\eta}}$, and is sometimes called a \emph{tangential basepoint}, following \cite[\S15]{Deligne-droiteprojective}. 
We write $\pi_1(\bigaffcurve_{\bar{\eta}},\bar{y})$ for the \'etale fundamental group of $\bigaffcurve_{\bar{\eta}}$ with respect to the geometric point $\bar{y}$. 
It is isomorphic to the profinite completion of the free group on six generators, and acquires a natural continuous $\Gamma_{k((t))}$-action since $\bar{y}$ comes from a $k((t))$-rational point. 
The natural map $\Gamma_{k((t))} \rightarrow \Gamma_k$ has a splitting (since $\Omega = \cup_{n\geq 1} k^s((t^{1/n}))$ because $k$ has characteristic $0$), so the group $\pi_1(\bigaffcurve_{\bar{\eta}},\bar{y})$ also has a continuous $\Gamma_k$-action. 
We will construct homomorphisms from $\pi_1(\bigaffcurve_{\bar{\eta}},\bar{y})$ into various groups by considering torsors of $\bigaffcurve_{\bar{\eta}}$ under these groups. The following lemma, which follows from the definition of the \'etale fundamental group, explains how this works.

\begin{lemma}\label{lemma: constructing torsors}
	Let $\mathcal{G}$ be a finite $k$-group equipped with the discrete topology. Let $T$ be a scheme over $k$ and $T\rightarrow \bigaffcurve_{\eta}$ a right $\mathcal{G}$-torsor. Let $\bar{t}\colon \Spec\Omega \rightarrow T$ be a geometric point above $\bar{y}$. Then we can associate to this data a continuous homomorphism $\phi_{\bar{t}}\colon \pi_1(\bigaffcurve_{\bar{\eta}},\bar{y}) \rightarrow \mathcal{G}_{k^s}$. It is surjective if and only if $T$ is geometrically connected. 
	Let $\bar{t}'$ be another geometric point of $T$ above $\bar{y}$.  Then $\bar{t}' = \bar{t}\cdot h$ for some $h \in \pi_1(\bigaffcurve_{\bar{\eta}},\bar{y})$ and $\phi_{\bar{t}'}$ is given by the composition of $\phi_{\bar{t}}$ with conjugation by $\phi_{\bar{t}}(h)$. 
\end{lemma}

Let $\bigprojcurve'\rightarrow \bigprojcurve_{\eta}$ be the $\Jac_{\eta}[2]$-torsor given by pulling back the multiplication-by-$2$ map $\Jac_{\eta} \xrightarrow{\times 2} \Jac_{\eta}$ via the Abel--Jacobi map with respect to the point $P_{\infty}$. 
There exists an obvious $k$-rational point above $P_{\infty}$ in $\bigprojcurve'$ corresponding to the origin in $\Jac_{\eta}$ upstairs. 
Define $T_1$ as the restriction of $C'$ to $\bigaffcurve_{\eta}$. 
Then the $k((t))$-rational point $y\colon \Spec k((t)) \rightarrow \bigaffcurve_{\eta}$ lifts to a $k((t))$-rational point $t_1\colon \Spec k((t)) \rightarrow T_1$. 
Using Lemma \ref{lemma: constructing torsors} we obtain a continuous $\Gamma_k$-equivariant homomorphism $\pi_1(\bigaffcurve_{\bar{\eta}},\bar{y}) \rightarrow \Jac_{\eta}[2]$.

On the other hand, we define the $\bigLambda_{\eta}/2\bigLambda_{\eta}$-torsor $T_2 \rightarrow \bigaffcurve_{\eta}$ as follows: recall from \cite[Proposition 3.6]{Thorne-thesis} that $\bigaffcurve_{\eta}$ can be realized as a closed subscheme of $\bigV_{\eta}$. 
We know the action map $\bigG_{\eta} \rightarrow \bigV_{\eta} : g\mapsto g\cdot \sigma(\eta)$ is \'etale, and in fact a torsor under the group $Z_{\bigG_{\eta}}(\sigma(\eta))$. 
Taking the pullback along $\bigaffcurve_{\eta} \rightarrow \bigV_{\eta}$ and transporting the torsor structure along the isomorphism $Z_{\bigG_{\eta}}(\sigma(\eta)) \simeq \bigLambda_{\eta}/2\bigLambda_{\eta}$ defines a $\bigLambda_{\eta}/2\bigLambda_{\eta}$-torsor $T_2$ such that the following diagram is commutative. (This diagram already appears right above Theorem 4.2 in \cite{Thorne-thesis}.)
\begin{center}
\begin{tikzcd}
	T_2 \arrow[d] \arrow[r] & G_{\eta} \arrow[d] \\
	\bigaffcurve_{\eta} \arrow[r]                & V_{\eta}          
\end{tikzcd}
\end{center}

In the proof of \cite[Theorem 4.15]{Thorne-thesis}, Thorne shows: 
\begin{lemma}\label{lemma: Thorne thesis torsor extends and is iso}
	The torsor $T_2 \rightarrow \bigaffcurve_{\eta}$ extends to a $\bigLambda_{\eta}/2\bigLambda_{\eta}$-torsor $\widetilde{\bigprojcurve} \rightarrow \bigprojcurve_{\eta}$. Moreover, the pushout of the torsor $\widetilde{\bigprojcurve}$ along the isomorphism $\bigLambda_{\eta}/2\bigLambda_{\eta} \simeq \Jac_{\eta}[2]$ from Proposition \ref{proposition: bridge jacobians root lattices} is isomorphic to $C'\rightarrow \bigprojcurve_{\eta}$. 
\end{lemma}

So again there exists a point $t_2\colon \Spec k((t)) \rightarrow T_2$ lifting the point $y$, which we will fix. We obtain a continuous $\Gamma_k$-equivariant homomorphism $\pi_1(\bigaffcurve_{\bar{\eta}},\bar{y}) \rightarrow \bigLambda_{\eta}/2\bigLambda_{\eta}$.

\begin{lemma}\label{lemma: existence of non-abelian torsors}
	We have the following. 
	\begin{enumerate}
		\item There exists an $\genHeis$-torsor $\widetilde{T}_1 \rightarrow \bigaffcurve_{\eta}$ which factors as $\widetilde{T}_1  \rightarrow T_1 \rightarrow \bigaffcurve_{\eta}$. The $k$-scheme $\widetilde{T}_1 $ is geometrically connected. Moreover there exists a $\Gamma_k$-equivariant continuous homomorphism $\pi_1(\bigaffcurve_{\bar{\eta}} , \bar{y}) \rightarrow \genHeis$ lifting the morphism $\pi_1(\bigaffcurve_{\bar{\eta}} , \bar{y}) \rightarrow \Jac_{\eta}[2]$.
		\item There exists a $\genExt$-torsor $\widetilde{T}_2 \rightarrow \bigaffcurve_{\eta}$ which factors as $\widetilde{T}_2 \rightarrow T_2 \rightarrow \bigaffcurve_{\eta}$. The $k$-scheme $\widetilde{T}_2$ is geometrically connected.
		Moreover there exists a $\Gamma_k$-equivariant continuous homomorphism $\pi_1(\bigaffcurve_{\bar{\eta}} , \bar{y}) \rightarrow \genExt$ lifting the morphism $\pi_1(\bigaffcurve_{\bar{\eta}} , \bar{y}) \rightarrow \bigLambda_{\eta}/2\bigLambda_{\eta}$.
	\end{enumerate}
\end{lemma}
\begin{proof}
	For Part 1, recall from \S\ref{subsection: Mumford theta groups} that there exists $k$-schemes $\Gmtorsor{\sh{M}}$ and $\Gmtorsor{\sh{M}^2}$ together with an $\genHeis$-torsor $p\colon \Gmtorsor{\sh{M}^2}\rightarrow \Gmtorsor{\sh{M}}$ and a commutative diagram 
	\begin{center}
		\begin{tikzcd}
			\Gmtorsor{\sh{M}^2} \arrow[d] \arrow[r, "p"] & \Gmtorsor{\sh{M}} \arrow[d] \\
			\Jac_{\eta} \arrow[r, "\times 2"]                      & \Jac_{\eta}                         
		\end{tikzcd}
	\end{center}
	Let $i\colon \bigaffcurve_{\eta} \rightarrow \Jac_{\eta}$ be the Abel-Jacobi map with respect to the point $P_{\infty}$. If follows from \cite[Exercise 10 of Chapter 11]{BirkenhakeLange-CAV} that the pullback of the line bundle $\sh{M}$ along $i$ is trivial. 
	In other words, the map $i\colon \bigaffcurve_{\eta} \rightarrow \Jac_{\eta}$ lifts to a map $\widetilde{i}\colon \bigaffcurve_{\eta} \rightarrow \Gmtorsor{\sh{M}}$. Taking the pullback of the torsor $p$ along $\widetilde{i}$ defines an $\genHeis$-torsor $\widetilde{T}_1  \rightarrow \bigaffcurve_{\eta}$ which factors as $\widetilde{T}_1 \rightarrow T_1 \rightarrow \bigaffcurve_{\eta}$ compatible with the torsor structures. 
	
	Let $\bar{t}_1 \colon \Spec \Omega \rightarrow T_1$ be a geometric point above $t_1$. 
	Choose a geometric point $\bar{t}_1'\colon \Spec \Omega \rightarrow \widetilde{T}_1 $ lying above $\bar{t}_1$. By Lemma \ref{lemma: constructing torsors} it determines a continuous homomorphism $\phi \colon \pi_1(\bigaffcurve_{\bar{\eta}},\bar{y}) \rightarrow \genHeis_{k^s}$, whose projection to $\Jac_{\eta}[2]$ gives the previously constructed morphism $\pi_1(\bigaffcurve_{\bar{\eta}},\bar{y})\rightarrow \Jac_{\eta}[2]$. Changing $\bar{t}_1'$ means conjugating $\phi$ by an element of the form $\phi(h)$ where $h\in \pi_1(\bigaffcurve_{\bar{\eta}},\bar{y})$ lies in the image of the map $\pi_1(\left(T_1\right)_{k^s},\bar{t}_1) \rightarrow \pi_1(\bigaffcurve_{\bar{\eta}},\bar{y})$. 
	But in that case $\phi(h) \in \{\pm 1\}$, so $\phi(h)$ lies in the centre of $\genHeis_{k^s}$. 
	We conclude that the homomorphism $\phi$ is independent of the choice of $\bar{t}_1'$, hence is $\Gamma_k$-equivariant. 
	Moreover, the image of $\phi$ is a subgroup of $\genHeis_{k^s}$ whose projection to $\Jac_{\bar{\eta}}[2]$ is surjective. Since the $\{\pm 1\}$-extension $\genHeis_{k^s} \rightarrow \Jac_{\bar{\eta}}[2]$ is not split, this implies that $\phi$ itself must be surjective. We conclude that $\widetilde{T}_1$ must be geometrically connected, concluding Part 1 of the lemma. 
	
	For Part 2, we complete the diagram in the definition of $T_2$ to the following diagram 
	\begin{center}
	\begin{tikzcd}
	\widetilde{T}_2 \arrow[d] \arrow[r]   & G^{sc}_{\eta} \arrow[d] \\
	T_2 \arrow[d] \arrow[r] & G_{\eta} \arrow[d]      \\
	\bigaffcurve_{\eta} \arrow[r]                & V_{\eta}               
	\end{tikzcd}
	\end{center}
	where both squares are pullback and $G^{sc}_{\eta} \rightarrow G_{\eta}$ is the natural projection. Since $G^{sc}_{\eta} \rightarrow V_{\eta}$ is a $Z_{\bigG^{sc}}(\sigma(\eta))$-torsor, the morphism $\widetilde{T}_2 \rightarrow \bigaffcurve_{\eta}$ is a $\genExt$-torsor. 
	A similar argument to Part 1 shows that this data defines a homomorphism $ \pi_1(\bigaffcurve_{\bar{\eta}} , \bar{y}) \rightarrow \genExt_{k^s} $ which is independent of any choices, surjective and $\Gamma_k$-equivariant.
	
\end{proof}

We have completed all the preparations for the proof of Proposition \ref{propostion: 2 central extensions coincide}, which we give now. 
Ignoring the dotted arrow, Lemma \ref{lemma: existence of non-abelian torsors} implies the existence of the following diagram, commutative by Lemma \ref{lemma: Thorne thesis torsor extends and is iso}:
\begin{center}
\begin{tikzcd}
	{\pi_1(\bigaffcurve_{\bar{\eta}},\bar{y})} \arrow[r] \arrow[rd] & \genHeis \arrow[r] \arrow[d, "\Psi", dashed] & {\Jac_{\eta}[2]} \arrow[d, "\simeq"] \\
	& \genExt \arrow[r]                            & \bigLambda_{\eta}/2\bigLambda_{\eta}
\end{tikzcd}
\end{center}
By Lemma \ref{lemma: 2 central extensions are isomorphic a abstract groups}, we are in the situation of Lemma \ref{lemma: group theory with central extension and the free group}, so there exists a unique isomorphism $\Psi\colon \genHeis_{k^s} \rightarrow \genExt_{k^s}$ such that the above diagram with the dotted arrow added is commutative. (Lemma \ref{lemma: group theory with central extension and the free group} can be applied even if $F_6$ is the profinite completion of the free group on six generators since we are dealing with finite quotients here.)
Since the maps $\pi_1(\bigaffcurve_{\bar{\eta}},\bar{y}) \rightarrow \genHeis$ and $\pi_1(\bigaffcurve_{\bar{\eta}},\bar{y}) \rightarrow \genExt$ are $\Gamma_k$-equivariant, $\Psi$ is $\Gamma_k$-invariant as well. This proves that $\genExt$ and $\genHeis$ are isomorphic. 

To prove uniqueness, we note that the scheme of isomorphisms $\genExt \simeq \genHeis$ compatible with the central extensions is a torsor under the group $\left(\bigLambda_{\eta}/2\bigLambda_{\eta}\right)^{\vee} \simeq \bigLambda_{\eta}/2\bigLambda_{\eta}$, by \cite[Lemma 2.4]{thorne-planequarticsAIT}. Since $\bigLambda_{\eta}/2\bigLambda_{\eta}$ does not have any non-identity $k$-rational points by Proposition \ref{proposition: monodromy of J[2]}, this completes the proof of Proposition \ref{propostion: 2 central extensions coincide}.

\begin{corollary}\label{corollary: commutative diagram corollary central extensions}
	There is a commutative diagram of $\bigB^{\rs}$-group schemes with exact rows
	\begin{center}
	\begin{tikzcd}
		1 \arrow[r] & \{\pm1\} \arrow[r]                & \bigG_{\bigB^{\rs}}^{sc} \arrow[r] & \bigG_{\bigB^{\rs}} \arrow[r] & 1 \\
		1 \arrow[r] & \{\pm1\} \arrow[r] \arrow[u, "="] & \bigHeis \arrow[r] \arrow[u]       & {\Jac[2]} \arrow[r] \arrow[u] & 1
	\end{tikzcd}
	\end{center}
	enjoying the following properties:
	\begin{enumerate}
		\item The rightmost vertical arrow equals the composite of the isomorphism $\Jac[2] \simeq Z_{\bigG}(\sigma|_{\bigB^{\rs}})$ from Proposition \ref{proposition: bridge jacobians root lattices} with the inclusion $Z_{\bigG}(\sigma|_{\bigB^{\rs}}) \hookrightarrow \bigG_{\bigB^{\rs}}$. 
		\item The right-hand square is cartesian. 
	\end{enumerate}
\end{corollary}

\subsection{Embedding the Selmer group}\label{subsection: twisting and embedding the selmer group}

We start with a well-known lemma which provides the link between the rational orbits of our representations and \'etale cohomology. 
Its proof will be postponed to the proof of Proposition \ref{proposition: G-orbits in terms of groupoids} and is largely formal. The case of a field is treated in \cite[Proposition 1]{BhargavaGross-AIT} and the more general case is based on the same idea. 
Recall that for a $\Q$-algebra $R$ and an element $b\in \bigB(R)$ we write $\bigV_b$ for the pullback of the morphism  $\bigpi\colon \bigV \rightarrow \bigB$ along $b$. 

\begin{lemma}\label{lemma: AIT}
	Let $R$ be a $\Q$-algebra. If $b\in \bigB^{\rs}(R)$ then there is a canonical bijection of sets
	$$ \bigG(R)\backslash \bigV_b(R) \simeq \ker\left(\HH^1(R,Z_{\bigG}(\sigma(b)) ) \rightarrow \HH^1(R,\bigG)\right).$$
	The distinguished orbit $\bigG(R)\cdot \bigsigma(b)$ corresponds to the trivial element in $\HH^1(R,Z_G(\sigma(b) ))$.
\end{lemma}

The bijection is given by sending the orbit $\bigG(R)\cdot v$ to the isomorphism class of the $Z_{\bigG}(\sigma(b))$-torsor $\{g\in \bigG \mid g\cdot v  = \sigma(b)  \} \rightarrow \Spec R$. 

\begin{lemma}\label{lemma: H^1(R,Sp) is trivial}
	Let $R$ be a $\Q$-algebra such that every locally free $R$-module of constant rank is free. For each $n\geq 1$ write $\Sp_{2n}$ for the split symplectic group over $\Q$ of rank $n$. 
	Then the pointed set $\HH^1(R,\Sp_{2n})$ is trivial for all $n\geq 1$.
	In particular, the pointed set $\HH^1(R,G^{sc})$ is trivial. 
\end{lemma}
\begin{proof}
	Since $\bigG\simeq \PSp_8$ we have $\bigG^{sc}\simeq \Sp_8$ so it suffices to prove the first part. The set $\HH^1(R,\Sp_{2n})$ is in canonical bijection with the set of isomorphism classes of pairs $(M,b)$, where $M$ is a projective $R$-module of rank $2n$ and $b: M\times M \rightarrow R$ is an alternating perfect pairing. Our assumptions imply that $M$ is free and the proof of \cite[Corollary 3.5]{Milnor-SymmetricBilinearForms} shows that any two alternating perfect pairings on $M$ are isomorphic. 
\end{proof}

We now piece all the ingredients obtained so far together and deduce our first main result.

\begin{theorem}\label{theorem: inject 2-descent into orbits}
	Let $R$ be a $\Q$-algebra such that every locally free $R$-module of constant rank is free and $b\in \bigB^{\rs}(R)$. Then there is a canonical injection $\eta_b \colon \Jac_b(R)/2\Jac_b(R) \hookrightarrow \bigG(R)\backslash \bigV_b(R)$ compatible with base change on $R$. 
	Moreover, the map $\eta_b$ sends the identity element to the orbit of $\bigsigma(b)$.
\end{theorem}
\begin{proof}
	By Corollary \ref{corollary: commutative diagram corollary central extensions}, we have a commutative diagram with exact rows of group schemes over $R$ (we continue to write $G$ and $G^{sc}$ for the base change of these $\Q$-groups to $R$): 
	\begin{center}
		\begin{tikzcd}
			1 \arrow[r] & \mu_2 \arrow[r]                 & \bigG^{sc} \arrow[r]    & \bigG \arrow[r]                        & 1 \\
			1 \arrow[r] & \mu_2 \arrow[r] \arrow[u, "="'] & \bigHeis_b \arrow[r] \arrow[u] & \Jac_b[2] \arrow[r] \arrow[u] & 1
		\end{tikzcd}
	\end{center}
	Moreover by Lemma \ref{lemma: AIT}, the kernel of the map of pointed sets $\HH^1(R,\Jac_b[2]) \rightarrow \HH^1(R,\bigG)$ induced by the rightmost vertical map of the diagram is in canonical bijection with the set of $\bigG(R)$-orbits in $\bigV_b(R)$. 
	Given $A\in \Jac_b(R)$ we define $\eta_b(A) \in \HH^1(R,\Jac_b[2])$ as the image of $A$ under the $2$-descent map $\Jac_b(R)/2\Jac_b(R) \hookrightarrow \HH^1(R,\Jac_b[2])$, given by the isomorphism class of the $\Jac_b[2]$-torsor $[2]^{-1}\left(A\right)$.
	To prove that $\eta_b(A)$ defines a $\bigG(R)$-orbit in $\bigV_b(R)$ we need to show that its class is killed under the map $\HH^1(R,\Jac_b[2]) \rightarrow \HH^1(R,G)$. 
	Using the triviality of $\HH^1(R,G^{sc})$ by Lemma \ref{lemma: H^1(R,Sp) is trivial} and the above commutative diagram, it is enough to show that $\eta_b(A)$ lies in the image of the map $\HH^1(R,\bigHeis_b) \rightarrow \HH^1(R,\Jac_b[2])$. 
	
	Recall from \S\ref{subsection: Mumford theta groups} that we have a commutative diagram 
	\begin{center}
		\begin{tikzcd}
			\Gmtorsor{\sh{M}^2_b} \arrow[d] \arrow[r, "p"] & \Gmtorsor{\sh{M}_b} \arrow[d] \\
			\Jac_b \arrow[r, "\times 2"]                      & \Jac_b                          
		\end{tikzcd}
	\end{center}
	where the vertical arrows are $\G_m$-torsors and where $p$ is an $\bigHeis_b$-torsor. 
	Since $\HH^1(R,\G_m)$ is trivial, the point $A\in \Jac_b(R)$ lifts to a point $\widetilde{A} \in \Gmtorsor{\sh{M}_b}(R)$. 
	Then the fibre of $p$ above $\widetilde{A}$ will be an $\bigHeis_b$-torsor lifting $\eta_b(A)$. This concludes the first part of the theorem.
	The definition of $\eta_b$ shows that it sends the identity element of $\Jac_b(R)/2\Jac_b(R)$ to the identity element of $\HH^1(R,\Jac_b[2])$. By Lemma \ref{lemma: AIT} this corresponds to the orbit of $\sigma(b)$, proving the second part of the theorem.

\end{proof}

\begin{corollary}\label{corollary: Sel2 embeds}
	Let $b\in \bigB^{\rs}(\Q)$ and write $\Sel_2 \Jac_b$ for the $2$-Selmer group of $\Jac_b$ over $\Q$. 
	Then the injection $\Jac_b(\Q)/2\Jac_b(\Q) \hookrightarrow \bigG(\Q)\backslash \bigV_b(\Q)$ of Theorem \ref{theorem: inject 2-descent into orbits} extends to an injection 
	$$\Sel_2 \Jac_b \hookrightarrow \bigG(\Q)\backslash \bigV_b(\Q).$$
\end{corollary}
\begin{proof}
	
	We have a commutative diagram for every place $v$: 
	\begin{center}
		\begin{tikzcd}
			{\Jac_b(\Q)/2\Jac_b(\Q)} \arrow[d] \arrow[r, "\delta"] & {\HH^1(\Q,\Jac_b[2])} \arrow[r] \arrow[d] & {\HH^1(\Q,\bigG)} \arrow[d] \\
			{\Jac_b(\Q_v)/2\Jac_b(\Q_v)} \arrow[r, "\delta_v"]     & {\HH^1(\Q_v,\Jac_b[2])} \arrow[r]         & {\HH^1(\Q_v,\bigG)}        
		\end{tikzcd}
	\end{center}
	To prove the corollary it suffices to prove that $2$-Selmer elements in $\HH^1(\Q,\Jac_b[2])$ are killed under the natural map $\HH^1(\Q,\Jac_b[2]) \rightarrow \HH^1(\Q,G)$. 
	By definition, an element of $\Sel_2 \Jac_b$ consists of a class in $\HH^1(\Q,\Jac_b[2])$ whose restriction to $\HH^1(\Q_v,\Jac_b[2])$ lies in the image of $\delta_v$ for every place $v$. 
	So by Theorem \ref{theorem: inject 2-descent into orbits} the image of such an element in $\HH^1(\Q_v,\bigG)$ is trivial for every $v$. Since the restriction map $\HH^2(\Q,\mu_2) \rightarrow \prod_{v} \HH^2(\Q_v,\mu_2)$ has trivial kernel by the Hasse principle for the Brauer group, the kernel of $\HH^1(\Q,G) \rightarrow \prod_{v} \HH^1(\Q_v,G)$ is trivial too. The result follows.
\end{proof}

\section{Integral orbit representatives}\label{section: integral representatives}

In this section, we introduce integral structures for the pair $(\bigG,\bigV)$ and prove that for large primes $p$, the image of the map from Theorem \ref{theorem: inject 2-descent into orbits} applied to $R = \Q_p$ lands in the orbits which admit a representative in $\Z_p$. See Theorem \ref{theorem: integral representatives exist} for a precise statement. 
In \S\ref{subsection: integral structures}, we extend our constructions over $\Z[1/N]$ for some sufficiently large integer $N$. 
In \S\ref{subsection: some groupoids} and \S\ref{subsection: compactifications} we introduce the necessary technical background for the proof of Theorem \ref{theorem: integral representatives exist}. 
In \S\ref{subsection: case of square-free discriminant} we prove the case of square-free discriminant. 
In \S\ref{subsection: proof of theorem integral representatives} we combine all the above ingredients to prove Theorem \ref{theorem: integral representatives exist} in full generality. 
Finally in \S\ref{subsection: integrality, a global corollary}, we deduce an integrality result for orbits over $\Q$ (as opposed to orbits over $\Q_p$).

\subsection{Integral structures}\label{subsection: integral structures}

The pair $(\bigG,\bigV)$ naturally extends to a pair $(\intbigG,\intbigV)$ over $\Z$ with similar properties.
Indeed, our choice of pinning of $\bigH$ in \S\ref{subsection: a stable grading} determines a Chevalley basis of $\bigh$, hence a $\Z$-form $\intlieh$ of $\bigh$ (in the sense of \cite{Borel-propertieschevalley}) with adjoint group $\intbigH$, a split semisimple group of type $E_6$ over $\Z$.
The $\Z$-lattice $\intbigV = \bigV\cap \intlieh$ is admissible; define $\intbigG$ as the Zariski closure of $\bigG$ in $\GL(\intbigV)$. 
The $\Z$-group scheme $\intbigG$ has generic fibre $\bigG$ and acts faithfully on the free $\Z$-module $\intbigV$ of rank $42$. 
The automorphism $\bigtheta\colon \bigH \rightarrow \bigH$ of \S\ref{subsection: a stable grading} extends by the same formula to an automorphism $\intbigH\rightarrow \intbigH$, still denoted by $\bigtheta$. 
We have $\intbigH^{\bigtheta}_{\Z[1/2]}=\intbigG_{\Z[1/2]}$ and $\intbigG_{\Z[1/2]}$ is a split reductive group of type $C_4$ over $\Z[1/2]$.


Our main properties and constructions obtained so far work over $\Z[1/N]$ for some sufficiently large integer $N$, as we will now explain. 
After rescaling the polynomials $p_2,\dots,p_{12}\in \Q[\bigV]^{\bigG}$ fixed in \S\ref{subsection: a family of curves} using the $\G_m$-action on $\bigV$ we can assume they lie in $\Z[\intbigV]^{\intbigG}$. 
Write $\intbigB \coloneqq \Spec\Z[p_2,\dots,p_{12}]$ and write $\pi\colon \intbigV \rightarrow \intbigB$ for the corresponding morphism which extends the morphism $V\rightarrow B$ on $\Q$-fibres, already denoted by $\pi$. 
Recall that $\Delta \in \Q[\bigV]^{\bigG}$ is the Lie algebra discriminant of $\bigh$, a $\bigG$-invariant polynomial of degree $72$. 
We can assume, again after suitable rescaling $p_2,\dots,p_{12}$ using the $\G_m$-action, that $\Delta\in \Z[\bigV]^{\bigG}$. 
We define $\intbigB^{\rs} \coloneqq \Spec \Z[p_2,\dots,p_{12}][\Delta^{-1}]$. 
We extend the family of curves given by Equation (\ref{equation : E6 family middle of paper}) to the family $\intbigcurve \rightarrow \intbigB$ given by that same equation. 

Let us call a positive integer $N$ \define{good} if the following properties are satisfied (set $S \coloneqq \Z[1/N]$): 

\begin{enumerate}
	\item Each prime dividing the order of the Weyl group (so an element of $\{2,3,5\}$) is a unit in $S$. 
	\item The discriminant locus $\{\Delta = 0\}_S \rightarrow \Spec S $ has geometrically integral fibres. Moreover $\Delta$ and $\Delta_0$ (which by formula (\ref{equation: definition discriminant Delta0}) defines an element of $\Z[\intbigB]$) are equal up to a unit in $\Z[1/N]$.
	\item The morphism $\intbigcurve_S\rightarrow \intbigB_S$ is flat and proper with geometrically integral fibres. It is smooth exactly above $\intbigB_S^{\rs}$. 
	\item $S[\intbigV]^{\intbigG} = S[p_2,p_5,p_6,p_8,p_9,p_{12}]$. The Kostant section extends to a section $\sigma\colon \intbigB_S \rightarrow \intbigV^{\reg}$ of $\pi$ satisfying the following property: for any $b\in \intbigB(\Z) \subset \intbigB_S(S)$, we have $\sigma(N\cdot b) \in \intbigV(\Z)$. 
	\item There exists open subschemes $\intbigV^{\rs} \subset \intbigV^{\reg} \subset \intbigV_S$ such that if $S\rightarrow k$ is a map to a field and $v\in \intbigV(k)$ then $v$ is regular if and only if $v\in \intbigV^{\reg}(k)$ and $v$ is regular semisimple if and only if $v\in \intbigV^{\rs}(k)$. Moreover, $\intbigV^{\rs}$ is the open subscheme defined by the nonvanishing of the discriminant polynomial $\Delta$ in $\intbigV_S$. 
	\item The action map $\intbigG_S \times \intbigB_S \rightarrow \intbigV^{\reg}, (g,b) \mapsto g\cdot \sigma(b)$ is \'etale and its image contains $\intbigV^{\rs}$.
	\item Let $\intJac \rightarrow \intbigB_S^{\rs}$ denote the relative Jacobian of $\intbigcurve_S \rightarrow \intbigB_S^{\rs}$. Then there is an isomorphism $\intJac[2] \simeq \Lambda/2\Lambda$ of \'etale sheaves on $\intbigB_S^{\rs}$ whose restriction to $\bigB^{\rs}$ is the isomorphism of Proposition \ref{proposition: bridge jacobians root lattices}. It intertwines the natural pairings on both sides. 
	\item The group schemes $\bigHeis$ and $\bigExt$ over $\bigB^{\rs}$ have natural extensions to finite \'etale group schemes over $\intbigB^{\rs}_S$ and there exists an isomorphism between the two extending the isomorphism from Proposition \ref{propostion: 2 central extensions coincide}. 
	
\end{enumerate}

\begin{proposition}
	There exists a good integer $N$. 
\end{proposition}
\begin{proof}
	This follows from the principle of spreading out. It suffices to consider each property in the above list separately. As an example we will treat Properties $3$, $4$ and $5$ in more detail, leaving the others to the reader. 
	
	For Property $3$, we first choose an $N$ such that $\intbigcurve_S \rightarrow \intbigB_S$ is flat and proper. By \cite[Théorème 12.2.1(x)]{EGAIV-3} the locus where the fibres are geometrically integral is an open subscheme of $\intbigB_S$. 
	Since the fibres of $\bigprojcurve\rightarrow \bigB$ are geometrically integral (use the contracting $\G_m$-action and the geometric integrality of the central fibre), this subscheme equals $\bigB$ over $\Q$.
	By spreading out, we can enlarge $N$ so that this subscheme is the whole of $\intbigB_S$. 
	Moreover the locus of $\intbigB_S$ above which the morphism $\intbigcurve_S\rightarrow \intbigB_S$ is smooth is an open subscheme which coincides with the open subscheme $\intbigB_S^{\rs}$ after base change to $\Q$ by Part 2 of Proposition \ref{proposition: bridge jacobians root lattices}. Again by spreading out, we can enlarge $N$ such that these two open subschemes coincide over $S$. 
	
	For Property $4$, note that $\Z[1/2][\intbigV]^{\intbigG}$ is a finitely generated $\Z[1/2]$-algebra by \cite[Theorem 2]{Seshadri-GeometricReductivityArbitaryBase} and the fact that $\intbigG$ is reductive over $\Z[1/2]$. Moreover it contains the subring $\Z[1/2][p_2,\dots,p_{12}]$. Since this inclusion of finitely generated $\Z[1/2]$-algebras is an equality after tensoring with $\Q$, the same holds after tensoring with $\Z[1/N]$ for some even $N$. The claim about the Kostant section follows from considering the denominators of the morphism $\sigma\colon \bigB \rightarrow \bigV$ and spreading out. 
	
	Finally we consider Property $5$. We will construct open subschemes $\intlieh^{\rs}_S\subset \intlieh^{\reg}_S\subset \intlieh_S$ with similar properties; the subschemes $\intbigV^{\rs}\subset \intbigV^{\reg}\subset \intbigV_S$ will be obtained by restricting them to $\intbigV_S$. 
	Let $Z\rightarrow \intlieh$ be the universal centralizer of the adjoint action of $\intbigH$ on $\intlieh$, so $Z = Z_{\intbigH}(\Id_{\intlieh})$. 
	If $k$ is any field and $x\in \intlieh(k)$ then by definition $x$ is regular if and only if the dimension of $Z_x$ equals $\rk \bigH =6$. 
	By \cite[Théorème 13.1.3]{EGAIV-3} and the fact that the dimension of a group scheme can be computed at the identity, the function $x\mapsto \dim Z_x$ is upper-semicontinuous on $\intlieh$. So the locus $\intlieh^{\reg}$ where the fibre has dimension $6$ is an open subscheme of $\intlieh$. 
	Let $Z^{\reg} \rightarrow \intlieh^{\reg}$ be the restriction of $Z$ to $\intlieh^{\reg}$. 
	By \cite[Remark 4.4.2]{Riche-KostantSectionUniversalCentralizer}, the morphism $Z^{\reg}_S \rightarrow \intlieh^{\reg}_S$ is smooth for some $N$. In that case the locus $\intlieh^{\rs}_S$ where the fibres are tori is an open subscheme of $\intlieh^{\reg}_S$ \cite[Exposé X; Corollaire 4.9]{SGA3-TomeII}, as required. 
	The statement about the discriminant locus follows from spreading out.

\end{proof}

We henceforth fix a good integer $N$ for the remainder of this paper. 
We can then extend our previous results to $S$-algebras rather than $\Q$-algebras. 
We mention in particular:

\begin{proposition}\label{proposition: inject 2-descent orbits spreading out}
	
Let $R$ be an $S$-algebra and $b\in \intbigB^{\rs}(R)$. Suppose that every locally free $R$-module of constant rank is free. Then there is an injective map 
	$$\eta_b \colon \intJac_b(R)/2\intJac_b(R) \rightarrow \intbigG(R)\backslash \intbigV_b(R)$$
	which is compatible with base change on $R$. Moreover it sends the identity element of $\intJac_b(R)/2\intJac_b(R)$ to the orbit of $\sigma(b)$. 
\end{proposition}

We are now ready to state the main theorem of this section whose proof will be given at the end of \S\ref{subsection: proof of theorem integral representatives}. 
Write $\sh{E}_p$ for the set of all $b\in \intbigB(\Z_p)$ which lie in $\bigB^{\rs}(\Q_p)$. It consists of those $b\in \intbigB(\Z_p)$ with nonzero discriminant.

\begin{theorem}\label{theorem: integral representatives exist}
	Let $p$ be a prime not dividing $N$. Then for any $b\in \sh{E}_p$ the image of the map 
	$$\Jac_b(\Q_p)/2\Jac_b(\Q_p) \rightarrow \bigG(\Q_p)\backslash \bigV_b(\Q_p)$$
	from Theorem \ref{theorem: inject 2-descent into orbits} is contained in the image of the map $\intbigV(\Z_p) \rightarrow \bigG(\Q_p)\backslash \bigV(\Q_p)$. 
\end{theorem}

\subsection{Some groupoids}\label{subsection: some groupoids}

In this section we follow \cite[\S4.3]{Thorne-Romano-E8} and define some groupoids which will be a convenient way to think about orbits in our representation and a crucial ingredient for the proof of Theorem \ref{theorem: integral representatives exist}.
Throughout this section we fix a scheme $X$ over $S =  \Z[1/N]$. 

Before we define the groupoids we need to define the outer isomorphism scheme, a technical complication which arises because the group $\intbigH$ has outer automorphisms.
Let $H', H''$ be reductive group schemes over $X$ whose geometric fibres are adjoint semisimple of Dynkin type $E_6$. (See \cite[Definition 3.1.1]{Conrad-reductivegroupschemes} for the definition of a reductive group scheme over a general base.)
Since $H'$ is \'etale locally isomorphic to $H''$, the scheme of isomorphisms of reductive $X$-groups $\Isom_X(H',H'')$ is an $\Aut_X(H')$-torsor.
Define $\Out_X(H',H'')$ as the push-out of this torsor under the map $\Aut_X(H')\rightarrow \Out_X(H')$. It is the quotient of $\Isom_X(H',H'')$ by the action of $H'$.
Since $\Out_X(H')$ is a finite \'etale group scheme of order $2$ \cite[Theorem 7.1.9(2)]{Conrad-reductivegroupschemes}, the $X$-scheme $\Out_X(H',H'')$ is finite \'etale of order $2$ as well.

We define the groupoid $\GrLie_X$ whose objects are triples $(H',\chi',\theta')$ where 
\begin{itemize}
	\item $H'$ is a reductive group scheme over $X$ whose geometric fibres are adjoint semisimple of Dynkin type $E_6$.
	\item $\chi'$ is a section of the $X$-scheme $\Out_X(\intbigH_X,H')$.
	\item $\theta'\colon H' \rightarrow H'$ is an involution of reductive $X$-group schemes such that for each geometric point $\bar{x}$ of $X$ there exists a maximal torus $A_{\bar{x}}$ of $H'_{\bar{x}}$ such that $\theta'$ acts as $-1$ on $X^*(A_{\bar{x}})$. 
\end{itemize}
A morphism $(H',\chi',\theta') \rightarrow (H'',\chi'',\theta'')$ in $\GrLie_X$ is given by an isomorphism $\phi \colon H'\rightarrow H''$ such that $\phi\circ \chi'=\chi''$ and $\phi \circ \theta' = \theta'' \circ \phi$. The triple $(\intbigH_S,[\Id_{\intbigH_S}],\bigtheta_S)$ of \S\ref{subsection: integral structures} defines an object of $\GrLie_S$ by \cite[Corollary 14]{GrossLevyReederYu-GradingsPosRank}.
We note that there is a natural notion of base change and the groupoids $\GrLie_X$ form a stack over the category of schemes over $S$ in the \'etale topology. 

\begin{proposition}\label{proposition: G-torsors in terms of groupoids}
	Let $X$ be an $S$-scheme. 
	The assignment $(H',\chi',\theta')\mapsto \Isom((\intbigH_X,[\Id],\bigtheta_X),(H',\chi',\theta'))$ defines a bijection between: 
	\begin{itemize}
		\item The isomorphism classes of objects in $\GrLie_X$. 
		\item The set $\HH^1(X,\intbigG)$.
	\end{itemize}
\end{proposition}
\begin{proof}
	We first prove that every two triples $(H',\chi',\theta'), (H'',\chi'',\theta'')$ in $\GrLie_X$ are \'etale locally isomorphic.
	The proof of this fact given below is very similar to the proof of \cite[Lemma 2.3]{Thorne-Romano-E8}; we reproduce it here for convenience.
	
	The question being \'etale local on $X$, we may assume that $H' = H''$ and $\chi'=\chi''$. 
    Let $T$ denote the $X$-scheme of elements $h\in H'$ such that $\Ad(h)\circ \theta' = \theta''$; it is a closed subscheme of $H'$ that is $X$-smooth by \cite[Proposition 2.1.2]{Conrad-reductivegroupschemes}.
    Since smooth surjective morphisms have sections \'etale locally, it suffices to prove that $T\rightarrow X$ is surjective.
    Since the construction of $T$ is compatible with base change we may assume that $X=\Spec k$ where $k$ is an algebraically closed field. 

    By assumption, there exist maximal tori $A',A''\subset H'$ on which $\theta', \theta''\in H'(k)$ act through $-1$.
    Using the conjugacy of maximal tori, we may assume that $A'=A''$ so $\theta' = a\cdot \theta''$ for some $a\in A'(k)$. 
    Writing $a=b^2$ for some $b\in A'(k)$
    , we see that $\theta' = b\cdot b\cdot \theta'' = b\cdot \theta'' \cdot b^{-1}$.
    Therefore $\theta'$ is $H'(k)$-conjugate (even $A'(k)$-conjugate) to $\theta''$, as desired.
 
	We now claim that $\Aut((\intbigH_S,[\Id],\bigtheta_S)) = \intbigG_S$, which would prove the proposition by \'etale descent. 
	Indeed, $\Aut((\intbigH_S,[\Id],\bigtheta_S))$ consists of inner automorphisms of $\intbigH_S$ commuting with $\bigtheta$. 
	Since $\intbigH_S$ is adjoint, these are precisely the elements of $\intbigG_S$.
	
\end{proof}


We define the groupoid $\GrLieE_X$ whose objects are $4$-tuples $(H',\chi',\theta',\gamma')$ where $(H',\chi',\theta')$ is an object of $\GrLie_X$ and $\gamma'\in \lieh'$ (the Lie algebra of $H'$) satisfying $\theta'(\gamma') = -\gamma'$. 
A morphism $(H',\chi',\theta',\gamma')\rightarrow (H'',\chi'',\theta'',\gamma'')$ in $\GrLieE_X$ is given by an isomorphism $\phi \colon H' \rightarrow H''$ defining a morphism in $\GrLie_X$ and mapping $\gamma'$ to $\gamma''$. 

We define a map $\GrLieE_X \rightarrow \intbigB(X)$ (where $\intbigB(X)$ is seen as a discrete category) as follows. For an object $(H',\chi',\theta',\gamma')$ in $\GrLieE_X$, choose a faitfully flat extension $X'\rightarrow X$ such that there exists an isomorphism $\phi\colon (H',\chi',\theta')_{X'} \rightarrow (\intbigH_X,[\Id],\bigtheta_X)_{X'}$ in $\GrLie_{X'}$. We define the image of the object $(H',\chi',\theta',\gamma')$ under the map $\GrLieE_X \rightarrow \intbigB(X)$ by $\bigpi(\phi(\gamma'))$. This procedure is independent of the choice of $\phi$ and $X'$ and by flat descent defines an element of $\intbigB(X)$. 
For $b\in \intbigB(X)$ we write $\GrLieE_{X,b}$ for the full subcategory of elements of $\GrLieE_{X,b}$ mapping to $b$ under this map. 

Recall that for $b\in \intbigB(X)$, $\intbigV_b$ denotes the fibre of $b$ of the map $\bigpi\colon \intbigV \rightarrow \intbigB$. 

\begin{proposition}\label{proposition: H1 of stabilizer and GrLieE}
	Let $X$ be an $S$-scheme and let $b \in \intbigB_S(X)$. Then the assignment $$\mathcal{A} \mapsto \Isom((\intbigH_X,[\Id],\bigtheta_X,\sigma(b)),\mathcal{A})$$ defines a bijection between:
	\begin{itemize}
		\item Isomorphism classes of objects in $\GrLieE_{X,b}$ that are \'etale locally isomorphic to $(\intbigH_X,[\Id],\bigtheta_X,\sigma(b))$.
		\item The set $\HH^1(X,Z_{\intbigG_S}(\sigma(b)))$.
	\end{itemize}
	If $b \in \intbigB_S^{\rs}(X)$, then every object of $\GrLieE_{X,b}$ is \'etale locally isomorphic to $(\intbigH_X,[\Id],\bigtheta_X,\sigma(b))$.
\end{proposition}
\begin{proof}
	Since $\Aut((\intbigH_S,[\Id],\bigtheta_S,\sigma(b))) = Z_{\intbigG_S}(\sigma(b))$, the first statement follows from \'etale descent.
	It suffices to prove that every object $(H',\chi',\theta',\gamma')$ is \'etale locally isomorphic to $(\intbigH_X,[\Id],\bigtheta_X,\sigma(b))$ if $b\in \intbigB_S^{\rs}(X)$.
	We may reduce to the case that $(H',\chi',\theta')=(\intbigH_X,[\Id],\bigtheta_X)$ by Proposition \ref{proposition: G-torsors in terms of groupoids}.
	By Property 6 of \S\ref{subsection: integral structures} (which is a spreading out of Proposition \ref{proposition: Kostant section E6} over $S$), the action map $\intbigG_X \times \intbigB^{\rs}_X \rightarrow \intbigV^{\rs}_X$ is \'etale and surjective. 
	Therefore it has sections \'etale locally, hence $\gamma'$ is \'etale locally $\intbigG$-conjugate to $\sigma(b)$. 
	
\end{proof}

The following proposition gives an interpretation of the (not necessarily regular semisimple) $G$-orbits of $V$ in terms  of the groupoids $\GrLieE_X$ and $\GrLie_X$. 

\begin{proposition}\label{proposition: G-orbits in terms of groupoids}
	Let $X$ be an $S$-scheme and let $b\in \intbigB(X)$. 
	The following sets are in canonical bijection:
	\begin{itemize}
		\item The set of $\intbigG(X)$-orbits on $\intbigV_{b}(X)$. 
		\item Isomorphism classes of objects $(H',\chi',\theta',\gamma')$ in $\GrLieE_{X,b}$ such that $(H',\chi',\theta') \simeq (\intbigH_S,\chi,\bigtheta_S)_X$ in $\GrLie_X$. 
	\end{itemize}
	Consequently if $b\in \intbigB_S^{\rs}(X)$, then the following sets are in canonical bijection:
	\begin{itemize}
		\item The set of $\intbigG(X)$-orbits on $\intbigV_{b}(X)$. 
		\item The kernel of the map $\HH^1(X,Z_{\intbigG_S}(\sigma(b)))\rightarrow \HH^1(X,\intbigG)$. 
	\end{itemize}
\end{proposition}
\begin{proof}
	For the first part, we construct an explicit bijection between these two sets. 
	If $v\in \intbigV_{b}(X)$ is a representative of a $\intbigG(X)$-orbit, we associate to $v$ the object $ (\intbigH_X,[\Id],\bigtheta_X,v)$ of $\GrLieE_{X,b}$. Changing $v$ by a $\intbigG(X)$-conjugate does not change the isomorphism class of this object, so this association is well-defined. 
	Conversely, if $(H',\chi',\theta',\gamma')$ is an object of $\GrLieE_{X,b}$ and $\phi \colon (H',\chi',\theta') \rightarrow (\intbigH_S,[\Id],\bigtheta_S)_X$ an isomorphism in $\GrLie_X$, we associate to it the element $v=\phi(\gamma') \in \intbigV_{b}(X)$. Changing the isomorphism $\phi$ does not change the $\intbigG(X)$-conjugacy class of $v$. 
	
 The second part follows from combining the first part with Propositions \ref{proposition: G-torsors in terms of groupoids} and \ref{proposition: H1 of stabilizer and GrLieE}.
\end{proof}

\begin{remark}
    The groupoids $\GrLie_X$ and $\GrLieE_X$ for varying $X$ are stacks in the \'etale topology over $S$, and one can show that $\GrLie \simeq \left[ S/\intbigG_S \right] $ and $\GrLieE \simeq \left[ \intbigV_S/\intbigG_S \right]$.
    We will not need these facts in what follows.
\end{remark}

\subsection{The compactified Jacobian}\label{subsection: compactifications}

Recall that $\intJac\rightarrow \intbigB_S^{\rs}$ denotes the relative Jacobian of the family of smooth curves $\intbigcurve^{\rs}_S \rightarrow \intbigB_S^{\rs}$. The morphism $\intJac \rightarrow \intbigB_S^{\rs}$ is proper and smooth.
In this section we introduce a compactification of this abelian scheme over $\intbigB_S$. The reader not interested in the details of the construction can simply admit its properties which are summarized in Corollary \ref{corollary: good compactifications exist}. 

We start with some generalities on torsion-free rank $1$ sheaves. By a \define{curve} over a field $k$ we mean a finite type scheme over $k$ such that every irreducible component has dimension $1$.
\begin{definition}
	Let $X$ be an integral projective curve over an algebraically closed field $k$. We say a coherent sheaf $I$ on $X$ is \define{torsion-free rank $1$} if it satisfies the following two conditions:
	\begin{enumerate}
		\item For each $p \in X$ the $\O_{X,p}$-module $I_p$ is torsion-free. 
		\item If $\eta \in X$ is the generic point then we have an isomorphism $I_{\eta} \simeq \O_{X,\eta}$ of $\O_{X,\eta}$-modules.
	\end{enumerate}
\end{definition}
If $X$ is smooth then every torsion-free rank $1$ sheaf is invertible, but for non-smooth $X$ this need not to be the case. 
For example, if $X$ is the projective closure of the plane curve $(y^2 = x^3)$ then the ideal sheaf of the origin is a torsion-free rank $1$ sheaf which is not invertible.
 
The above definition can be generalized to a family of curves.
\begin{definition}
	Let $\mathcal{X} \rightarrow T$ be a flat projective morphism whose geometric fibres are integral curves. A locally finitely presented $\O_{\mathcal{X}}$-module $I$ is \define{$T$-relatively torsion-free rank $1$} if the following conditions are satisfied:
	\begin{enumerate}
		\item The sheaf $I$ is flat over $T$.
		\item For every geometric point $t$ of $T$ the sheaf $I_t$ is torsion-free rank $1$ on $\mathcal{X}_t$.
	\end{enumerate}
\end{definition}

We apply the above definitions to our situation of interest.
The morphism $\intbigcurve_S \rightarrow \intbigB_S$ is flat, projective and its geometric fibres are integral curves. 
The Euler characteristic of the structure sheaf of the geometric fibres is constant, equal to $1-3 = -2$. 
The point at infinity defines a section $P_{\infty}\colon \intbigB_S \rightarrow \intbigcurve_S$ whose image lands in the smooth locus of the morphism.
Let $F$ be the functor sending a $\intbigB_S$-scheme $T$ to the set 
	\begin{align*}
	\left\{(I,\phi) \mid I \text{ is } T\text{-relatively torsion-free rank }1 \text{ on } \intbigcurve_T \rightarrow T ,\, \phi\colon \left(P_{\infty,T}\right)^*I\simeq \O_T    \right\}	/\simeq .
	\end{align*}
	Here we require $\phi$ to be an isomorphism of $\O_T$-modules, and we say two pairs $(I,\phi)$ and $(I',\phi')$ are isomorphic if there exists an isomorphism of $\O_{\intbigcurve_T}$-modules $I\simeq I'$ identifying $\phi$ with $\phi'$.
	 Let $F^0$ be the subfunctor of $F$ consisting of those torsion-free rank $1$ sheaves with Euler-characteristic $-2$ in each fibre. 
	 Altman and Kleiman \cite[Theorem 8.1]{AltmanKleiman-CompactifyingThePicardScheme} have shown that $F^0$ is representable. 
	 \begin{definition}
	 We call the scheme $\CJac \rightarrow \intbigB_S$ representing the functor $F^0$ the \define{compactified Jacobian} of the family $\intbigcurve_S \rightarrow \intbigB_S$. 
	 \end{definition}
	 
By \cite[Theorem 8.5]{AltmanKleiman-CompactifyingThePicardScheme} the morphism $\CJac \rightarrow \intbigB_S$ is projective\footnote{There are several nonequivalent definitions of a projective morphism but in this case they all agree, see \cite[Tag \href{https://stacks.math.columbia.edu/tag/0B45}{0B45}]{stacksproject}.}. 
Moreover since every torsion-free rank $1$ sheaf on a smooth curve is invertible, the restriction of $\CJac$ to $\intbigB_S^{\rs}$ is isomorphic to $\intJac$.  
The fibres of $\intbigcurve_S \rightarrow \intbigB_S$ have only planar singularities; we may therefore appeal to \cite[Theorem 9]{AltmanKleimanSteven-IrreducibilityCompactifiedJacobian} to obtain the following good properties of $\CJac$:

\begin{proposition}\label{proposition: compactified Jacobian props AIK}
    The morphism $\CJac \rightarrow \intbigB_S$ is flat and its geometric fibres are integral of dimension $3$.
\end{proposition}


The crucial additional property of $\CJac$, which follows from the fact that $\bigprojcurve\rightarrow \bigB$ is a semi-universal deformation of its central fibre, is the following. 

\begin{proposition}\label{proposition: smoothness compactified jacobian}
    For every geometric point $\Spec k \rightarrow \Spec S = \Spec \Z[1/N]$, the scheme $\CJac_k$ is smooth.  
\end{proposition}
\begin{proof}
    By \cite[Corollary B.2]{FantechiGottschevStraten-EulerNumberCompactifiedJacobian}, $\CJac_k$ is smooth in a neighbourhood of the fibre above $0 \in \intbigB_k$. (In loc. cit. it is assumed that the characteristic of the base field is $0$ but the proof given works for any algebraically closed field of characteristic not dividing $N$.)
    To see that $\CJac_k$ is smooth everywhere, we use the contracting $\G_m$-action. Recall that we have defined a $\G_{m,k}$-action on $\intbigcurve_k \rightarrow \intbigB_k$ in \S\ref{subsection: a family of curves}. By functoriality this induces a $\G_{m,k}$-action on $\CJac_k$ such that the morphism $\CJac_k \rightarrow \intbigB_k$ is $\G_{m,k}$-equivariant. 
    If $Z$ is the singular locus of $\CJac_k$ then $Z$ is a closed subscheme which is invariant under the action of $\G_{m,k}$. Since the closure of every orbit of $\intbigB_k$ contains $0\in \intbigB_k$, this subscheme must intersect the fibre above $0\in \intbigB_k$ nontrivially, if it is nonempty. We conclude that $Z$ is empty and $\CJac_k$ is smooth, as required.
\end{proof}

\begin{remark} 
Although the total space $\CJac_k $ is smooth, the morphism $\CJac_k\rightarrow \intbigB_k$ will not be smooth over points which do not lie in $\intbigB_k^{\rs}$. 	
\end{remark}

For later reference, we summarize the relevant properties of $\CJac$ in the following corollary.

\begin{corollary}\label{corollary: good compactifications exist}
	The morphism $\CJac \rightarrow \intbigB_S$ constructed above is flat and projective and its restriction to $\intbigB^{\rs}_S \subset \intbigB_S$ is isomorphic to $\intJac \rightarrow \intbigB^{\rs}_S$. The morphism $\CJac \rightarrow \Spec S$ is smooth with geometrically integral fibres. For every geometric point $\Spec k \rightarrow \Spec S$, $\intJac_k$ is dense in $\CJac_k$ and the locus of $\CJac_k$ where the morphism $\CJac_k \rightarrow \intbigB_k$ is smooth is an open subset whose complement has codimension at least two in $\CJac_k$. 
	    
\end{corollary}
\begin{proof}
    The first sentence follows from Proposition \ref{proposition: compactified Jacobian props AIK} and the definition of $\CJac \rightarrow \intbigB_S$. 
    The smoothness of $\CJac \rightarrow \Spec S$ follows from Proposition \ref{proposition: smoothness compactified jacobian} and the flatness of $\CJac \rightarrow \Spec S$. The integrality of the geometric fibres of $\CJac\rightarrow \Spec S$ follows from the smoothness of $\CJac \rightarrow \Spec S$, the irreducibility of the fibres of $\CJac \rightarrow \intbigB_S$ and Lemma \ref{lemma: irreducibility fibres} below. 
    Moreover since $\intJac_k$ and $\CJac_k$ are both irreducible of the same dimension, $\intJac_k$ is dense in $\CJac_k$.
    Finally we prove the claim about the smooth locus of the morphism $\CJac_k \rightarrow \intbigB_k$; for the remainder of the proof we denote this morphism by $\phi$.
    Let $Z \subset \CJac_k$ denote the (reduced) closed subscheme where $\phi$ fails to be smooth. 
	The smoothness of $\intJac_k \rightarrow \intbigB^{\rs}_k$ shows that $Z$ is supported above the complement of $\intbigB^{\rs}_k$ in $\intbigB_k$. Moreover since the fibres of $\phi$ are geometrically integral they intersect $Z$ in a proper closed subset. Combining these two facts shows that $Z$ has codimension at least two in $\CJac_k$. 
\end{proof}

\begin{lemma}\label{lemma: irreducibility fibres}
	Let $f\colon X\rightarrow Y$ be a flat morphism of schemes which is locally of finite presentation. Suppose that $Y$ and the fibres of $f$ are irreducible. Then $X$ is irreducible. 
\end{lemma}
\begin{proof}
	Since $f$ is open, this follows from \cite[Tag \href{https://stacks.math.columbia.edu/tag/004Z}{004Z}]{stacksproject}. 
\end{proof}


\subsection{The case of square-free discriminant}\label{subsection: case of square-free discriminant}

In this section we prove Theorem \ref{theorem: integral representatives exist} in the case of square-free discriminant. We follow the homonymous section \cite[\S5.1]{Thorne-Romano-E8} closely.
We start with some preparatory lemmas.
The first two lemmas are very similar to \cite[Lemma 5.2 and 5.3]{Thorne-Romano-E8}; their proofs will be omitted.

\begin{lemma}\label{lemma: trivial kernel of H1(R,G)->H1(K,G)}
	Let $R$ be a Noetherian regular integral domain with fraction field $K$ such that every locally free $R$-module of finite rank is free. Then the map $\HH^1(R,\intbigG) \rightarrow \HH^1(K,\intbigG)$ has trivial kernel.
\end{lemma}

\begin{lemma}\label{lemma: injective H^1 for quasifinite etale gp scheme}
	Let $X$ be a Dedekind scheme (i.e. a regular integral one-dimensional noetherian scheme) with function field $K$. Let $\Gamma$ be a quasi-finite \'etale commutative group scheme over $X$. Suppose that $\Gamma$ is a N\'eron model of its generic fibre: for every \'etale morphism $U\rightarrow X$ with $U$ a Dedekind scheme with function field $K(U)$, we have $\Gamma(U) = \Gamma(K(U))$. 
	Then the map $\HH^1(X,\Gamma) \rightarrow \HH^1(K,\Gamma)$ is injective.
\end{lemma}

The following lemma is a special case of a result proven by Poonen and Stoll concerning hypersurfaces of arbitrary degree and dimension. 

\begin{lemma}\label{lemma: sqfree disc implies regular node}

	Let $R$ be a discrete valuation ring in which $N$ is a unit. Let $K = \Frac R$ and let $\ord_K: K^{\times} \twoheadrightarrow \Z$ be the normalized discrete valuation. 
	Let $b\in \intbigB(R)$ and suppose that $\ord_K \Delta(b) = 1$. Then $\intbigcurve_b$ is regular and its special fibre contains a unique singularity, which is a node. 
\end{lemma}

\begin{proof}
	Recall from \S\ref{subsection: discriminant polynomial} that $\Delta_0\in \Z[\intbigB]$ denotes the (divided) discriminant of a plane quartic curve. (It was originally defined as an element of $\Q[\bigB]$ but by the same formula it defines an element of $\Z[\intbigB]$.)
	Proposition \ref{proposition: discriminant Delta and Delta0 agree} and our assumptions on $N$ imply that $\Delta(b)$ and $\Delta_0(b)$ agree up to an element of $\Z[1/N]^{\times}$. So $\ord_K\Delta_0(b)=1$. 
	The lemma now follows from the main result of \cite{PoonenStoll-Hypersurfacesdiscriminantuniformizer}.
	 
\end{proof}

\begin{lemma}\label{lemma: squarefree disc properties of elements of V}
    Let $R$ be a discrete valuation ring with residue field $k$ in which $N$ is a unit.
    Let $\bar{k}$ be an algebraic closure of $k$.
    Let $K = \Frac R$ and let $\ord_K: K^{\times} \twoheadrightarrow \Z$ be the normalized discrete valuation.
    Let $x\in \intbigV(R)$ with $b=\pi(x)\in \intbigB(R)$ and suppose that $\ord_K \Delta(b)=1$.
    Then the reduction $x_k$ of $x$ in $\intbigV(k)$ is regular and $\intbigG(\bar{k})$-conjugate to $\sigma(b)_k$.
    In addition the $R$-group scheme $Z_{\intbigG}(x)$ is quasi-finite \'etale and has special fibre of order $2^5$.
\end{lemma}
\begin{proof}
    We are free to replace $R$ by a discrete valuation ring $R'$ containing $R$ such that any uniformizer in $R$ is also a uniformizer in $R'$.
    Therefore we may assume that $R$ is complete and $k$ algebraically closed.
    
    Let $x_k=y_s+y_n$ be the Jordan decomposition of $x_k\in \intbigV(k)$ as a sum of its semisimple and nilpotent parts. 
    Let $\intlieh_{0,k}=\mathfrak{z}_{\intlieh}(y_s)$ and $\intlieh_{1,k}=\image(\Ad(y_s))$.
    Then $\intlieh_{k}=\intlieh_{0,k}\oplus \intlieh_{1,k}$, where $\Ad(x_k)$ acts nilpotently on $\intlieh_{0,k}$ and invertibly on $\intlieh_{1,k}$.
    By Hensel's lemma, this decomposition lifts to an $\Ad(x)$-invariant decomposition of free $R$-modules $\intlieh_R = \intlieh_{0,R}\oplus \intlieh_{1,R}$, where $\Ad(x)$ acts topologically nilpotently on $\intlieh_{0,R}$ and invertibly on $\intlieh_{1,R}$.
    We claim that there exists a unique closed subgroup $L\subset \intbigH_R$ with Lie algebra $\intlieh_{0,R}$ such that $L$ is $R$-smooth with connected fibres.
    The uniqueness follows from \cite[Exp. XIV, Proposition 3.12]{SchemasenGroupesII}.
    To show existence, choose a regular semisimple element $\bar{r}$ of the reductive Lie algebra $\mathfrak{z}_{\intlieh}(y_s)$ and an arbitrary lift $r\in \intlieh_{0,R}$.
    The centralizer $\mathfrak{z}_{\intlieh}(r)$ is a Cartan subalgebra of $\intlieh_R$ whose reduction mod $k$ contains $y_s$.
    Since $k=k^s$, the algebra $\mathfrak{z}_{\intlieh}(r)$ is split so there exists an element $y_{s,R}\in \mathfrak{z}_{\intlieh}(r)$ lifting $y_s$ such that $\mathfrak{z}_{\intlieh}(y_{s,R})=\intlieh_{0,R}$.
    Then $L= Z_{\intbigH}(y_{s,R})$ is $R$-smooth, has Lie algebra $\intlieh_{0,R}$, and has connected fibres by \cite[Theorem 3.14]{Steinberg-Torsioninreductivegroups}.
    The construction shows that $L_k=Z_{\intbigH}(y_s)$.
    
    Lemma \ref{lemma: sqfree disc implies regular node} shows that the curve $\intbigcurve_{b,k}$ has a unique nodal singularity.
    Therefore\footnote{The proof of that corollary only depends on \cite[\S6.6]{Slodowy-simplesingularitiesalggroups} so is valid in any characteristic which is very good for $\lieh$, i.e. different from $2,3,5$; see the remark at the end of \cite[\S6.6]{Slodowy-simplesingularitiesalggroups}} \cite[Corollary 3.16]{Thorne-thesis} the derived group of $L$ has type $A_1$ and the centre $Z(L)$ of $L$ has rank $5$.
    Moreover the restriction $\theta_L$ of $\theta$ to $L$ is a stable involution, in the sense that for each geometric point of $\Spec R$ there exists a maximal torus of $L$ on which $\theta$ acts as $-1$, by \cite[Lemma 2.5]{Thorne-thesis}.
    There is an isomorphism $L/Z(L)\simeq \PGL_2$ inducing an isomorphism $\intlieh_{R,0}^{\der}\simeq \intlieh_{R,0}/\mathfrak{z}(\intlieh_{R,0}) \simeq \liesl_{2,R}$ under which $\theta_L$ corresponds to the involution $\xi = \Ad\left(\text{diag}(1,-1) \right)$. (The isomorphism $\intlieh_{R,0}^{\der}\simeq \intlieh_{R,0}/\mathfrak{z}(\intlieh_{R,0})$ exists by our assumptions on the residue characteristic of $N$, and by the same logic as the proof of Lemma \ref{proposition: G-torsors in terms of groupoids} any two stable involutions on $\liesl_{2,R} $ are \'etale locally conjugate.)
    The claims in the lemma now follow easily from explicit calculations in $\liesl_{2,R}$. 
    
    Indeed, to show that $x_k$ is regular it suffices to show that $y_n$ is regular nilpotent in $\mathfrak{z}_{\intlieh}(y_s)=\intlieh_{0,k}$.
    Let $x'$ denote the projection of $x$ in $\lieh_{0,R}^{\der}$. 
    The image of $x'$ under the isomorphism $\intlieh_{R,0}^{\der}\rightarrow \liesl_{2,R}$ corresponds to an element of the form $$\begin{pmatrix} 0 & a \\ b & 0 \end{pmatrix}$$ with $\ord_K(ab)=1$.
    Therefore the reduction of $x'$ in $\liesl_{2,k}$ is regular nilpotent, as desired.
    
    We show that $x_k$ is $\intbigG(k)$-conjugate to $\sigma(b)_k$. 
    By \cite[Corollary 2.6 and Theorem 2.20]{Levy-Vinbergtheoryposchar} (which extends Vinberg theory to good positive characteristic), the semisimple parts of $x_k$ and $\sigma(b)_k$ are $\intbigG(k)$-conjugate.
    Moreover both $x_k$ and $\sigma(b)_k$ are regular.
    Therefore it suffices to prove that $L^{\theta}(k)$ acts transitively on the regular nilpotent elements of $\intlieh_{k,0}^{\theta=-\Id}$.
    Using the fact that $\intbigH$ is adjoint, the character group of $Z(L)$ is given by the $E_6$ root lattice modulo the span of a root. Therefore $Z(L)$ has connected fibres.
    It follows that the exact sequence 
    $$
    1\rightarrow Z(L_k) \rightarrow L_k \rightarrow \PGL_{2,k} \rightarrow 1
    $$
    induces a surjection $L_k^{\theta} \rightarrow \PGL_{2,k}^{\xi}$.
    Since $\PGL_{2,k}^{\xi}$ acts transitively on the regular nilpotents of $\liesl_{2,k}^{\xi=-\Id}$, the statement for $L_k$ follows.
    
    Finally by \cite[Proposition 2.8]{Thorne-thesis}, $Z_{\intbigG}(x)_k=Z(L_k)[2]$.
    Therefore since $Z(L_k)$ is connected, $Z_{\intbigG}(x)_k$ has order $2^5$.

\end{proof}

In the next proposition, we will use a slight abuse of notation and for any $b\in \intbigB(R)$ with $\Delta(b)\neq 0$ we will write $\intJac_b$ (which is a priori only defined if $\Delta(b) \in R^{\times}$) for the $K$-scheme $\intJac_{b_K}$.

\begin{proposition}\label{prop: integral reps squarefree discr}
	Let $R$ be a discrete valuation ring in which $N$ is a unit. Let $K = \Frac R$ and let $\ord_K: K^{\times} \twoheadrightarrow \Z$ be the normalized discrete valuation. 
	Let $b\in \intbigB(R)$ and suppose that $\ord_K \Delta(b)\leq 1$. 
	Then:
	\begin{enumerate}
		\item If $x\in \intbigV_b(R)$, then $Z_{\intbigG}(x)(K) = Z_{\intbigG}(x)(R)$. 
		\item The natural map $\alpha\colon\intbigG(R)\backslash \intbigV_b(R) \rightarrow \intbigG(K)\backslash \intbigV_b(K)$ is injective and its image contains $\eta_b\left(\intJac_b(K)/2\intJac_b(K)\right)$. 
		\item If further $R$ is complete and has finite residue field then the image of $\alpha$ equals $\eta_b\left(\intJac_b(K)/2\intJac_b(K)\right)$. 
	\end{enumerate}
\end{proposition}

The proof is very similar to the proof of \cite[Proposition 5.4]{Thorne-Romano-E8}, where an analogous result for the stable $3$-grading on $E_8$ is proved.

\begin{proof}
    If $\hat{R}$ is the completion of $R$ with fraction field $\hat{K}$, we have the equality $\intbigG(\hat{K})=\intbigG(\hat{R})\intbigG(K)$ \cite[Th\'eor\`eme 3.2]{Nisnevich-Espaceshomogenesprincipaux}. 
    We may therefore assume that $R$ is complete. 
    
    If $\ord_K \Delta(b)=0$, $\intJac_b$ is smooth over $R$ and $\intJac_b(K)=\intJac_b(R)$.
    Since $Z_{\intbigG}(x)$ is finite \'etale over $R$, the first part follows.
    By Proposition \ref{proposition: G-orbits in terms of groupoids} and Lemma \ref{lemma: injective H^1 for quasifinite etale gp scheme}, $\alpha$ is injective.
    Proposition \ref{proposition: inject 2-descent orbits spreading out} implies that $\eta_b\colon \intJac_b(K)/2\intJac_b(K)\rightarrow \intbigG(K)\backslash \intbigV_b(K)$ factors through $\intbigG(R)\backslash \intbigV_b(R)$, so the second part follows.
    If the residue field $k$ is finite, the pointed sets $\HH^1(R, \intbigG)$ and $\HH^1(R,\intJac_b)$ are trivial by \cite[III.3.11(a)]{milne-etalecohomology} and Lang's theorem.
    The third part then follows from the fact that the $2$-descent map $ \intJac_b(R)/2\intJac_b(R)\rightarrow \HH^1(R,\intJac_b[2])$ is an isomorphism.
    
    We now assume that $\ord_K \Delta(b)=1$.
Lemma \ref{lemma: sqfree disc implies regular node} implies that $\intbigcurve_b/R$ is regular and has a unique singularity, which is a node.
Write $\NeronJac_b$ for the N\'eron model of $\intJac_b$.  
The results of \cite[Chapter 9]{BLR-NeronModels} (in particular Theorem 1 of \S9.5 and Example 8 of \S9.2 therein) imply that $\NeronJac_b$ is a smooth group scheme over $R$ with connected fibres and that the special fibre of $\NeronJac_b$ is an extension of a $2$-dimensional abelian variety by a rank $1$ torus. The quasi-finite \'etale commutative group scheme $\NeronJac_b[2]$ has generic fibre of order $2^6$ and special fibre of order $2^5$. 


We claim that the map $\intbigG \rightarrow \intbigV_b^{\reg},\, g \mapsto g\cdot \sigma(b)$ is a torsor for the \'etale group scheme $Z_{\intbigG}(\sigma(b))$. 
Since this map is smooth (Property 6 of \S\ref{subsection: integral structures}) and surjective in the generic fibre (Proposition \ref{prop : graded chevalley}), it suffices to prove that it is surjective in the special fibre. 
Since every closed point of $\intbigV_{b,k}^{\reg}$ lifts to an element of $\intbigV_b(R')$ for some finite extension $R \subset R'$ of ramification index $1$, this follows from Lemma \ref{lemma: squarefree disc properties of elements of V} applied to $R'$.

We now prove the first part. Since $x$ is \'etale locally $\intbigG$-conjugate to $\sigma(b)$ over $R$ by the previous paragraph, it suffices to consider the case $x=\sigma(b)$. We show that the isomorphism $Z_{\intbigG}(\sigma(b))_K \simeq \intJac_b[2] $ of (a $\Z[1/N]$-analogue of) Proposition \ref{proposition: bridge jacobians root lattices} extends to an isomorphism $Z_{\intbigG}(\sigma(b)) \simeq \NeronJac_b[2] $. Indeed, by the N\'eron mapping property the former isomorphism extends to an open immersion $Z_{\intbigG}(\sigma(b))\rightarrow \NeronJac_b[2] $ of separated quasi-finite \'etale group schemes over $R$. Since the special fibre of $Z_{\intbigG}(\sigma(b))$ has order $2^5$ by Lemma \ref{lemma: squarefree disc properties of elements of V}, this is an isomorphism.
Part 1 then follows from the equality $\NeronJac_b[2](K) = \NeronJac_b[2](R)$.

To prove the remaining parts, note that the surjectivity of $\intbigG \rightarrow \intbigV_b^{\reg},\, g \mapsto g\cdot \sigma(b)$ implies that (in the notation of \S\ref{subsection: some groupoids}) every object of $\GrLieE_{R,b}$ is \'etale locally isomorphic to $(\intbigH_R,[\Id]_R,\theta_R,\sigma(b))$. By Propositions \ref{proposition: H1 of stabilizer and GrLieE} and \ref{proposition: G-orbits in terms of groupoids}, the $\intbigG(R)$-orbits of $\intbigV_b(R)$ are in bijection with the kernel of the map $\HH^1(R, Z_{\intbigG}(\sigma(b)))\rightarrow \HH^1(R,\intbigG)$. 
Since the map $\HH^1(R, Z_{\intbigG}(\sigma(b)))\rightarrow \HH^1(K, Z_{\intbigG}(\sigma(b)))$ is injective (using the isomorphism $Z_{\intbigG}(\sigma(b)))\simeq \NeronJac_b[2]$ and Lemma \ref{lemma: injective H^1 for quasifinite etale gp scheme}), the map $\intbigG(R)\backslash \intbigV_b(R) \rightarrow \intbigG(K)\backslash \intbigV_b(K)$ is injective too.
To show that the image of $\intbigG(R)\backslash \intbigV_b(R)  \rightarrow \intbigG(K)\backslash \intbigV_b(K)$ contains $\eta_b\left(\intJac_b(K)/2\intJac_b(K)\right)$, note that we have an exact sequence of smooth group schemes
\begin{align*}
0 \rightarrow \NeronJac_b[2]\rightarrow \NeronJac_b \xrightarrow{\times 2} \NeronJac_b \rightarrow 0,
\end{align*}
since $\NeronJac_b$ has connected fibres. This implies the existence of a commutative diagram:
\begin{center}
	\begin{tikzcd}
		 \NeronJac_b(R)/2\NeronJac_b(R) \arrow[d] \arrow[r , "="] & \intJac_b(K)/2\intJac_b(K) \arrow[d] \\
		{\HH^1(R,\NeronJac_b[2])} \arrow[r]      & {\HH^1(K,\intJac_b[2])}                        
	\end{tikzcd}
	\end{center}
It therefore suffices to prove that every element in the image of the map $\NeronJac_b(R)/2\NeronJac_b(R) \rightarrow  \HH^1(R,\NeronJac_b[2])$ has trivial image in $\HH^1(R,\intbigG)$. This follows from the injectivity of the map $\HH^1(R,\intbigG) \rightarrow \HH^1(K,\intbigG)$ (Lemma \ref{lemma: trivial kernel of H1(R,G)->H1(K,G)}).

If $R$ has finite residue field then Lang's theorem implies that $\HH^1(R,\intbigG) = \{1\}$. 
In this case the $\intbigG(R)$-orbits on $\intbigV_b(R)$ are in bijection with $\HH^1(R,\NeronJac_b[2])$.
The triviality of $\HH^1(R,\NeronJac_b)$ (again by Lang's theorem) shows that $\HH^1(R,\NeronJac_b[2])$ is in bijection with $\NeronJac_b(R)/2\NeronJac_b(R) = \intJac_b(K)/2\intJac_b(K)$. 
This proves Part 3, completing the proof of the proposition.

\end{proof}

The following corollary considers arbitrary Dedekind schemes. Since such schemes do not satisfy the conditions of Theorem \ref{theorem: inject 2-descent into orbits} (they can carry nontrivial vector bundles), we must switch our focus from orbits to groupoids, in the language of \S\ref{subsection: some groupoids}.

\begin{corollary}\label{corollary: int reps sqfree, general dedekind scheme}
	Let $X$ be a Dedekind scheme in which $N$ is a unit with function field $K$. 
	For every closed point $p$ of $X$ write $\ord_{p} \colon K^{\times} \twoheadrightarrow \Z$ for the normalized discrete valuation of $p$. 
	Let $b\in \intbigB(X)$ be a morphism such that $\ord_{p}(\Delta(b))\leq 1$ for all $p$. 
	Let $P\in \intJac_b(K)/2\intJac_b(K)$ and let $\eta_b(P)\in G(K) \backslash V_b(K)$ be the corresponding orbit from Theorem \ref{theorem: inject 2-descent into orbits}.
	Then the object of $\GrLieE_{K,b}$ corresponding to $\eta_b(P)$ using Proposition \ref{proposition: G-orbits in terms of groupoids} uniquely extends to an object of $\GrLieE_{X,b}$.

\end{corollary}
\begin{proof}

    By the same logic as the proof of Proposition \ref{prop: integral reps squarefree discr}, the morphism
    $\intbigG \rightarrow \intbigV_b^{\reg},\, g \mapsto g\cdot \sigma(b)$ is a torsor for the \'etale group scheme $Z_{\intbigG}(\sigma(b))$ and the isomorphism $Z_{\intbigG}(\sigma(b))_K \simeq \intJac_b[2] $ of Proposition \ref{proposition: bridge jacobians root lattices} extends to an isomorphism $Z_{\intbigG}(\sigma(b)) \simeq \NeronJac_b[2]$, where $\NeronJac_b \rightarrow X$ denotes the N\'eron model of $\intJac_b$.
	So every object of $\GrLieE_{X,b}$ is \'etale locally isomorphic to $(\intbigH_X,[\Id]_X,\theta_X,\sigma(b))$.
	Therefore by Proposition \ref{proposition: H1 of stabilizer and GrLieE} the set of isomorphism classes of objects in $\GrLieE_{X,b}$ is in bijection with the pointed set $\HH^1(X,\NeronJac_b[2])$. 
	
	Let $\mathcal{A}\in \HH^1(K,\intJac_b[2])$ be the class corresponding to $\eta_b(P)$ under Proposition \ref{proposition: G-orbits in terms of groupoids}. It suffices to prove that $\mathcal{A}$ uniquely lifts under the natural map $\HH^1(X,\NeronJac_b[2]) \rightarrow \HH^1(K,\intJac_b[2])$.
	The equality $\NeronJac_b(X)=\intJac_b(K)$ implies that the $2$-descent map $\intJac_b(K)/2\intJac_b(K)\rightarrow \HH^1(K,\intJac_b[2])$ factors through $\HH^1(X,\NeronJac_b[2])\rightarrow \HH^1(K,\intJac_b[2])$, so $\mathcal{A}$ indeed lifts.
	The uniqueness follows from the injectivity of the map $\HH^1(X,\NeronJac_b[2]) \rightarrow \HH^1(K,\intJac_b[2])$ (Lemma \ref{lemma: injective H^1 for quasifinite etale gp scheme}).

\end{proof}

\subsection{The proof of Theorem \ref{theorem: integral representatives exist}}\label{subsection: proof of theorem integral representatives}

We now treat the general case. We will do this by deforming to the case of square-free discriminant, with the help of the following Bertini type theorem over $\Z_p$.

\begin{proposition}\label{proposition: Bertini type theorem}
	Let $p$ be a prime number. Let $\mathcal{Y} \rightarrow \Z_p$ be a smooth, quasiprojective morphism of relative dimension $d\geq 1$ with geometrically integral fibres. 
	Let $\mathcal{D} \subset \mathcal{Y}$ be an effective Cartier divisor. Assume that $\mathcal{Y}_{\F_p}$ is not contained in $\mathcal{D}$ (i.e. $\mathcal{D}$ is horizontal) and that $\mathcal{D}_{\Q_p}$ is reduced.
	Let $P\in \mathcal{Y}(\Z_p)$ be a section such that $P_{\Q_p} \not\in \mathcal{D}_{\Q_p}$. 
	Then there exists a closed subscheme $\mathcal{X} \hookrightarrow \mathcal{Y}$ containing the image of $P$ satisfying the following properties.
	\begin{itemize}
		\item $\mathcal{X} \rightarrow \Z_p$ is smooth of relative dimension $1$ with geometrically integral fibres.
		\item $\mathcal{X}_{\F_p}$ is not contained in $\mathcal{D}$ and the (scheme-theoretic) intersection $\mathcal{X}_{\Q_p} \cap \mathcal{D}_{\Q_p}$ is reduced.
	
	\end{itemize} 
\end{proposition}

\begin{proof}
If $d=1$ we can take $\mathcal{X} = \mathcal{Y}$ and there is nothing to prove. Thus for the rest of the proof we may assume that $d\geq 2$.
Fix a locally closed embedding $\mathcal{Y} \subset \P_{\Z_p}^n$.
We will induct on $d$ by finding a suitable hypersurface section using Bertini theorems over $\F_p$ and $\Q_p$.
Combining \cite[Theorem 1.2]{Poonen-BertiniTheoremsFiniteFields} and \cite[Theorem 1.1]{CharlesPoonen}, there exists a hypersurface $H$ in $\P^n_{\F_p}$ such that the (scheme-theoretic) intersection $\mathcal{Y}_{\F_p} \cap H$ is smooth, geometrically irreducible of codimension $1$ in $\mathcal{Y}_{\F_p}$, contains the point $P_{\F_p}$ and is not contained in $\mathcal{D}$. 

We will lift this hypersurface to a hypersurface in $\P^n_{\Z_p}$ with similar properties, as follows. Let $M$ be the projective space over $\Q_p$ parametrizing hypersurfaces of degree $\deg H$ in $\P^n_{\Q_p}$ containing the point $P_{\Q_p}$. By the classical Bertini theorem over $\Q_p$, there exists an open dense subscheme $U$ of $M$ such that every hypersurface $H'$ in $U$ has the property that $H' \cap \mathcal{Y}_{\Q_p}$ is smooth, geometrically irreducible of codimension $1$ and that $H'  \cap \mathcal{D}_{\Q_p}$ is reduced. 
The subset of $M(\Q_p)$ whose reduction mod $p$ is the hypersurface $H$ is an open $p$-adic ball of $M(\Q_p)$. Consequently, it intersects $U(\Q_p)$ nontrivially. (Since an open $p$-adic ball in a projective $\Q_p$-space cannot be contained in a proper Zariski closed subscheme.) So there exists a hypersurface $\mathcal{H} \subset \P^n_{\Z_p}$ lifting $H$ such that $\mathcal{H}_{\Q_p} \in U(\Q_p)$.
 
By \cite[Theorem 22.6]{Matsumura-CommutativeRingTheory}, the scheme $\mathcal{Y}\cap \mathcal{H}$ is flat over $\Z_p$. 
It follows that the scheme $\mathcal{Y} \cap \mathcal{H}\rightarrow \Z_p$ is smooth with geometrically integral fibres.  
By construction the special fibre of $\mathcal{Y} \cap \mathcal{H}$ is not contained in $\mathcal{D}$ and the generic fibre of $ \mathcal{H} \cap \mathcal{D}$ is reduced.
The proposition now follows by replacing $\mathcal{Y}$ by $\mathcal{Y}\cap \mathcal{H}$ and induction on the relative dimension of $\mathcal{Y} \rightarrow \Z_p$.

\end{proof}

We come back to our situation of interest. Recall from $\S 4.1$ that $\sh{E}_p$ denotes the subset of elements of $\intbigB(\Z_p)$ of nonzero discriminant.

\begin{corollary}\label{corollary: deform the point in the jacobian general case integral}
	Let $p$ be a prime not dividing $N$.
	Let $b\in \sh{E}_p$ and $P \in \Jac_b(\Q_p)$.
	Then there exists a morphism $\mathcal{X}\rightarrow\Z_p$ which is of finite type, smooth of relative dimension $1$ and has geometrically integral fibres, together with a point $x \in \mathcal{X}(\Z_p)$ satisfying the following properties. 
	\begin{enumerate}
		\item There exists a morphism $\widetilde{b}\colon \mathcal{X} \rightarrow \intbigB_{\Z_p}$ with the property that $\widetilde{b}(x) = b$ and that the discriminant $\Delta(\widetilde{b})$, seen as a map $\mathcal{X} \rightarrow \A^1_{\Z_p}$, is not identically zero on the special fibre and is square-free on the generic fibre of $\mathcal{X}$.
		\item Write $\mathcal{X}^{\rs}$ for the open subscheme of $\mathcal{X}$ where $\Delta(\widetilde{b})$ does not vanish. Then there exists a morphism $\widetilde{P}\colon \mathcal{X}^{\rs} \rightarrow \intJac$ lifting the morphism $ \mathcal{X}^{\rs} \rightarrow \intbigB_{\Z_p}^{\rs}$ satisfying $\widetilde{P}(x_{\Q_p}) = P$. 
	\end{enumerate}

\end{corollary}
\begin{proof}
	We apply Proposition \ref{proposition: Bertini type theorem} with $\mathcal{Y} = \CJac_{\Z_p}$, the compactified Jacobian introduced in \S\ref{subsection: compactifications}.
	We define $\mathcal{D}$ to be the pullback of the discriminant locus $\{ \Delta = 0 \} \subset \intbigB_{\Z_p}$ under the morphism $\CJac_{\Z_p} \rightarrow \intbigB_{\Z_p}$. Since the latter morphism is proper, we can extend $P \in \Jac_b(\Q_p)$ to an element of $\CJac_b(\Z_p)$, still denoted by $P$. 
	
	We claim that the triple $(\mathcal{Y}, \mathcal{D},P)$ satisfies the assumptions of Proposition \ref{proposition: Bertini type theorem}. Indeed, the properties of $\mathcal{Y}$ follow from Corollary \ref{corollary: good compactifications exist}.
	Moreover $\mathcal{Y}_{\F_p}$ is not contained in $\mathcal{D}$ since $\Delta$ is nonzero mod $p$ by our assumptions on $N$ in \S\ref{subsection: integral structures}. 
	Since $\CJac_{\Q_p} \rightarrow \intbigB_{\Q_p}$ is smooth outside a subset of codimension $2$ in $\CJac_{\Q_p}$ and $\{\Delta = 0\}_{\Q_p} \subset \bigB_{\Q_p}$ is reduced, the scheme $\mathcal{D}_{\Q_p}$ is reduced too. 
	Finally $P_{\Q_p}\not \in \mathcal{D}_{\Q_p}$ since $b$ has nonzero discriminant. 
	
	We obtain a closed subscheme $\mathcal{X} \hookrightarrow \CJac_{\Z_p}$ satisfying the conclusion of Proposition \ref{proposition: Bertini type theorem}.
	Write $x\in \mathcal{X}(\Z_p)$ for the section corresponding to $P$, $\widetilde{b}$ for the restriction of $\CJac_{\Z_p} \rightarrow \intbigB_{\Z_p}$ to $\mathcal{X}$ and $\widetilde{P}$ for the restriction of the inclusion $\mathcal{X} \hookrightarrow \CJac_{\Z_p}$ to $\mathcal{X}^{\rs}$. We claim that the tuple $(\mathcal{X},x,\widetilde{b},\widetilde{P})$ satisfies the conclusion of the corollary. 
	This follows readily from Proposition \ref{proposition: Bertini type theorem}, except the statement that the discriminant map $\mathcal{X} \rightarrow \A^1_{\Z_p}$ is square-free on the generic fibre.
    This statement is equivalent to the pullback of the discriminant locus $\{\Delta = 0\} \subset \bigB_{\mathbb{Q}_p}$ along $\tilde{b}_{\mathbb{Q}_p}\colon \mathcal{X}_{\mathbb{Q}_p} \rightarrow \bigB_{\mathbb{Q}_p}$ being reduced. 
    Since this pullback is $\mathcal{X}_{\mathbb{Q}_p} \cap \mathcal{D}_{\Q_p}$ which is reduced by Proposition \ref{proposition: Bertini type theorem}, the statement is true and the corollary follows.

\end{proof}

We have done all the preparations for the proof of Theorem \ref{theorem: integral representatives exist}, which we give now. We keep the notation from this section and assume that we have made a choice of $(\mathcal{X},x,\widetilde{b},\widetilde{P})$ satisfying the conclusion of Corollary \ref{corollary: deform the point in the jacobian general case integral}.
The strategy is to extend the orbit $\eta_b(P)$ (which corresponds to the point $x_{\Q_p}$) to larger and larger subsets of $\mathcal{X}$.
 
Let $y \in \mathcal{X}$ be a closed point of the special fibre with nonzero discriminant which has an affine open neighbourhood containing $x_{\Q_p}$. 
Let $R$ be the semi-local ring of $\mathcal{X}$ at $x_{\Q_p}$ and $y$. Since every projective module of constant rank over a semi-local ring is free \cite{Hinohara-projmodulessemilocalring}, we can apply Theorem \ref{theorem: inject 2-descent into orbits} to obtain an element of $\bigG(R)\backslash \bigV_{\widetilde{b}}(R)$. 
We can spread this out to an element of $\bigG(U_1)\backslash \bigV_{\widetilde{b}}(U_1)$ where $U_1$ is an open subset of $\mathcal{X}$ containing $x_{\Q_p}$ and $y$.
Under the correspondence of Proposition \ref{proposition: G-orbits in terms of groupoids}, this corresponds to an object $(H_1,\chi_1,\theta_1,\gamma_1)$ of $\GrLieE_{U_1,\widetilde{b}}$ whose pullback along the point $x_{\Q_p} \in U_1(\Q_p)$ corresponds to the orbit $\eta_b(P)$.

Let $U_2 = \mathcal{X}_{\Q_p}$.  By Corollary \ref{corollary: int reps sqfree, general dedekind scheme}, the restriction of the object $(H_1,\chi_1,\theta_1,\gamma_1)$ to $U_1\cap U_2$ extends to an object $(H_2,\chi_2,\theta_2,\gamma_2)$ of $\GrLieE_{U_2,\widetilde{b}}$. We can glue these two objects to obtain an object $(H_0,\chi_0,\theta_0,\gamma_0)$ of $\GrLieE_{U_0,\widetilde{b}}$, where $U_0 = U_1 \cup U_2$. 
We observe that the complement of $U_0$ is a union of finitely many closed points since the special fibre of $\mathcal{X}$ is irreducible.
By Lemma \ref{lemma: extend objects complement codimension 2} below, we can extend $(H_0,\chi_0,\theta_0,\gamma_0)$ to an object $(H_3,\chi_3,\theta_3,\gamma_3) \in \GrLieE_{\mathcal{X},\widetilde{b}}$. 
Let $(H_4,\chi_4,\theta_4,\gamma_4)\in \GrLieE_{\Z_p,b}$ denote the pullback of the previous object along the point $x\colon \Spec \Z_p \rightarrow \mathcal{X}$. 
Since $\HH^1(\Z_p,\intbigG) = \{1\}$, Proposition \ref{proposition: G-orbits in terms of groupoids} implies that $(H_4,\chi_4,\theta_4,\gamma_4)$ determines an element of $\intbigG(\Z_p)\backslash \intbigV_b(\Z_p)$ mapping to $\eta_b(P)$ under the natural map $\intbigG(\Z_p) \backslash \intbigV_b(\Z_p) \rightarrow \bigG(\Q_p)\backslash \bigV_b(\Q_p)$. 
This completes the proof of Theorem \ref{theorem: integral representatives exist}.

\begin{lemma}\label{lemma: extend objects complement codimension 2}
	Let $X$ be an integral regular scheme of dimension $2$, and let $U\subset X$ be an open subset whose complement has dimension $0$. 
	If $b\in \intbigB_S(X)$, then restriction induces an equivalence of categories $\GrLieE_{X,b} \rightarrow \GrLieE_{U,b|_U}$.
\end{lemma}
\begin{proof}
	We will use the following fact \cite[Lemme 2.1(iii)]{ColliotTheleneSansuc-Fibresquadratiques} repeatedly: if $Y$ is an affine $X$-scheme of finite type, then restriction of sections $Y(X)\rightarrow Y(U)$ is bijective.
	To prove essential surjectivity, let $(H',\chi',\theta',\gamma')$ be an object of $\GrLieE_{U,b|_U}$.
	By \cite[Th\'eoreme 6.13]{ColliotTheleneSansuc-Fibresquadratiques} and Proposition \ref{proposition: G-torsors in terms of groupoids}, $(H',\chi',\theta')$ extends to an object $(H'',\chi'',\theta'')$ of $\GrLie_{X}$.
	If $Y$ is the closed subscheme of $\lieh''$ of elements $\gamma$ satisfying $\theta''(\gamma)=-\gamma$ and $\gamma$ maps to $b$ in $\intbigB(X)$, then $Y$ is affine and of finite type over $X$. 
	It follows by the fact above that $\gamma'$ lifts to an element $\gamma''\in \lieh''(X)$ and that $(H'',\chi'',\theta'',\gamma'')$ defines an object of $\GrLieE_{X,b}$. 
	Since the scheme of isomorphisms $\Isom_{\GrLieE}(\mathcal{A},\mathcal{A}')$ between two objects of $\GrLieE_{X,b}$ is $X$-affine, fully faithfulness follows again from the above fact.
\end{proof}

\subsection{A global consequence}\label{subsection: integrality, a global corollary}

Recall that $\sh{E}_p = \intbigB(\Z_p) \cap \bigB^{\rs}(\Q_p)$.
Define $\sh{E} \coloneqq \intbigB(\Z) \cap \bigB^{\rs}(\Q)$.
We state the following corollary, whose proof is completely analogous to the proof of \cite[Corollary 5.8]{Thorne-Romano-E8} and uses the fact that $\intbigG$ has class number $1$ (Proposition \ref{proposition: tamagawa}).
\begin{corollary}\label{corollary: weak global integral representatives}
	Let $b_0 \in \sh{E}$. Then for each prime $p$ dividing $N$ we can find an open compact neighbourhood $W_p$ of $b_0$ in $\sh{E}_p$ and an integer $n_p\geq 0$ with the following property. Let $M = \prod_{p | N} p^{n_p}$. Then for all $b\in \sh{E} \cap \left(\prod_{p| N} W_p \right)$ and for all $y \in \Sel_2(\Jac_{M\cdot b})$, the orbit $\eta_{M\cdot b}(y) \in \bigG(\Q) \backslash \bigV_{M\cdot b}(\Q)$ contains an element of $\intbigV_{M\cdot b}(\Z)$. 
\end{corollary} 

This statement about integral representatives will be strong enough to obtain the main theorems in \S\ref{section: proof of main theorems}.

\section{Counting}\label{section: counting}

In this section we will apply the counting techniques of Bhargava to provide estimates for the integral orbits of bounded height in the representation $(\intbigG,\intbigV)$.

\subsection{Heights and measures} \label{subsection: heights and measures}

In this section we introduce measures on various spaces and study the relations between them. 
The results are used in the calculations of \S\ref{section: proof of main theorems}. 
Recall that $\intbigB = \Spec \Z[p_2,p_5,p_6,p_8,p_9,p_{12}]$ and we have a $\G_m$-equivariant morphism $\bigpi \colon \intbigV \rightarrow \intbigB$.
For any $b\in \bigB(\Real)$ we define the \define{height} of $b$ by the formula
$$\height(b) \coloneqq \sup |p_i(b)|^{72/i}.$$
We have $\height(\lambda\cdot b) = |\lambda|^{72} \height(b)$ for all $\lambda \in \Real^{\times}$ and $b\in \bigB(\Real)$.
We define $\height(v) = \height(\bigpi(v))$ for any $v\in \bigV(\Real)$. Note that for each $a\in \Real_{>0}$ the set of elements of $\intbigB(\Z)$ of height less than $a$ is finite.

Let $\omega_{\bigG}$ be a generator for the one-dimensional $\Q$-vector space of left-invariant top differential forms on $\bigG$ over $\Q$. 
It is uniquely determined up to an element of $\Q^{\times}$ and it determines Haar measures $dg$ on $\bigG(\Real)$ and $\bigG(\Q_p)$ for each prime $p$.

\begin{proposition}\label{proposition: tamagawa}
    \begin{enumerate}
        \item $\intbigG$ has class number $1$:  $\intbigG(\A^{\infty}) = \intbigG(\Q)\intbigG(\widehat{\Z})$.
        \item The product $\vol\left(\intbigG(\Z)\backslash \intbigG(\Real) \right) \cdot \prod_p \vol\left(\intbigG(\Z_p)\right)$ converges absolutely and equals $2$, the Tamagawa number of $\bigG$.
    \end{enumerate}

\end{proposition}
\begin{proof}
    The group $\intbigG$ is the Zariski closure of $\bigG$ in $\GL(\intbigV)$ and $\bigG$ contains a maximal $\Q$-split torus consisting of diagonal matrices of $\GL(\intbigV)$.
    Therefore $\intbigG$ has class number $1$ by \cite[Theorem 8.11; Corollary 2]{PlatonovRapinchuk-Alggroupsandnumbertheory}.
	So the product in the second part equals the Tamagawa number $\tau(\bigG)$ of $\bigG\simeq \PSp_8$.
	Now use the identities $\tau(\PSp_8)=2\tau(\Sp_8)$ \cite[Theorem 2.1.1]{Ono-relativetheorytamagawa} and $\tau(\Sp_8)=1$ (because $\Sp_8$ is simply connected).
\end{proof}

Let $\omega_V$ be a generator for the free rank one $\Z$-module of left-invariant top differential forms on $\intbigV$.
Then $\omega_V$ is uniquely determined up to sign and it determines Haar measures $dv$ on $\bigV(\Real)$ and $\bigV(\Q_p)$ for every prime number $p$.
We define the form $\omega_B = dp_2 \wedge dp_5 \wedge dp_6 \wedge dp_8 \wedge dp_9 \wedge dp_{12} $ on $\intbigB$.
It defines measures $db$ on $\bigB(\Real)$ and $\bigB(\Q_p)$ for every prime $p$.

\begin{lemma}\label{lemma: the constants W0 and W}
	There exists a constant $W_0\in \Q^{\times}$ with the following properties:
	\begin{enumerate}
		\item Let $\intbigV(\Z_p)^{\rs} \coloneqq \intbigV(\Z_p)\cap \bigV^{\rs}(\Q_p)$ and define a function $m_p\colon \intbigV(\Z_p)^{\rs} \rightarrow \Real_{\geq 0}$ by the formula
		\begin{equation}\label{equation: def mp(v)}
		m_p(v) \coloneqq \sum_{v' \in \intbigG(\Z_p)\backslash\left( G(\Q_p)\cdot v\cap \intbigV(\Z_p) \right)} \frac{\#Z_{\intbigG}(v)(\Q_p)  }{\#Z_{\intbigG}(v)(\Z_p) }.
		\end{equation}
		Then $m_p(v)$ is locally constant. 
		\item Let $\intbigB(\Z_p)^{\rs} \coloneqq \intbigB(\Z_p)\cap \bigB^{\rs}(\Q_p)$ and let $\psi_p\colon \intbigV(\Z_p)^{\rs}  \rightarrow \Real_{\geq 0}$ be a bounded, locally constant function which satisfies $\psi_p(v) = \psi_p(v')$ when $v,v'\in \intbigV(\Z_p)^{\rs}$ are conjugate under the action of $G(\Q_p)$. 
		Then we have the formula 
		\begin{equation}\label{equation: p-adic property W0}
		\int_{v\in \intbigV(\Z_p)^{\rs}} \psi_p(v) d v = |W_0|_p \vol\left(\intbigG(\Z_p)\right) \int_{f\in \intbigB(\Z_p)^{\rs}} \sum_{g\in \bigG(\Q_p)\backslash \intbigV_b(\Z_p) } \frac{m_p(v)\psi_p(v)  }{\# Z_{\intbigG }(v)(\Q_p)} d b .
		\end{equation}
		
		\item Let $U_0\subset \bigG(\Real)$ and $U_1\subset B^{\rs}(\Real)$ be open subsets such that the product morphism $\mu: U_0 \times U_1 \rightarrow V(\Real)^{\rs}$, given by $(g,b) \mapsto g\cdot \bigsigma(b)$, is injective. Then we have the formula
		\begin{equation}\label{equation: archimedean property W_0}
		\int_{v\in \mu\left(U_0\times U_1 \right)} dv = |W_0|_{\infty} \int_{g\in U_0} dg \int_{b\in U_1} db.
		\end{equation}
		
	\end{enumerate}

\end{lemma}
\begin{proof}
	The proof is identical to the proof of \cite[Proposition 3.3]{Romano-Thorne-ArithmeticofsingularitiestypeE}. 
	Here we use the fact that the sum of the degrees of the invariants equals the dimension of the representation: $2+5+6+8+9+12 = 42 = \dim_{\Q}\bigV$.
\end{proof}

We henceforth fix a constant $W_0 \in \Q^{\times}$ satisfying the properties of Lemma \ref{lemma: the constants W0 and W}.

\subsection{Counting integral orbits}\label{subsection: counting with no congruence}

In this section we count integral orbits in the representation $\intbigV$. 
For any $\intbigG(\Z)$-invariant subset $X\subset \intbigV(\Z)$, define 
$$N(X,a) \coloneqq \sum_{\substack{v\in \intbigG(\Z)\backslash X \\ \height(v)<a}} \frac{1}{\# Z_{\intbigG}(v)(\Z)}.$$
Let $k$ be a field of characteristic not dividing $N$. 
We say an element $v\in \intbigV(k)$ is \define{$k$-reducible} if it has zero discriminant or if it is $\intbigG(k)$-conjugate to the Kostant section $\bigsigma(\bigpi(v))$, and \define{$k$-irreducible} otherwise. 
We say an element $v\in \intbigV(k)$ is \define{$k$-soluble} if it has nonzero discriminant and lies in the image of the map $\eta_b\colon \intJac_b(k)/2\intJac_b(k) \rightarrow \intbigG(k)\backslash \intbigV_b(k)$ from Theorem \ref{theorem: inject 2-descent into orbits} where $b = \pi(v)$. 

For any $X\subset \intbigV(\Z)$ write $X^{irr}\subset X$ for the subset of $\Q$-irreducible elements. 
Write $\bigV(\Real)^{sol} \subset \bigV(\Real)$ for the subset of $\Real$-soluble elements. 
Recall that we have fixed a constant $W_0\in\Q^{\times}$ in \S\ref{subsection: heights and measures}.

\begin{theorem}\label{theorem: counting R-soluble elements, no congruence}
	We have 
	\begin{displaymath}
	N(\intbigV(\Z)^{irr} \cap \bigV(\Real)^{sol},a) = \frac{|W_0|}{8}\vol\left(\intbigG(\Z)\backslash \bigG(\Real)\right) \vol\left(\left\{b \in \bigB(\Real) \mid \height(b) < a  \right\}   \right)+ o\left(a^{7/12}\right).
	\end{displaymath}

\end{theorem}
 
It will suffice to prove the following proposition. 
Recall that there exists $\G_m$-actions on $\bigV$ and $\bigB$ such that the morphism $\bigpi\colon \bigV \rightarrow \bigB$ is $\G_m$-equivariant, giving actions of $\Real_{>0}$ on $\bigV(\Real)$ and $\bigB(\Real)$. 

\begin{proposition}\label{prop: counting sections}
	Let $U\subset \bigB^{\rs}(\Real)$ be a connected open semialgebraic subset stable under the action of $\Real_{>0}$ and let $s: U \rightarrow \bigV^{\rs}(\Real)$ be a semialgebraic $\Real_{>0}$-equivariant section of $\bigpi$ such that $s(U) \cap \{v\in \bigV(\Real) \mid \height(v) = 1\}$ is a bounded subset of $\bigV(\Real)$. 
	Then 
	$$N(\bigG(\Real)\cdot s(U) \cap \intbigV(\Z)^{irr},a) = \frac{|W_0|}{\#Z_{\bigG}(v_0)(\Real)}\vol\left(\intbigG(\Z)\backslash \bigG(\Real)\right) \vol\left(\left\{b \in U \mid \height(b) < a  \right\}   \right) +o\left(a^{7/12} \right),$$
	where $v_0$ is any element of $s(U)$.
\end{proposition}

\begin{proof}[Proof that Proposition \ref{prop: counting sections} implies Theorem \ref{theorem: counting R-soluble elements, no congruence}]

Arguing exactly as in \cite[\S1.9]{Thorne-E6paper}, we can find connected semialgebraic open subsets $L_i \subset \{b\in \bigB^{\rs}(\Real) \mid \height(b) = 1\}$ and sections $s_i:L_i \rightarrow \bigV(\Real)$ which are semialgebraic for $i=1,\dots,r$ such that if $U_i \coloneqq \Real_{>0} \cdot L_i$ then (if we continue to write the unique extension of $s_i$ to a $\Real_{>0}$-equivariant map $U_i \rightarrow \bigV(\Real)$ by $s_i$):
	$$\bigV^{\rs}(\Real) = \bigcup_{i=1}^r \bigG(\Real) \cdot s_i(U_i).$$
	Each $U_i$ is connected and the set $\bigV(\Real)^{sol} \subset \bigV^{\rs}(\Real)$ is open and closed by Lemma \ref{lemma: soluble elements open and closed}.
	So the image $s_i(U_i)$ either consists only of $\Real$-soluble elements or contains no $\Real$-soluble elements at all. Therefore by replacing $r$ by a smaller integer, we may write 
	$\bigV(\Real)^{sol} = \bigcup_{i=1}^r \bigG(\Real) \cdot s_i(U_i)$.
	
	Note that if $b\in \bigB^{\rs}(\Real)$ the number of $\bigG(\Real)$-orbits on $\bigV_b(\Real)^{sol}$ equals $\# \Jac_b(\Real)/2\Jac_b(\Real)$. Moreover the quantity $\# \Jac_b(\Real)/2\Jac_b(\Real)/\#\Jac_b[2](\Real)$ is independent of $b$, and equals $1/8$; this is a general fact about real abelian threefolds.
	Theorem \ref{theorem: counting R-soluble elements, no congruence} then follows from the inclusion-exclusion principle applied to the decomposition $\bigV(\Real)^{sol} = \bigcup_{i=1}^r \bigG(\Real) \cdot s_i(U_i)$, together with Proposition \ref{prop: counting sections} applied to (the connected components of) $U_I=\pi\left( \cap_{i\in I} \bigG(\Real)\cdot s_i(U_i) \right)$ for every $I\subset \{1,\dots,r\}$.

\end{proof}

\begin{lemma}\label{lemma: soluble elements open and closed}
	The subset $\bigV(\Real)^{sol}\subset \bigV^{\rs}(\Real)$ is open and closed in the Euclidean topology. 
\end{lemma}
\begin{proof}
	We first prove that for each $b\in \bigB^{\rs}(\Real)$, we can find an open connected neighbourhood $U\subset \bigB^{\rs}(\Real)$ of $b$ and a partition $W_1 \sqcup \dots \sqcup W_n $ of $\bigV(\Real)_U$ ($=$ the subset of $\bigV(\Real)$ mapping to $U$) such that:
	\begin{enumerate}
		\item For all $i$, $W_i$ is open and closed in $\bigV(\Real)_U$ and stable under the action of $\bigG(\Real)$.
		\item For all $i$, if two elements $v, v' \in W_i$ have the same image in $U$, then $v$ and $v'$ are $\bigG(\Real)$-conjugate. 
	\end{enumerate}
	Indeed, Lemma \ref{lemma: AIT} implies that $\bigV_b(\Real)$ consists of finitely many $\bigG(\Real)$-orbits; let $v_1, \dots, v_n \in \bigV_b(\Real)$ be a system of representatives.
	Similarly the space $\bigV(\Real)$ contains finitely many $\bigG(\Real)$-conjugacy classes of Cartan subalgebras; let $\liec_1,\dots,\liec_k$ be a system of representatives. 
	Then every $v\in \bigV^{\rs}(\Real)$ is $\bigG(\Real)$-conjugate to an element of $\liec_j^{\rs}(\Real)$ for some unique $j$, and two elements of $\liec_j^{\rs}(\Real)$ are $\bigG(\Real)$-conjugate if and only if they are conjugate under the finite group $N_{\bigG}(\liec_j)(\Real)$.
	So after conjugation we may assume that there exists a function $f\colon \{1,\dots,n \} \rightarrow \{1,\dots,k \}$ such that $v_i \in \liec_{f(i)}^{\rs}(\Real)$ for all $i=1,\dots,n$. 
	For each $j$ write $\pi_j\colon \liec_j^{\rs}(\Real) \rightarrow \bigB^{\rs}(\Real)$ for the $\Real$-points of the quotient map. Then $\pi_j$ is a proper local homeomorphism and $N_{\bigG}(\liec_j)(\Real)$ acts on its fibres. 
	By \cite[Proposition 9.3.9]{RealAlgebraicgeometry} we can find a semialgebraic connected open subset $U\subset \bigB^{\rs}(\Real)$ containing $b$ and semialgebraic sections $s_i\colon U \rightarrow \liec_{f(i)}^{\rs}(\Real)$ such that for each $j$, every $v\in \liec_j^{\rs}(\Real)$ with $\pi_j(v)\in U$ is $\bigG(\Real)$-conjugate to an element of $s_i(U)$ for some unique $i$ with $f(i)=j$. 
	If we set $W_i = \bigG(\Real)\cdot s_i(U) $ then the $W_i$ form a partition of $\bigV(\Real)_U $ with the required properties. 
	
	Next one can similarly show that for every $b\in \bigB^{\rs}(\Real)$, there exists an open neighbourhood $U\subset \bigB^{\rs}(\Real)$ of $b$ such that the family of compact Lie groups $\Jac(\Real) \rightarrow \bigB^{\rs}(\Real)$ is trivialized above $U$, as well as the finite groups $\HH^1(\Real,\Jac_b[2])$ and $\HH^1(\Real,\Jac_b)[2]$. 
	Suppose moreover that we further shrink $U$ such that there exists a partition $\bigV(\Real)_U = W_1\sqcup \dots \sqcup W_n$ with the properties as above. 
	Then the map $\bigV(\Real)_U \rightarrow \HH^1(\Real,\Jac_b[2])$, obtained from Lemma \ref{lemma: AIT} and by identifying $\HH^1(\Real,\Jac_{b'}[2])$ with $\HH^1(\Real,\Jac_{b}[2])$ for each $b'\in U$, is constant on each $W_i$. 
	
	Combining the previous paragraphs shows that for every $v\in \bigV^{\rs}(\Real)$ that is $\Real$-soluble (resp. not $\Real$-soluble), there exists an open neighbourhood $W\subset \bigV^{\rs}(\Real)$ of $v$ such that every element of $W$ is $\Real$-soluble (resp. not $\Real$-soluble). 
	This completes the proof.

\end{proof}

So to prove Theorem \ref{theorem: counting R-soluble elements, no congruence} it remains to prove Proposition \ref{prop: counting sections}. 
The proof of this proposition is the same as the proof of \cite[Theorem 3.1]{Thorne-E6paper} but by systematically using multisets and keeping track of the stabilizers as in \cite[\S10]{Bhargava-Gross-hyperellcurves}. (See the proof of \cite[Theorem 6.6]{Laga-F4paper} for a detailed exposition of such an orbit-counting result in a very similar set-up.)
We note that `cutting off the cusp' has been carried out in \cite{Thorne-E6paper}.
The only missing ingredient is Proposition \ref{proposition: estimates on red and bigstab}, whose proof we give below.

To state the proposition we first introduce some notation. Let $\alpha_0\in \Phi(H,T)$ be the highest root of $H$ with respect to the root basis fixed in \S\ref{subsection: a stable grading}. Let $a_0\in X^*(T^{\theta})$ be the restriction of $\alpha_0$ to $T^{\theta}$. 
Then $a_0$ is a weight for the $T^{\theta}$-action on $\bigV$.
If $v\in \bigV$ we can decompose $v$ into eigenvectors $\sum_a v_a$ where $a$ runs over the weights of $T^{\theta}$ on $\bigV$ and $T^{\theta}$ acts on $v_a$ via $a$. Write $V(a_0)$ for the subset of $v\in \bigV$ with the property that $v_{a_0}=0$. We call $V(a_0)$ the \define{cuspidal region}. 
Thorne has proven in \cite[\S2.3]{Thorne-E6paper} that the number of irreducible integral points in the cuspidal region is negligible.


By an identical argument to \cite[\S10.7]{Bhargava-Gross-hyperellcurves} (see also the discussion after \cite[Lemma 6.17]{Laga-F4paper}), Lemma \ref{lemma: red and bigstab mod p} below implies that the number of $\Q$-reducible elements in the main body is negligible. It also implies the following proposition which will be used in the proof of Theorem \ref{theorem: main theorem}. 

\begin{proposition}\label{proposition: estimates on red and bigstab}
	Let $V^{bigstab}$ denote the subset of $\Q$-irreducible elements $v\in \intbigV(\Z)$ with $\#Z_G(v)(\Q) >1$. Then $N(V^{bigstab},a) = o(a^{7/12})$.
\end{proposition}

Let $N$ be the integer of \S\ref{subsection: integral structures} and let $p$ be a prime not dividing $N$. We define $\bigV_p^{red}\subset \intbigV(\Z_p)$ to be the set of vectors whose reduction mod $p$ is $\F_p$-reducible.
We define $\bigV_p^{bigstab} \subset \intbigV(\Z_p)$ to be the set of vectors $v\in \intbigV(\Z_p)$ such that $p | \Delta(v)$ or whose image in $\intbigV(\F_p)$ has nontrivial stabilizer in $\intbigG(\F_p)$. 

\begin{lemma}\label{lemma: red and bigstab mod p}
	We have 
	$$\lim_{Y\rightarrow +\infty} \prod_{N<p<Y} \int_{V_p^{red}} dv = 0,$$
	and
	$$\lim_{Y\rightarrow +\infty} \prod_{N<p<Y} \int_{V_p^{bigstab}} dv = 0.$$
	
\end{lemma}
\begin{proof}
	The proof is very similar to the proof of \cite[Proposition 6.9]{Thorne-Romano-E8}. We only treat the case of $V_p^{bigstab}$, the case of $V_p^{red}$ being analogous and treated in detail in \cite[\S10.7]{Bhargava-Gross-hyperellcurves}. 
	Let $p$ be a prime not dividing $N$. We have the formula
	$$ \int_{V_p^{bigstab}} dv = \frac{1}{\#\intbigV(\F_p)}\# \{v\in \intbigV(\F_p) \mid \Delta(v) = 0 \text{ or } Z_{\intbigG}(v)(\F_p) \neq 1 \}   .$$
	Since $\{ \Delta = 0 \}$ is a hypersurface we have
	$$\frac{1}{\#\intbigV(\F_p)}\# \{v\in \intbigV(\F_p) \mid \Delta(v) = 0\} = O(p^{-1}).  $$
	If $v\in \intbigV^{\rs}(\F_p)$ then $\#Z_{\intbigG}(v)(\F_p)$ depends only on $\pi(v)$ by (the $\Z[1/N]$-analogue of) Lemma \ref{lemma: centralizers with same invariants isomorphic}.
	Moreover by Proposition \ref{proposition: G-orbits in terms of groupoids} and Lang's theorem we have $\#\intbigV^{\rs}(\F_p) = \#\intbigG(\F_p) \#\intbigB^{\rs}(\F_p) $.
	So to prove the lemma it suffices to prove that there exists a $0< \delta <1$ such that
	$$
	\frac{1}{\# \intbigB^{\rs}(\F_p)}\#\{b\in \intbigB^{\rs}(\F_p) \mid Z_{\intbigG}(\sigma(b))(\F_p) \neq 1 \} \rightarrow \delta
	$$
	as $p \rightarrow +\infty$.
	We will achieve this using the results of \cite[\S9.3]{Serre-lecturesonNx(p)}. 
	Recall from \S\ref{subsection: a stable grading} that $\bigT$ is a split maximal torus of $\bigH$ with Lie algebra $\liet$ and Weyl group $W$.
	These objects spread out to objects $\intbigT, \intbigH,\intbigt$ over $\Z$. In \S\ref{subsection: further properties of J[2]} we have defined a $W$-torsor $f\colon \liet^{\rs}\rightarrow \bigB^{\rs}$ which extends to a $W$-torsor $\intbigt_S^{\rs} \rightarrow \intbigB_S^{\rs}$, still denoted by $f$. The group scheme $\Jac[2] \rightarrow \intbigB^{\rs}_S$ is trivialized along $f$ and the monodromy action is given by the natural action of $W$ on $\Lambda_T/2\Lambda_T$ by the same logic as Proposition \ref{proposition: monodromy of J[2]}.
	Let $C\subset W$ be the subset of elements of $W$ which fix some nonzero element of $\Lambda_T/2\Lambda_T$.
	Then \cite[Proposition 9.15]{Serre-lecturesonNx(p)} implies that
	$$
	\frac{1}{\# \intbigB^{\rs}(\F_p)}\#\{b\in \intbigB^{\rs}(\F_p) \mid Z_{\intbigG}(\sigma(b))(\F_p) \neq 1 \}  = \frac{\#C}{\#W}+O(p^{-1/2}).
	$$
	To finish the proof it suffices to show that $C \neq W$. 
	Let $w_{cox}\in W$ be a Coxeter element. Then the determinant of $1-w_{cox}$ on $\Lambda$ is \cite[Theorem 10.6.1]{Carter-SimpleGroupsLieType1972}:
	$$ \prod_{i} \left(1-e^{2\pi i (\deg(p_i)-1)/12} \right) = \Phi_{12}(1)\Phi_3(1)=3.$$
	(Here $\Phi_n$ denotes the $n$-th cyclotomic polynomial.)
	Since this determinant is odd, the action of a Coxeter element on the mod $2$ root lattice does not fix any nonzero vector. This implies that $w_{cox} \not\in C$, as desired.

\end{proof}

\subsection{Counting with congruence conditions}\label{subsection: congruence conditions}

We now introduce variants of Theorem \ref{theorem: counting R-soluble elements, no congruence} by imposing certain congruence conditions. 
We start with a version which involves only finitely many such congruence conditions. 
Let $M$ be a positive integer and $w\colon \intbigV(\Z/M\Z) \rightarrow \Real$ a function. For any $\intbigG(\Z)$-invariant subset $X\subset \intbigV(\Z)$ we write 
$$N_w(X,a) \coloneqq \sum_{\substack{v\in \intbigG(\Z)\backslash X \\ \height(v)<a}} \frac{w\left(v \mod M \right) }{\# Z_{\intbigG}(v)(\Z)}.$$
We write $\mu_w$ for the average of $w$ where we put the uniform measure on $\intbigV(\Z/M\Z)$.
The following theorem follows from the proof of Theorem \ref{theorem: counting R-soluble elements, no congruence} in the same way as \cite[\S2.5]{BS-2selmerellcurves}.
Recall that we have fixed a constant $W_0\in\Q^{\times}$ in \S\ref{subsection: heights and measures}.

\begin{theorem}\label{theorem: counting R-soluble elements, finite congruence}
	We have
	\begin{displaymath}
	N_w(\intbigV(\Z)^{irr} \cap \bigV(\Real)^{sol},a) = \mu_w\frac{|W_0|}{8}\vol\left(\intbigG(\Z)\backslash \bigG(\Real)\right) \vol\left(\left\{b \in \bigB(\Real) \mid \height(b) < a  \right\}   \right)+ o\left(a^{7/12}\right).
	\end{displaymath}

\end{theorem}

We now consider the situation where we impose infinitely many congruence conditions which is needed to sieve out those orbits not corresponding to $2$-Selmer elements. 
Suppose we are given for each prime $p$ a $\intbigG(\Z_p)$-invariant function $w_p\colon \intbigV(\Z_p) \rightarrow [0,1]$ with the following properties:
\begin{itemize}
	\item The function $w_p$ is locally constant outside the closed subset $\{v\in \intbigV(\Z_p) \mid \Delta(v) = 0\} \subset \intbigV(\Z_p)$. 
	\item For $p$ sufficiently large, we have $w_p(v) = 1$ for all $v \in \intbigV(\Z_p)$ such that $p^2 \nmid \Delta(v)$. 
\end{itemize}
In this case we can define a function $w\colon \intbigV(\Z) \rightarrow [0,1]$ by the formula $w(v) = \prod_{p} w_p(v)$ if $\Delta(v) \neq 0$ and $w(v) = 0$ otherwise. Call a function $w\colon \intbigV(\Z) \rightarrow [0,1]$ defined by this procedure \define{acceptable}. 
For any $\intbigG(\Z)$-invariant subset $X\subset \intbigV(\Z)$ we define 
\begin{equation}
N_w(X,a) \coloneqq \sum_{\substack{v\in \intbigG(\Z)\backslash X \\ \height(v)<a}} \frac{w(v)}{\# Z_{\intbigG}(v)(\Z)}.
\end{equation}
The proof of the following inequality is standard. (Details can be found in the first part of the proof of \cite[Theorem 2.21]{BS-2selmerellcurves}.) 

\begin{theorem}\label{theorem: counting infinitely many congruence conditions}
	If $w\colon \intbigV(\Z) \rightarrow [0,1]$ is an acceptable function we have
	\begin{displaymath}
	N_w(\intbigV(\Z)^{irr}\cap \bigV(\Real)^{sol} ,a) \leq \frac{|W_0|}{8} \left(\prod_p \int_{\intbigV(\Z_p)} w_p(v) d v \right)  \vol\left(\intbigG(\Z) \backslash \bigG(\Real) \right)\vol\left(\{b\in \bigB(\Real) \mid \height(b) < a \} \right)  + o(a^{7/12}). 
	\end{displaymath}

\end{theorem}

\begin{remark}
	If we would be able to prove a so-called uniformity estimate bounding the error term occurring in Theorem \ref{theorem: counting R-soluble elements, no congruence} similar to \cite[Theorem 2.13]{BS-2selmerellcurves}, then we can strengthen the above inequality to an actual equality, which would lead to an equality in Theorem \ref{theorem: main theorem}.
\end{remark}

To count $2$-Selmer elements in $\sh{E}_{\min}$ we will require a slight variant of the above theorem.
We write $\bigB(\Real)_{\min}\subset \bigB(\Real)$ for the subset of elements $b$ satisfying the following condition: either $p_5(b)>0$, or $p_5(b)=0$ and $p_9(b)\geq 0$. 
For any $X \subset \bigV(\Real)$ we write $X_{\min}\subset X$ for the subset of elements whose image under $\pi$ lies in $\bigB^{\rs}(\Real)_{\min}$.

\begin{theorem}\label{theorem: counting infinitely many congruence conditions minimal}
	Let $w\colon \intbigV(\Z) \rightarrow [0,1]$ be an acceptable function satisfying $w(v) = w(-v)$ for all $v\in \intbigV(\Z)$. Then we have
	\begin{displaymath}
	N_w(\intbigV(\Z)^{irr}\cap \bigV(\Real)^{sol}_{\min} ,a) \leq \frac{|W_0|}{8} \left(\prod_p \int_{\intbigV(\Z_p)} w_p(v) d v \right)  \vol\left(\intbigG(\Z) \backslash \bigG(\Real) \right)\vol\left(\{b\in \bigB(\Real)_{\min} \mid \height(b) < a \} \right)  + o(a^{7/12}). 
	\end{displaymath}
\end{theorem}
\begin{proof}
	This can be proved by adapting the counting arguments in \S\ref{subsection: counting with no congruence}, but we can deduce it easily from Theorem \ref{theorem: counting infinitely many congruence conditions}.
	Indeed, observe that $p_i(-b) = (-1)^ip_i(b)$ for any $b\in \bigB(\Real)$. So, away from elements $v$ with $p_5(v) = p_9(v) = 0$, we see that every $\intbigG(\Z)$-orbit in $\bigV(\Real)_{\min}$ gives rise to exactly two $\intbigG(\Z)$-orbits in $\intbigV(\Real)$. 
	Moreover an element $v\in \bigV(\Real)$ is $\Real$-soluble if and only if $-v$ is.
	Since the number of $\intbigG(\Z)$-orbits in $\intbigV(\Z)$ whose invariants $p_5$ and $p_9$ vanish is $o(a^{7/12})$, we see that 
	\begin{align*}
		N_w(\intbigV(\Z)^{irr}\cap \bigV(\Real)^{sol} ,a) = 2N_w(\intbigV(\Z)^{irr}\cap \bigV(\Real)^{sol}_{\min} ,a)+o\left(a^{7/12} \right).
	\end{align*}
	The theorem now follows from the equality $\vol\left(\{b\in \bigB(\Real) \mid \height(b) <a  \} \right) = 2 \vol\left(\{b\in \bigB(\Real)_{\min} \mid \height(b) <a  \} \right) $.
	
\end{proof}

\section{Proof of the main theorem}\label{section: proof of main theorems}

In this section we prove the first main theorem stated in the introduction. 
Recall that we write $\sh{E}$ for the set of elements $b\in \intbigB(\Z)$ of nonzero discriminant. We write $\sh{E}_{\min} \subset \sh{E}$ for the subset of $b\in \sh{E}$ such that:
\begin{itemize}
	\item No prime $q$ has the property that $q^i$ divides $p_i(b)$ for all $i\in \{2,5,6,8,9,12  \}$.
	\item Either $p_5(b)>0$, or $p_5(0) =0$ and $p_9(b) \geq 0$. 
\end{itemize}
The set $\sh{E}_{\min}$ is in canonical bijection with the set of isomorphism classes of pairs $(X,P)$ where $X/\Q$ is a smooth, geometrically connected and projective curve of genus $3$ which is not hyperelliptic and $P \in X(\Q)$ is a marked hyperflex point (this follows from \cite[Lemma 4.1]{Thorne-E6paper}). 
We recall that we have defined a height function $\height$ for $\sh{E}$ in \S\ref{subsection: heights and measures}. We say a subset $\mathcal{F}\subset \sh{E}$ is defined by \define{finitely many congruence conditions} if $\mathcal{F}$ is the preimage of a subset of $\intbigB(\Z/N\Z)$ under the reduction map $\sh{E} \rightarrow \intbigB(\Z/N\Z)$ for some $N\geq 1$.

\begin{theorem}\label{theorem: main theorem}
	Let $\mathcal{F}\subset \sh{E}$ be a subset defined by finitely many congruence conditions or $\mathcal{F} = \sh{E}_{\min}$.
	Then we have 
	\begin{equation*}
	\limsup_{a\rightarrow \infty} \frac{ \sum_{b\in \mathcal{F},\; \height(b)<a }\# \Sel_2\Jac_b  }{\#  \{b \in \mathcal{F} \mid \height(b) < a \}}	\leq 3.
	\end{equation*}
\end{theorem}

The proof is along the same lines as the discussion in \cite[\S7]{Thorne-Romano-E8}. We will assume that $\mathcal{F} = \sh{E}_{\min}$, the other case being very similar.

We first prove a `local' result. Recall that $\sh{E}_p$ is the set of elements $b\in \intbigB(\Z_p)$ of nonzero discriminant, and define $\sh{E}_{p,\min} \subset \sh{E}_p$ to be the subset of those $b$ that do not lie in $p\cdot \intbigB(\Z_p)$. (Recall that there is a $\G_m$-action on $\intbigB$ which satisfies $\lambda\cdot p_i = \lambda^i p_i$.)

\begin{proposition}\label{proposition: local result of main theorem}
	Let $b_0 \in \sh{E}_{\min}$. Then we can find for each prime $p$ dividing $N$ an open compact neighbourhood $W_p$ of $b_0$ in $\sh{E}_p$ such that the following condition holds. Let $\sh{E}_W = \sh{E} \cap \left(\prod_{p | N} W_p \right)$, and let $\sh{E}_{W,\min} = \sh{E}_W \cap \sh{E}_{\min}$. Then we have 
	\begin{equation*}
	\limsup_{a\rightarrow \infty} \frac{ \sum_{b\in \sh{E}_{W,\min},\; \height(b)<a }\# \Sel_2\Jac_b   }{\#  \{b \in \sh{E}_{W,\min} \mid \height(b) < a \}}	\leq 3.
	\end{equation*}
\end{proposition}
\begin{proof}
	Choose the sets $W_p$ and integers $n_p\geq 0$ for $p| N$ satisfying the conclusion of Corollary \ref{corollary: weak global integral representatives}. We assume after shrinking the $W_p$ that they satisfy $W_p \subset \sh{E}_{p,\min}$. 
	If $p$ does not divide $N$, set $W_p = \sh{E}_{p,\min}$ and $n_p = 0$. Let $M = \prod_{p} p^{n_p}$. 
	
	For $v\in \intbigV(\Z)$ with $\bigpi(v) = b$, define $w(v) \in \Q_{\geq 0}$ by the following formula:
	\begin{displaymath}
	w(v) = 
	\begin{cases}
	\left( \sum_{v'\in \intbigG(\Z)\backslash \left( \intbigG(\Q)\cdot v \cap \intbigV(\Z) \right)}  \frac{\# Z_{\intbigG}(v')(\Q)}{\# Z_{\intbigG}(v')(\Z)} \right)^{-1} & \text{if }b\in p^{n_p}\cdot W_p \text{ and } \bigG(\Q_p)\cdot v \in \eta_{b}(\Jac_b(\Q_p)/2\Jac_b(\Q_p)) \text{ for all }p, \\
	0 & \text{otherwise.}
	\end{cases}
	\end{displaymath}
	Define $w'(v)$ by the formula $w'(v) = \#Z_{\intbigG}(v)(\Q) w(v)$. 
	Corollary \ref{corollary: Sel2 embeds} and Corollary \ref{corollary: weak global integral representatives} imply that if $b\in M \cdot \sh{E}_{W,\min}$, non-identity elements in the $2$-Selmer group of $\Jac_b$ correspond bijectively to $\bigG(\Q)$-orbits in $\bigV_b(\Q)$ that intersect $\intbigV(\Z)$ nontrivially, that are $\Q$-irreducible and that are soluble at $\Real$ and $\Q_p$ for all $p$. In other words, we have the formula:
	\begin{equation}\label{equation: selmer count vs orbit count}
	\sum_{\substack{b \in \sh{E}_{W,\min} \\ \height(b) <a}}\left( \#\Sel_2(\Jac_b)-1 \right) 
	= \sum_{\substack{b \in M\cdot\sh{E}_{W,\min} \\ \height(b) <M^{72}a}}\left( \#\Sel_2(\Jac_b)-1 \right)
	= N_{w'}(\intbigV(\Z)^{irr}\cap\bigV(\Real)^{sol}_{\min}  ,M^{72}a).
	\end{equation}
	Proposition \ref{proposition: estimates on red and bigstab} implies that
	\begin{equation}\label{equation: compare w and w'}
	 N_{w'}(\intbigV(\Z)^{irr}\cap \bigV(\Real)^{sol}_{\min}  ,M^{72}a) =  N_{w}(\intbigV(\Z)^{irr}\cap \bigV(\Real)^{sol}_{\min},M^{72}a) + o(a^{7/12}).
	\end{equation}
	It is more convenient to work with $w(v)$ than with $w'(v)$ because $w(v)$ is an acceptable function in the sense of \S\ref{subsection: congruence conditions}. 
	Indeed, for $v\in \intbigV(\Z_p)$ with $\bigpi(v)=b$, define $w_p(v) \in \Q_{\geq 0}$ by the following formula
	\begin{displaymath}
	w_p(v) = 
	\begin{cases}
	\left( \sum_{v'\in \intbigG(\Z_p)\backslash \left( \intbigG(\Q_p)\cdot v \cap \intbigV(\Z_p) \right)}  \frac{\# Z_{\intbigG}(v')(\Q_p)}{\# Z_{\intbigG}(v')(\Z_p)} \right)^{-1} & \text{if }b\in p^{n_p}\cdot W_p \text{ and } \bigG(\Q_p)\cdot v \in \eta_{b}(\Jac_b(\Q_p)/2\Jac_b(\Q_p)   ), \\
	0 & \text{otherwise.}
	\end{cases}
	\end{displaymath}
	Then an argument identical to \cite[Proposition 3.6]{BS-2selmerellcurves} shows that $w(v)  =\prod_pw_p(v)$ for all $v\in\intbigV(\Z)$. The remaining properties for $w(v)$ to be acceptable follow from Part 1 of Lemma \ref{lemma: the constants W0 and W} and Proposition \ref{prop: integral reps squarefree discr}. 
	From Lemma \ref{lemma: the constants W0 and W} we obtain the formula
	\begin{equation}\label{equation: mass formula w}
	\int_{v\in \intbigV(\Z_p)} w_p(v) d v = |W_0|_p \vol\left(\intbigG(\Z_p) \right) \int_{b \in p^{n_p}\cdot {W_p}} \frac{\#\Jac_b(\Q_p)/2\Jac_b(\Q_p)}{\#\Jac_b[2](\Q_p)}d b.
	\end{equation}
	Using the equality $\#\Jac_b(\Q_p)/2\Jac_b(\Q_p) = |1/8|_p  \#\Jac_b[2](\Q_p)$ which holds for all $b\in \sh{E}_p$, we see that the integral on the right hand side equals $|1/8|_p\vol(p^{n_p}\cdot W_p)=|1/8|_pp^{-n_p\dim_{\Q}V} \vol(W_p)$.
	Combining the identities (\ref{equation: selmer count vs orbit count}) and (\ref{equation: compare w and w'}) shows that 
	\begin{align*}
	\limsup_{a\rightarrow +\infty} a^{-7/12} \sum_{\substack{b \in \sh{E}_{W,\min} \\ \height(b) <a}}\left( \#\Sel_2(\Jac_b)-1 \right)
	& = \limsup_{a\rightarrow +\infty} a^{-7/12}N_w(\intbigV(\Z)^{irr}\cap \bigV(\Real)^{sol}_{\min} ,M^{72}a).\\
	\end{align*}
	This in turn by Theorem \ref{theorem: counting infinitely many congruence conditions minimal} is less then or equal to
	 \begin{displaymath}
	 \frac{|W_0|}{8} \left(\prod_p \int_{\intbigV(\Z_p)} w_p(v) d v \right)  \vol\left(\intbigG(\Z) \backslash \bigG(\Real) \right) 2^5M^{42}.
	 \end{displaymath}
	 Using (\ref{equation: mass formula w}) this simplifies to
	\begin{displaymath}
	\vol\left(\intbigG(\Z)\backslash \intbigG(\Real) \right) \prod_p \vol\left(\intbigG(\Z_p)\right) 2^5\prod_{p} \vol(W_p).
	\end{displaymath}
	On the other hand, an elementary sieving argument shows that
	\begin{displaymath}
	\lim_{a\rightarrow +\infty} \frac{\#  \{b \in \sh{E}_{W,\min} \mid ht(b) < a \}}{a^{7/12}} = 2^5\prod_p \vol(W_p).
	\end{displaymath}
	We conclude that 
	\begin{displaymath}
	\limsup_{a\rightarrow \infty} \frac{ \sum_{b\in \sh{E}_{W,\min},\; ht(b)<a } \left(\# \Sel_2\Jac_b-1 \right)  }{\#  \{b \in \sh{E}_{W,\min} \mid ht(b) < a \}}	\leq \vol\left(\intbigG(\Z)\backslash \intbigG(\Real) \right) \cdot \prod_p \vol\left(\intbigG(\Z_p)\right).
	\end{displaymath}
	Since the Tamagawa number of $\intbigG$ is $2$ (Proposition \ref{proposition: tamagawa}), the proposition follows. 
\end{proof}

To deduce Theorem \ref{theorem: main theorem} from Proposition \ref{proposition: local result of main theorem}, choose for each $i\geq1$ sets $W_{p,i} \subset\sh{E}_p$ (for $p$ dividing $N$) such that if $W_i = \sh{E} \cap \left( \prod_{p | N} W_{p,i} \right)$, then $W_i$ satisfies the conclusion of Proposition \ref{proposition: local result of main theorem} and we have a countable partition $\sh{E}_{\min} = \sh{E}_{W_1,\min} \sqcup \sh{E}_{W_2,\min} \sqcup \cdots$. 
By an argument identical to the proof of Theorem 7.1 in \cite{Thorne-Romano-E8}, we see that for any $\varepsilon >0$, there exists $k\geq 1$ such that 
\begin{displaymath}
\limsup_{a\rightarrow +\infty}  \frac{ \sum_{\substack{b \in \sqcup_{i\geq k} \sh{E}_{W_i,\min} , \height(b) < a  }} \left(\#\Sel_2\Jac_b -1\right)      }{ \# \{b \in \sh{E}_{\min} \mid \height(b) < a  \}  }<\varepsilon.
\end{displaymath}
This implies that 
\begin{align*}
\limsup_{a\rightarrow +\infty}  \frac{ \sum_{\substack{b\in \sh{E}_{\min} , \height(b) < a  }} \left(\#\Sel_2\Jac_b -1 \right)     }{ \# \{b \in \sh{E}_{\min} \mid \height(b) < a  \}  } &\leq 2 \limsup_{a\rightarrow +\infty}\frac{\# \{b \in \sqcup_{i<k} \sh{E}_{W_i,\min} \mid \height(b) < a \}  }{ \# \{b \in \sh{E}_{\min} \mid \height(b) < a  \}  } +\varepsilon \\
&\leq 2+\varepsilon.  
\end{align*}
Since the above inequality is true for any $\varepsilon >0$, we conclude the proof of Theorem \ref{theorem: main theorem}.

\section{Applications to rational points} \label{section: applications to rational points}

The aim of the last section of this paper is to prove the following concrete consequence of Theorem \ref{theorem: main theorem}.
Recall that for each $b\in \sh{E}$ we have a smooth projective curve $\bigprojcurve_b/\Q$ with marked rational point $P_{\infty} \in \bigprojcurve_b(\Q)$.

\begin{theorem}\label{theorem: poonen stoll analogue}
	A positive proportion of curves $\bigprojcurve_b$ for $b$ in $\sh{E}$ have only one rational point. More precisely, the quantity
	\begin{equation*}
	\liminf_{a\rightarrow \infty} \frac{ \# \{b \in \sh{E} \mid \height(b)<a ,\, \bigprojcurve_b(\Q)=\{P_{\infty}\} \}   }{\#  \{b \in \sh{E} \mid \height(b) < a \}}
	\end{equation*}
is strictly positive.
	
\end{theorem}

The proof will be given at the end of this section. We will achieve this by building on the work of Poonen and Stoll \cite{PoonenStoll-Mosthyperellipticnorational} where they prove the corresponding result for odd hyperelliptic curves. We advise the reader to consult the introduction of that paper where the strategy of the proof is carefully explained.
We start by introducing some notation from \cite{PoonenStoll-Mosthyperellipticnorational}.
\begin{itemize}
    \item For a field $k$ and integer $g \geq 1$ we let $\P$ be the usual map $k^g \setminus \{0 \} \rightarrow \P^{g-1}(k)$.
We write $\rho$ for the reduction map $\P^{g-1}(\Q_p)  = \P^{g-1}(\Z_p) \rightarrow \P^{g-1}(\F_p)$ or for the composition $\Q_p^g \setminus \{0\} \xrightarrow{\P} \P^{g-1}(\Q_p) \xrightarrow{\rho} \P^{g-1}(\F_p) $.
If $T$ is a subset of a set $S$ and $f$ is a function defined only on $T$, then $f(S)$ means $f(T)$.
    \item If $A$ is an abelian variety over $\Q_p$ of dimension $g$ we write $\log$ for the logarithm homomorphism
$A(\Q_p) \rightarrow \HH^0(A,\Omega^1_{A/\Q_p})^{\vee} \simeq \Q_p^g,$
see \cite[\S4]{PoonenStoll-Mosthyperellipticnorational}.
The map $\log$ is a local isomorphism with kernel $A(\Q_p)_{tors}$, the torsion points of $A(\Q_p)$.
The image of $\log$ is a lattice in $\Q_p^g$, so after choosing an appropriate basis of $1$-forms $\log$ is a surjective homomorphism $A(\Q_p) \rightarrow \Z_p^g$.
    \item We define $\rho\log$ as the composition of $\log\colon A(\Q_p) \rightarrow \Z_p^g$ with the partially defined map $\rho\colon \Z_p^g \dashrightarrow \P^{g-1}(\F_p)$. The map $\rho \log$ is defined on $A(\Q_p) \setminus A(\Q_p)_{tors}$.
    \item If $A$ is an abelian variety over $\Q$ we have the $2$-Selmer group $\Sel_2A$ associated to $A$, which comes with a homomorphism $\Sel_2 A \rightarrow A(\Q_2)/2A(\Q_2)$.
Write $\sigma$ for the composite of the latter homomorphism with the mod $2$ reduction of the logarithm map $\log \otimes \F_2 \colon A(\Q_2)/2A(\Q_2)  \rightarrow \F_2^g $: it defines a homomorphism $\sigma \colon \Sel_2 A \rightarrow \F_2^g$.
\end{itemize}


Recall that we have defined the abelian scheme $\Jac \rightarrow \bigB^{\rs}$ as the Jacobian of the family of smooth projective curves $\bigprojcurve^{\rs}\rightarrow \bigB^{\rs}$ in \S\ref{subsection: a family of curves}.

\begin{proposition}\label{proposition: generic manin mumford}
	Let $k$ be a field of characteristic zero with separable closure $k^s$ and let $\Spec k \rightarrow \bigB^{\rs}$ be a map to the generic point of $\bigB^{\rs}$. 
	Let $X/k$ be the curve corresponding to this map, with marked point $P_{\infty} \in X(k)$. Let $J_X$ be the Jacobian variety of $X$. 
	Use the point $P_{\infty}$ to embed $X$ in $J_X$. Let $J_X(k^s)_{tors}$ denote the torsion points in $J_X(k^s)$. 
	Then we have $X(k^s)\cap J_X(k^s)_{tors} = \{0\}$.
\end{proposition}
\begin{proof}
	We may assume that $k = \CC(p_2,p_5,p_6,p_8,p_9,p_{12})$ and $X$ is given by projective closure of the equation $y^3=x^4+(p_2x^2+p_5x+p_8)y+p_6x^2+p_9x+p_{12}$.
	Since this equation is the versal deformation of the singularity $y^3=x^4$, 
	\cite[Theorem 1(2)]{Wajnryb-monodromygroupplanecurvesingularity} shows that the monodromy group contains $\ker\left(\Sp_6(\Z) \rightarrow \Sp_6(\Z/2\Z) \right)$.
	 Suppose $P \in X(k^s)$ is a torsion point of exact order $n>1$.
	By an argument identical to \cite[Theorem 7.1]{PoonenStoll-Mosthyperellipticnorational} using the monodromy action and the fact that $X$ is not hyperelliptic, we may assume that $n=2$ or $4$.  
	If $n=4$ then $3P-3P_{\infty}$ is linearly equivalent to $Q-P_{\infty}$ for some $Q\in X(k^s)$ different from $P_{\infty}$, again using the monodromy action. 
	So $Q+2P_{\infty} \sim 3P$ and the line bundle $\mathcal{O}(Q+2P_{\infty})$ has at least $2$ independent global sections. Since the divisor $4P_{\infty}$ is canonical Riemann-Roch implies that $2P_{\infty}-Q$ is linearly equivalent to an effective divisor. This shows that $Q = P_{\infty}$, contradicting our previous assumptions.
	If $n=2$ then $2P-2P_{\infty}$ is a principal divisor, again a contradiction.
	 We have obtained a contradiction in all cases, proving the proposition. 
\end{proof}

For every prime $p$ we obtain a family of $p$-adic Lie groups $J(\Q_p)\rightarrow \bigB^{\rs}(\Q_p)$.
As before we define $\sh{E}_p = \intbigB(\Z_p)\cap \bigB^{\rs}(\Q_p)$.
We define a measure on $\sh{E}_p$ by restricting the measure on $\intbigB(\Z_p) = \Z_p^6$ defined in \S\ref{subsection: heights and measures}.
Following \cite[\S8.2]{PoonenStoll-Mosthyperellipticnorational}, we say $U \subset \sh{E}_p$ is a \define{congruence class} if $U$ is the preimage of a subset of $\intbigB(\Z_p/p^e\Z_p)$ under the reduction map $\intbigB(\Z_p) \rightarrow \intbigB(\Z_p/p^e\Z_p)$ for some $e\geq 1$.
We say a congruence class $U$ is \define{trivializing} if $\Jac(\Q_p) \rightarrow \bigB^{\rs}(\Q_p)$ can be trivialized above $U$, in the sense of \cite[Definition 8.1]{PoonenStoll-Mosthyperellipticnorational}.

The following equidistribution result is a crucial ingredient in the proof of Theorem \ref{theorem: poonen stoll analogue} and readily follows from the proof of Theorem \ref{theorem: main theorem}. (See \cite[Theorem 12.4]{Bhargava-Gross-hyperellcurves} for more details.)

\begin{theorem}\label{theorem: equidistribution selmer}
	Let $U \subset\sh{E}_2$ be a trivializing congruence class. For any $w\in \F_2^3$, the average size of $\#\{s\in \Sel_2\Jac_b\setminus \{0\} \mid \sigma(s)=w \}$ as $b$ varies in $\sh{E}\cap U$, is bounded above by $1/4$.
\end{theorem}

As for the average size of the $2$-Selmer group we only obtain an upper bound, but this will be enough for our purposes. 

Let $Z\subset \sh{E}_p$ be the subset of $b\in \sh{E}_p$ such that $\bigprojcurve_b(\Q_p)\cap \Jac_b(\Q_p)_{tors}\neq \{0\}$, where $\bigprojcurve_b$ is embedded in $\Jac_b$ via the Abel-Jacobi map with basepoint $P_{\infty}$.

\begin{lemma}\label{lemma: density nontrivial torsion is zero}
	The set $Z$ is closed in $\sh{E}_p$ and of measure zero. Moreover the set of all $b\in \sh{E}$ such that $b \in Z$ has density zero. 
\end{lemma}
\begin{proof}
	The first part follows from the previous proposition in the same way as \cite[Proposition 8.5]{PoonenStoll-Mosthyperellipticnorational} follows from \cite[Theorem 7.1]{PoonenStoll-Mosthyperellipticnorational}.
	The second part follows from the previous lemma in a similar way as \cite[Corollary 8.6]{PoonenStoll-Mosthyperellipticnorational} follows from \cite[Proposition 8.5]{PoonenStoll-Mosthyperellipticnorational}.
\end{proof}

\begin{lemma}\label{lemma: rho locally constant in trivializing congruence class}
	Let $U\subset \sh{E}_p$ be a trivializing congruence class. 
	Let $Z$ be as in Lemma \ref{lemma: density nontrivial torsion is zero}.
	Then $\rho \log \bigprojcurve_b(\Q_p)$ in $\P^2(\F_p)$ is locally constant as $b$ varies in $U\setminus Z$.
\end{lemma}
\begin{proof}
    The proof is very similar to that of \cite[Proposition 8.7]{PoonenStoll-Mosthyperellipticnorational}; we sketch the details. 
	Let $U' = U\setminus Z$.
	Choose an isomorphism of $p$-adic analytic manifolds $\Jac(\Q_p)_{U'}\simeq \Z_p^3\times F\times U'$ over $U'$, where $F$ is a finite group. 
	We have a chain of analytic maps of $p$-adic manifolds
	\begin{align*}
	\bigprojcurve(\Q_p)_{U'} \rightarrow \Jac(\Q_p)_{U'} \xrightarrow{\log} \Z_p^3\times U' \twoheadrightarrow \Z_p^3 \dashrightarrow \P^2(\Q_p) \xrightarrow{\rho} \P^{2}(\F_p), 
	\end{align*}
	except that the dashed arrow is only defined on $\Z_p^3 \setminus \{0\}$.
	The inverse image of $0 \in \Z_p^3$ in $\bigprojcurve(\Q_p)_{U'}$ is $P_{\infty,U'}$, the section at infinity. Since the latter is a smooth divisor on $C(\Q_p)_{U'}$, the composition $\bigprojcurve(\Q_p)_{U'} \setminus P_{\infty,U'} \rightarrow \P^{2}(\F_p)$ extends to a continuous map $e\colon \bigprojcurve(\Q_p)_{U'} \rightarrow \P^2(\F_p)$.

	By continuity the fibres of $e$ are open and closed. 
	So are their images in $U'$, since $\bigprojcurve\rightarrow \bigB$ is flat and proper.
	Thus for each $c\in \F_2^3$, the set of $b\in U'$ such that $c\in e(\bigprojcurve_b(\Q_p))$ is open and closed.
	By considering intersections and complements of such sets, we see that $e(\bigprojcurve_b(\Q_p))$ is locally constant as $b$ varies in $U'$. The lemma follows from the equality $ \rho\log(\bigprojcurve_b(\Q_p)) =e(\bigprojcurve_b(\Q_p))$.
	
\end{proof}

\begin{proposition}\label{proposition: existence good curve}
	There exists an element $b\in \sh{E}$ such that $b \in \sh{E}_2 \setminus Z$ and $\#\rho \log \bigprojcurve_b(\Q_2) = 2$.
\end{proposition}
\begin{proof}
	We choose $b\in \sh{E}$ such that $\bigprojcurve_b$ is isomorphic to the projective closure of the smooth curve $y^3+y = x^4+x+1$. Let $\mathcal{X}/\Z_2$ be the projective curve over $\Z_2$ given by the latter equation. 
	Then $\mathcal{X}$ has good reduction at $2$ and $\#\mathcal{X}(\F_2)=1$. Let $\NeronJac/\Z_2$ be the Jacobian of $\mathcal{X}$. We have $\NeronJac[2](\overbar{\F}_2) = 0$ because $\mathcal{X}_{\F_2}$ is up to substitution given by the supersingular normal form of \cite[Proposition 2.1]{Nart-nonhyperellipticcharacteristictwo}, so $\NeronJac[2](\F_2)$ is trivial too.
	To determine $\NeronJac[2](\Q_2)$, we explicitly compute the bitangents of $\mathcal{X}_{\Q_2}$ different from the line at infinity. They are of the form $y = ax+b$ for some $a,b\in \overbar{\Q}_2$. We solve for the equation $x^4+x+1-(ax+b)^3-(ax+b) = (x^2+cx+d)^2$ where $c,d\in \overbar{\Q}_2$. 
	Then $c$ and $d$ are polynomials in $a$ and $b$ and we are left with two polynomial conditions in $a$ and $b$. The resultant of these two polynomials with respect to the variable $b$ is up to a constant equal to 
	\begin{dmath*}
		4096+12288 a-126976 a^3+110592 a^6-165888 a^7-40704 a^9+70656 a^{10}-34560 a^{11}+17280 a^{15}+1344 a^{18}+480 a^{19}+a^{27}.
	\end{dmath*}
	A calculation in \texttt{Magma} \cite{MAGMA} shows that this polynomial is irreducible in $\Q_2$, so the absolute Galois group of $\Q_2$ acts transitively on these $27$ bitangents. 
	Thus Lemma \ref{lemma: bitangents and 2-torsion} implies that $\NeronJac[2](\Q_2 ) = 0$.
	By \cite[Lemma 10.1]{PoonenStoll-Mosthyperellipticnorational} we see that the image of $\NeronJac(\Q_2)$ under the logarithm map with respect to a $\Z_2$-basis of $\HH^0(\NeronJac,\Omega^1_{\NeronJac/\Z_2} )$ is $\left(2\Z_2\right)^3$. 
	We can compute the logarithm map explicitly on $\mathcal{X}(\Q_2)$ as follows. Since every element of $\mathcal{X}(\Q_2)$ reduces to the point at infinity $P_{\infty}$, the set $\mathcal{X}(\Q_2)$ consists of a single residue disk around $P_{\infty}$. 
	Homogenizing the above equation and setting $y$ equal to $1$ gives the equation 
	\begin{equation*}
	z+z^3 = x^4+xz^3+z^4. 	
	\end{equation*}
	The point $P_{\infty}$ now corresponds to the point $(0,0)$ and $x$ is a uniformizer at $(0,0)$. 
	The map $Q \mapsto x(Q)$ defines a homeomorphism $\mathcal{X}(\Q_2) \simeq 2\Z_2$.
	Taking the derivative of the above equation leads us to define
	\begin{equation*}
		\omega_1 = \frac{dx}{3z^2+1-3z^2x-4z^3}.
	\end{equation*}
	Moreover we set $\omega_2 = x\omega_1$ and $\omega_3 = z\omega_1$.
	Then $\{\omega_1,\omega_2,\omega_3\}$ forms a basis for the $\Z_2$-module $\HH^0(\mathcal{X},\Omega^1_{\mathcal{X}/\Z_2})$.
	The logarithm map on $\mathcal{X}(\Q_2)$ is given by explicitly integrating these $1$-forms. 
	A computation reveals that 
	\begin{align*}
		z = x^4-x^{12}+O(x^{13}), \\
		\omega_1 = \left(1-3x^8+3x^9+ O(x^{12}) \right) dx. 
	\end{align*}
	Here each $\omega_i$ has a power series expansion with coefficients in $\Z_2$.
	This implies that the logarithm map, using the uniformizer $x$ and the differentials $\omega_i$, is explicitly given by 
	\begin{align*}
		x \mapsto \left(x-\frac{x^9}{3}+\frac{3x^{10}}{10}+O(x^{13}), \frac{x^2}{2}-\frac{3x^{10}}{10}+O(x^{11}), \frac{x^5}{5}-O(x^{13})  \right)
	\end{align*}
	 This description shows that $\rho \log \mathcal{X}(\Q_2)=\{(1:1:0),(1:0:0)\}$. Moreover the last power series has no roots in $2\Z_2$ apart from $0$ by Newton polygon considerations. This implies that $\mathcal{X}(\Q_2) \cap \NeronJac(\Q_2)_{tors} = \{0\}$ hence $b$ does not lie in $Z$.

\end{proof}

We are now ready to prove Theorem \ref{theorem: poonen stoll analogue}. Let $U\subset \sh{E}_2$ be a trivializing congruence class containing an element $b_0\in \sh{E}$ satisfying the conclusion of Proposition \ref{proposition: existence good curve}. Shrink $U$ using Lemma \ref{lemma: rho locally constant in trivializing congruence class} so that the image of $\rho\log\bigprojcurve_b(\Q_2) \subset \P^2(\F_2)$ is constant for all $b\in U' = U\setminus Z$, say equal to $I$. 
Then \cite[Corollary 6.3]{PoonenStoll-Mosthyperellipticnorational} shows that $\bigprojcurve_b(\Q) = \{P_{\infty}\}$ for all $b\in \sh{E} \cap U'$ with the property that the map $\sigma\colon \Sel \Jac_b \rightarrow \F_2^3$ is injective and $I \cap \P\sigma(\Sel_2 \Jac_b) = \emptyset$.
By Theorem \ref{theorem: equidistribution selmer} and Lemma \ref{lemma: density nontrivial torsion is zero}, the proportion of $b\in \sh{E} \cap U$ satisfying these conditions is at least $1-1/4-\#I/4=1/4>0$.
This proves the theorem.



\begin{footnotesize}
\textsc{Jef Laga  }\;  \texttt{jcsl5@cam.ac.uk}   \newline
\textsc{Department of Pure Mathematics and Mathematical Statistics, Wilberforce Road, Cambridge, CB3 0WB, UK}
\end{footnotesize}

\end{document}